\documentclass[reqno]{amsart}
\usepackage{amsmath}
\usepackage{amsthm}
\usepackage{amssymb}
\usepackage{stmaryrd}

\usepackage[utf8]{inputenc}
\usepackage[T1]{fontenc}
\usepackage{lmodern}
\usepackage{eucal}
\pdfoutput=1 
\usepackage{microtype}
\usepackage{url}
\usepackage{enumitem}
\usepackage{tikz-cd}
\usepackage{bbm}
\usepackage{longtable,booktabs}

\usepackage[letterpaper, left = 1.5in, right = 1.5in, top = 1.5in, bottom = 1.5in]{geometry}

\swapnumbers

\usepackage{hyperref}

\usepackage{xcolor}
\hypersetup{
    colorlinks,
    linkcolor={red!50!black},
    citecolor={blue!50!black},
    urlcolor={blue!80!black}
}

\usepackage[capitalise,nameinlink]{cleveref}
\crefname{subsection}{Subsection}{Subsections}
\crefname{subsubsection}{Subsubsection}{Subsubsections}

\theoremstyle{definition}
\newtheorem{theorem}{Theorem}[subsection]
\newtheorem{defn}[theorem]{Definition}

\newtheorem{ex}[theorem]{Example}
\newtheorem{cor}[theorem]{Corollary}
\newtheorem{lemma}[theorem]{Lemma}
\newtheorem{prop}[theorem]{Proposition}
\newtheorem{rmk}[theorem]{Remark}
\newtheorem{warning}[theorem]{Warning}

\newtheorem{question}[theorem]{Question}

\newtheorem{construction}[theorem]{Construction}
\newtheorem*{rmk*}{Remark}
\newtheorem*{ex*}{Example}
\newtheorem*{theorem*}{Theorem}
\newtheorem*{defn*}{Definition}

\newcommand{\bbZ}{\mathbb{Z}}

\newcommand{\bbH}{\mathbb{H}}

\newcommand{\bbN}{\mathbb{N}}

\newcommand{\bbQ}{\mathbb{Q}}
\newcommand{\bbR}{\mathbb{R}}

\newcommand{\bbC}{\mathbb{C}}

\newcommand{\Sp}{\mathcal{S}\mathrm{p}}

\newcommand{\Fun}{\mathrm{Fun}}

\newcommand{\Ind}{\operatorname{Ind}}
\newcommand{\Res}{\operatorname{Res}}

\newcommand{\Map}{\operatorname{Map}}
\newcommand{\Hom}{\operatorname{Hom}}
\newcommand{\Aut}{\operatorname{Aut}}

\newcommand{\Fib}{\operatorname{Fib}}

\newcommand{\Cof}{\operatorname{Cof}}
\newcommand{\coker}{\operatorname{Coker}}
\newcommand{\colim}{\operatorname*{colim}}

\newcommand{\im}{\operatorname{Im}}
\renewcommand{\ker}{\operatorname{Ker}}

\newcommand{\Sub}{\operatorname{Sub}}

\newcommand{\res}{\operatorname{res}}
\newcommand{\tr}{\operatorname{tr}}

\newcommand{\Cl}{\operatorname{Cl}}

\newcommand{\id}{\operatorname{id}}

\newcommand{\trp}{\mathfrak{t}}
\newcommand{\resp}{\mathfrak{r}}

\newcommand{\h}{{\mathrm{h}}}
\newcommand{\tate}{{\mathrm{t}}}

\newcommand{\cyc}{\mathrm{cyc}}
\newcommand{\red}{\mathrm{red}}

\newcommand{\bs}{{-}}
 
\newcommand{\ul}{\underline}
\newcommand{\ol}{\overline}

\newcommand{\rat}{W}

\newcommand{\ceil}[1]{\left\lceil#1\right\rceil}
\newcommand{\floor}[1]{\left\lfloor#1\right\rfloor}

\newcommand\xqed[1]{%
  \leavevmode\unskip\penalty9999 \hbox{}\nobreak\hfill
  \quad\hbox{#1}}
\newcommand\tqed{\xqed{$\triangleleft$}}

\makeatletter
\DeclareRobustCommand{\tvdots}{%
  \vbox{\baselineskip4\p@\lineskiplimit\z@\kern0\p@\hbox{.}\hbox{.}\hbox{.}}}
\makeatother

\makeatletter
\@namedef{subjclassname@2020}{\textup{2020} Mathematics Subject Classification}
\makeatother

\makeatletter
\newcommand{\raisemath}[1]{\mathpalette{\raisem@th{#1}}}
\newcommand{\raisem@th}[3]{\raisebox{#1}{$#2#3$}}
\makeatother

\begin{document}

\title{$C_{p^n}$-equivariant Mahowald invariants}
\author{William Balderrama}
\address{Mathematical Institute \\ University of Bonn \\  Endenicher Allee 60, D-53115 Bonn, Germany}
\email{williamb@math.uni-bonn.de}

\author{Yueshi Hou}
\address{Department of Mathematics, University of California San Diego, La Jolla, CA 92093, USA}
\email{yuh091@ucsd.edu}

\author{Shangjie Zhang}
\address{Department of Mathematics, University of California San Diego, La Jolla, CA 92093, USA}
\email{shz046@ucsd.edu}

\begin{abstract}
The classical Mahowald invariant is an operation that
systematically produces new elements in the stable homotopy groups of spheres from known ones. We introduce the $C_{p^n}$-Mahowald invariant: a relation $\pi_\star S_{\raisemath{3pt}{C_{p^{n-1}}}}\rightharpoonup \pi_\ast S$ between equivariant and classical stable stems which recovers the classical Mahowald invariant when $n=1$. We compute the $C_{p^n}$-Mahowald invariants of all elements in the Burnside ring $\smash{A(C_{p^{n-1}}) = \pi_0S_{\raisemath{2pt}{C_{p^{n-1}}}}}$, extending Mahowald and Ravenel's computation of $M_{C_p}(p^k)$. As a consequence, we determine the image of the $C_p$-geometric fixed point map $\Phi^{C_p} \colon \pi_V S_{\raisemath{2pt}{C_{p^{n}}}}\rightarrow \pi_0 \smash{S_{\raisemath{2pt}{C_{p^{n}}/C_p}}}\cong A(C_{p^{n-1}})$ when $V$ is fixed point free, extending classical theorems of Bredon, Landweber, and Iriye for $n=1$.
\end{abstract}

\maketitle


\section{Introduction}

\subsection{\texorpdfstring{$C_{p^n}$}{C\_pn}-equivariant stable stems and geometric fixed points}

Homotopy classes of self-maps $S^n\rightarrow S^n$ of a positive-dimensional sphere are classified by their \textit{degree}, which can be any integer. The equivariant story is more subtle. Given a virtual $G$-representation $\alpha = V-W$, the  \textit{$G$-equivariant stable $\alpha$-stem} is defined as
\[
\pi_\alpha S_G = \colim_{t\rightarrow\infty} [S^{V+t\bbR[G]},S^{W+t\bbR[G]}]_\ast^G.
\]
Here, $[\bs,\bs]_\ast^G$ denotes $G$-equivariant homotopy classes of maps.  The degree furnishes an isomorphism $\pi_0 S = \pi_0 S_e \cong \bbZ$, and so one is led to ask the following fundamental question.

\begin{question}\label{problem:fpdegree}
Fix a finite group $G$, a subgroup $H\subset G$, and a virtual $G$-representation $\alpha$ satisfying $|\alpha^H| = 0$. What is the image of the $H$-geometric fixed point map
\[
\Phi^H\colon \pi_\alpha S_G\rightarrow \pi_{\alpha^H} S \cong \bbZ,\qquad \Phi^H(f\colon S^{V}\rightarrow S^W) = \deg(f^H\colon S^{V^H}\rightarrow S^{W^H})?\tag*{$\triangleleft$}
\]
\end{question}

Classical work on \cref{problem:fpdegree} was used to describe $\pi_\alpha S_G$ when $|\alpha^K|\leq 0$ for all $K\subset G$ \cite{dieckpetrie1978geometric, hauschild1977zerspaltung, tornehave1982equivariant}, giving information about the Picard group of $G$-spectra \cite{fausklewismay2001picard,angeltveit2021picard}. As one allows $H$ to range over the subgroups of $G$, geometric fixed points detect equivalences and nilpotence \cite{iriye1982equivariant, barthelgreenleeshausmann2020balmer}, and accordingly play a central role in the study of chromatic equivariant phenomena \cite{balmersanders2016spectrum, barthelhausmannnaumannnikolausnoelstapleton2019balmer, kuhnlloyd2022chromatic,bhattacharyaguillouli2022rmotivic, behrenscarlisle2024periodic}. They have been used to describe the rationalization $\bbQ \otimes \pi_\star S_G$ \cite{greenleesquigley2023ranks}, and from our perspective \cref{problem:fpdegree} asks for explicit integral information about the ``height zero portion'' of $\pi_\star S_G$. A major theme in equivariant homotopy theory is the prominence of transchromatic phenomena, already made apparent in work on Tate spectra \cite{davismahowald1984spectrum, greenleessadofsky1996tate, hoveysadofsky1996tate, andomoravasadofsky1998completions, lilormanquigley2022tate} and Mahowald invariants \cite{shick1987root, mahowaldshick1988root, sadofsky1992root, brunergreenlees1995bredon, behrens2007some, quigley2019tmf}, meaning that even this torsion-free information is intricately tied up with higher height torsion phenomena.

The equivariant stable stems themselves were introduced by Bredon \cite{bredon1967equivariant,bredon1968equivariant} for the group $G = C_2$ of order $2$, whose work on \cref{problem:fpdegree} culminated in a conjecture that was proved by Landweber \cite{landweber1969equivariant}: if $\sigma$ is the sign representation of $C_2$, then 
\begin{equation*}\label{eq:landweber}
\im(\Phi^{C_2} \colon \pi_{k\sigma}S_{C_2}\rightarrow \pi_0 S) = g(k)\bbZ,\quad\text{where}\quad g(k) = \begin{cases}2^{4l+1}&k\in\{8l,8l+1\},\\
2^{4l+2}&k=8l+2,\\
2^{4l+3}&k=8l+3,\\
2^{4l+4}&k\in\{8l+4,\ldots,8l+7\}.\end{cases}
\end{equation*}
The analogous theorem at odd primes was later proved by Iriye \cite{iriye1989images}: if $p$ is an odd prime and $V$ is a sum of $k>0$ faithful complex characters of $C_p$, then
\[
\im(\Phi^{C_p}\colon \pi_V S_{C_p}\rightarrow\pi_0 S) = p^{1+\floor{\frac{k}{p-1}}}\bbZ.
\]
Despite significant advances in equivariant homotopy theory since the 1960s and 1980s, little progress has been made on \cref{problem:fpdegree} for larger groups, except when $|\alpha^K| \leq 0$ for all $K\subset G$ as mentioned above. This paper develops effective methods for resolving this and related questions when $\alpha = V$ is a non-virtual representation and $H\subset G$ is a \textit{cyclic} subgroup of a $p$-group. In particular, we prove the following.

\begin{theorem}[{\ref{cor:fpdegree}}]\label{thm:fixeddegree}
Let $V$ be a sum of $k>0$ faithful complex characters of $C_{p^n}$. Then the image of the degree function
\[
\deg\colon \pi_{V} S_{C_{p^n}}\rightarrow \Hom(\{1,\ldots,n\},\bbZ),\qquad (\deg y)(i) = \Phi^{C_{p^{i}}}(y)
\]
has a basis of $n$ functions $f_{1,2k}^{p,n},f_{2,2k}^{p,n},\ldots,f_{n,2k}^{p,n}$, where
\[
f_{s,2k}^{p,n}(i) = \begin{cases}
p^{n-s+p^{s-i}\left(1+\floor{\frac{k}{p^{s-1}(p-1)}}\right)}&1\leq i \leq s,\\ 0&s < i \leq n,
\end{cases}
\]
except when $p=2$ and $i=s=1$ where
\[
f_{1,2k}^{2,n}(1) = \begin{cases}
2^{n+k}&k\equiv 0,1,3\pmod{4},\\
2^{n+k+1}&k\equiv 2 \pmod{4}.
\end{cases}
\tag*{\vspace{\baselineskip}$\triangleleft$}
\]
\end{theorem}

The reason for the even subscripts will appear in \cref{ssec:rootintro} below.

\begin{cor}[{\ref{cor:fpimage}}]\label{thm:fpimage}
Let $V$ be a sum of $k>0$ faithful complex characters of $C_{p^n}$. Then the image of the $C_p$-geometric fixed point map
\[
\Phi^{C_p}\colon \pi_V S_{C_{p^n}}\rightarrow\pi_{V^{C_p}}S_{C_{p^n}/C_p} \cong \pi_0 S_{C_{p^{n-1}}}
\]
has a basis of $n$ elements $X_{1,k}^{p,n},\ldots,X_{n,k}^{p,n}$, where if we write $\tr$ and $N$ for the transfer and norm on the $0$th equivariant stable stem, then
\[
X_{s,k}^{p,n} = \tr_{C_{p^{s-1}}}^{C_{p^{n-1}}}\left(N_e^{C_{p^{s-1}}}\left(p^{1+\floor{\frac{k}{p^{s-1}(p-1)}}}\right)\right),
\]
except when $p=2$ and $s=1$ where $X^{2,n}_{1,k} = \tr_e^{C_{2^{n-1}}}(f_{1,2k}^{2,1}(1))$.
\tqed
\end{cor}

\begin{ex}\label{ex:cpn}
Except for the group $C_2$, \cref{thm:fixeddegree} implies that if $V$ is a sum of $k$ faithful complex characters of $C_{p^n}$, then
\[
\im(\Phi^{C_{p^n}}\colon \pi_V S_{C_{p^n}}\rightarrow\pi_0 S \cong \bbZ) = p^{1+\floor{\frac{k}{p^{n-1}(p-1)}}}\bbZ.
\]
For example, there exists a $C_{p^n}$-equivariant map $f\colon S^{V\oplus U}\rightarrow S^U$ for some $U$ with the property that the induced map $f^{C_{p^n}}\colon S^{U^{C_{p^n}}}\rightarrow S^{U^{C_{p^n}}}$ on fixed points is of degree $p$ if and only if $k < p^{n-1}(p-1)$. More generally, if $1 \leq i \leq n$ then
\[
\im(\Phi^{C_{p^i}}\colon \pi_V S_{C_{p^n}}\rightarrow \pi_0 S \cong\bbZ) = \min(f_{i,2k}^{p,n}(i),\ldots,f_{n,2k}^{p,n}(i))\bbZ,
\]
but this does not admit any particular simplification for $i<n$.
\tqed
\end{ex}


We now describe another perspective on \cref{thm:fixeddegree}. Classically, Serre's finiteness theorem implies that $\pi_n S$ is finite for $n \neq 0$, and Nishida's nilpotence theorem asserts that every element of $\pi_n S$ is nilpotent for $n \neq 0$. Equivariantly, the groups $\pi_\star S_G$ are finitely generated in each degree, but contain nonzero torsion-free summands and non-nilpotent elements in infinitely many degrees as soon as $G$ is nontrivial. Iriye \cite{iriye1982nilpotency} proved that an element of $\pi_\star S_G$ is nilpotent if and only if it is torsion, so one is led to study the reduced ring
\[
\pi_\star^{\red}S_G = \pi_\star S_G/\sqrt{0} = \pi_\star S_G / \text{torsion}.
\]
This was studied for $G = C_2$ and $C_4$ by Crabb in \cite{crabb1989periodicity}, and a presentation of $\pi_\star^{\red}S_{C_2}$ was given by Belmont, Xu, and the third-named author in \cite{belmontxuzhang2024reduced}. While Greenlees and Quigley \cite{greenleesquigley2023ranks} have recently studied the rank of the equivariant stable stems, we can go further using \cref{thm:fixeddegree} and \cref{thm:fpimage}: the degree map of \cref{thm:fixeddegree} is a rational isomorphism, and therefore its image gives the following integral description of $\pi_\star^{\red}S_{C_{p^n}}$ in certain $RO(C_{p^n})$-degrees. 

\begin{theorem}[{\ref{thm:presentation}}]
Let $L$ be a faithful complex character of $C_{p^n}$. Then the reduced ring $\pi_{\ast L}^\red S_{C_{p^n}}$ is generated by classes
\[
y_{n,i,c} \in \pi_{(p^{i-1}c(p-1)-1)L}^{\red}S_{C_{p^n}}
\]
for $c\geq 0$ and $1\leq i \leq n$, except when $p=2$ and $i=1$ where we require $c\not\equiv 2\pmod{4}$. The class $y_{n,n,0} = a_L$ is the Euler class of $L$, and in general these classes satisfy $y_{n,i,c} = \tr_i^n(y_{i,i,c})$. A complete set of relations is given by
\begin{align*}
y_{n,i,c} \cdot y_{n,j,d} &= p^{n-j} a_L y_{n,i,(c+dp^{j-i})}\qquad\qquad\qquad(i\leq j),\\
a_L^{p^{i-1}(p-1)}\cdot y_{n,i,c+1}&=py_{n,i,c}+\sum_{0< j< i}\frac{p^{p^{i-j}}-p^{p^{i-j-1}}}{p^{i-j}}y_{n,j, p^{i-j}c},
\end{align*}
except when $p=2$ and $i=1$ and $c \equiv 2 \pmod{4}$ where the second relation must be replaced by $a_L^2 \cdot y_{n,1,4k+3} = 4y_{n,1,4k+1}$.
\tqed
\end{theorem}

\subsection{\texorpdfstring{$C_{p^n}$}{C\_pn}-Mahowald invariants}\label{ssec:rootintro}

The preceding discussion considers what degrees arise from equivariant stable maps $S^V \to S_{C_{p^n}}$. As the even subscripts suggest, \cref{thm:fixeddegree} is a special case of a stronger theorem. As the dimension of $V$ grows, the space of realizable degrees shrinks. Rather than merely asking what degrees can or cannot be realized, one can instead attempt to pin down the precise obstructions to realizability, with the above evenness corresponding to the fact that these obstructions may generally live halfway between complex dimensions, on odd-dimensional skeleta of $S^V$. Under the assumption that $V$ is a sum of faithful characters, these obstructions lie in the nonequivariant stable stems, and are precisely encoded by the \emph{$C_{p^n}$-equivariant Mahowald invariants} of the title.

This point of view is already present in the classical case. In \cite{brunergreenlees1995bredon}, Bruner and Greenlees showed that Landweber's theorem on the image of $\Phi^{C_2}$ is equivalent to a computation of the degrees of the Mahowald invariants of powers of $2$, which had been computed by Mahowald and Ravenel \cite{mahowaldravenel1993root}. We extend this in full generality to $C_{p^n}$-equivariant homotopy theory.

The Segal conjecture, resolved by Carlsson \cite{carlsson1983equivariant}, implies that
\[
S^{\tate C_{p^n}}\simeq (S^{C_{p^{n-1}}})_p^\wedge,
\]
where $S^{\tate C_{p^n}}$ is the Tate construction for the trivial action of $C_{p^n}$ on the sphere spectrum and $S^{C_{p^{n-1}}}$ is the fixed point spectrum of the $C_{p^{n-1}}$-equivariant sphere spectrum. Hence a choice of periodic resolution for $C_{p^n}$, and we shall be explicit about our choices in \cref{ssec:spoke}, gives rise to a \textit{Tate spectral sequence}
\[
E_1^{s,\ast} = \bigoplus_{i\in\bbZ}\pi_s S^{-i} \Rightarrow \pi_s S_{C_{p^{n-1}}}.
\]

\begin{defn}[{\ref{def:cpnmifull}}]\label{def:cpnmahowald}
The \textit{$C_{p^n}$-Mahowald invariant}
\[
M_{C_{p^n}}\colon \pi_\star S_{C_{p^{n-1}}} \rightharpoonup \pi_\ast S
\]
is the relation defined by
\[
y \in M_{C_{p^n}}(x)\qquad\iff\qquad \text{$y$ detects $x$ in the Tate spectral sequence}.\tag*{$\triangleleft$}
\]
\end{defn}

When $n = 1$, this recovers the classical Mahowald invariant $M = M_{C_p}\colon \pi_\ast S \rightharpoonup \pi_\ast S$, classically called the root invariant. This construction has its origins in Mahowald's early work on metastable homotopy groups \cite{mahowald1967metastable}, and has since been studied extensively by many people; see \cite{mahowaldravenel1993root,behrens2007some} for surveys. Related equivariant generalizations of the Mahowald invariant have been studied by Greenlees--May \cite{greenleesmay1995generalized}, and more recently by Hopkins--Lin--Shi--Xu \cite{hopkinslinshixu2022intersection} and Botvinnik--Quigley \cite{botvinnikquigley2023symmetries}.

The classical $C_p$-Mahowald invariants of powers of $p$ were computed by Mahowald and Ravenel \cite{mahowaldravenel1993root}, who prove that $M_{C_{p}}(p^k)$ always contains a certain element of order $p$ related to the image of $J$. We carry out the full $C_{p^n}$-equivariant analogue of this, computing the $C_{p^n}$-Mahowald invariants of all elements of the Burnside ring $A(C_{p^{n-1}}) = \pi_0 S_{C_{p^{n-1}}}$. In the course of doing so, we provide an analogue of Bruner and Greenlees' theorem in \cref{prop:bg}, relating the $C_{p^n}$-Mahowald invariant to certain ``spoke-graded'' $C_{p^n}$-equivariant stable stems. Rather than give the full statement, let us just unravel it in the particular case at hand.

In general, if $X\in A(C_{p^{n-1}})$ is a virtual $C_{p^{n-1}}$-set, then $|M_{C_{p^n}}(X)| > 0$ if and only if $|X|\equiv 0 \pmod{p^n}$, which in turn holds if and only if $X = \Phi(\tilde{X})$ for some virtual $C_{p^n}$-set $\tilde{X}$ in the augmentation ideal of $A(C_{p^n})$. Here $\Phi(\tilde{X}) = \tilde{X}^{C_p}$, and this $\tilde{X}$ is uniquely determined. Given this, we then have $|M_{C_{p^n}}(X)| \geq 2k$ if and only if $\tilde{X}$ is divisible by the Euler class $a_V$ of some, and thus any, sum $V$ of $k$ faithful complex characters:
\begin{equation}\label{eq:introextension}
\begin{tikzcd}
S^0\ar[r,"\tilde{X}"]\ar[d,"a_V"']&S_{C_{p^n}}\\
S^V\ar[ur,dashed]
\end{tikzcd}.
\end{equation}

\begin{rmk}
This relates \cref{def:cpnmahowald} to the $G$-equivariant Mahowald invariants considered by Hopkins--Lin--Shi--Xu in \cite{hopkinslinshixu2022intersection}: in particular $|M_{C_{p^n}}(X)| \in \{2k,2k+1\}$ if and only if, in their notation, $|\smash{M_{a_L}^{\raisemath{-2pt}{C_{p^{n}}}}(X)}| = kL$, where $L$ is a faithful complex character.
\tqed
\end{rmk}

If $|M_{C_{p^n}}(X)| < 2k$, then no lift exists in \cref{eq:introextension}. In this case, we say $|M_{C_{p^n}}(X)| = l$ if and only if $\tilde{X}$ extends through the $l$-skeleton of some, and thus any, equivariant cell structure on $S^V$, and extends no further. The representation sphere $S^V$ admits a cell structure with one free cell in every positive degree, meaning the set of obstructions to extending through the $(l+1)$-skeleton is a subset of the nonequivariant stable $l$-stem. This set of obstructions, for a particular choice of $V$ and equivariant cell structure on $S^V$, is exactly the $C_{p^n}$-Mahowald invariant $M_{C_{p^n}}(X)\subset \pi_l S$. 

In other words, whereas \cref{thm:fixeddegree} describes the image of the $C_p$-geometric fixed point map $\Phi\colon\pi_V S_{C_{p^n}}\rightarrow A(C_{p^{n-1}})$, the $C_{p^n}$-Mahowald invariant $M_{C_{p^n}}\colon A(C_{p^{n-1}})\rightharpoonup \pi_\ast S$ describes what exactly prevents a given element from being in this image. To describe these Mahowald invariants, we require a couple of preliminary definitions.

\begin{defn}\label{def:marks}
We write
\[
\phi\colon A(C_{p^{n-1}})\rightarrowtail \Hom(\{1,\ldots,n\},\bbZ),\qquad (\phi X)(i) = |X^{C_{p^{i-1}}}|
\]
for the \emph{marks homomorphism}.
\tqed
\end{defn}

\begin{defn}\label{def:thebasis}
Define the functions
\[
f_{s,k}^{p,n}\colon \{1,\ldots,n\}\rightarrow\bbZ
\]
for $1\leq s \leq n$ and $k\geq 1$ by
\[
f_{s,k}^{p,n}(i) = \begin{cases} p^{n-s+p^{s-i}\ceil{\frac{k+1}{2p^{s-1}(p-1)}}}&1\leq i \leq s,\\ 0&s < i \leq n,
\end{cases}
\]
with the following $2$-primary exceptions:
\begingroup
\allowdisplaybreaks
\begin{align*}
f_{1,k}^{2,1}(1) &=\begin{cases}2^{4l+1}&k\in\{8l,8l+1\},\\
2^{4l+2}&k=8l+2,\\
2^{4l+3}&k=8l+3,\\
2^{4l+4}&k\in\{8l+4,\ldots,8l+7\};\end{cases}
\\
f_{1,k}^{2,n}(1) &= \begin{cases}2^{n+4l}&k\in\{8l,8l+1\},\\
2^{n+4l+1}&k\in\{8l+2,8l+3\},\\
2^{n+4l+3}&k\in\{8l+4,\ldots,8l+7\};
\end{cases}&&(n\geq 2)\\
f_{2,k}^{2,2}(1) &= \begin{cases}2^{4l+2}&k\in\{8l,8l+1,8l+2\},\\
2^{4l+3}&k=8l+3,\\
2^{4l+4}&k\in\{8l+4,\ldots,8l+7\};
\end{cases}\\
f_{2,k}^{2,2}(2) &= \begin{cases} 2^{2l+1}&k\in\{8l,8l+1,8l+2\},\\
2^{2l+2}&k\in\{8l+3,\ldots,8l+7\}.
\end{cases}
\end{align*}
\endgroup
These extend the functions defined in \cref{thm:fixeddegree}, as the reader may verify.
\tqed
\end{defn}

We can use these to state the main theorems of this paper.

\begin{defn}\label{def:gammafilt}
Define a filtration of $A(C_{p^{n-1}})$ by
\[
\Gamma_k(C_{p^{n-1}}) = \{X\in A(C_{p^{n-1}}) : |M_{C_{p^n}}(X)|\geq k\}.\tag*{$\triangleleft$}
\]
\end{defn}

\begin{theorem}[{\ref{thm:cpnmahowaldfilt}}]\label{thm:mainroot1}
For all $k > 0$, the functions $f_{1,k}^{p,n},f_{2,k}^{p,n},\ldots,f_{n,k}^{p,n}$ form a basis for $\phi(\Gamma_k(C_{p^{n-1}}))\subset \Hom(\{1,\ldots,n\},\bbZ)$.
\tqed
\end{theorem}

\begin{theorem}[{\ref{thm:micomputation}}]\label{thm:mainroot2}
For $i\geq 1$, there exist generators
\[
j_{2i(p-1)-1}^{(p)} \in \pi_{2i(p-1)-1}S
\]
of the $p$-torsion in the image of $J$ in this degree ($i$ even if $p=2$), independent of $n$, for which if $X\in A(C_{p^{n-1}})$ satisfies $|M_{C_{p^n}}(X)| = k$ and we write $\phi(X) = c_1 f_{1,k}^{p,n}+\cdots + c_n f_{n,k}^{p,n}$, then:
\begin{enumerate}
\item If $p>2$, then $k=2p^{l-1}c(p-1)-1$ for some $l\geq 1$ and $p\nmid c$. In this case $j_{2p^{l-1}c(p-1)-1}^{(p)}$ has order $p^l$, and if $t = \min(n,l)$ then
\[
(p^{t-1}c_1+p^{t-2}c_2+\cdots+c_t) \cdot p^{l-t}j_{2p^{l-1}c(p-1)-1}^{(p)} \in M_{C_{p^n}}(X).
\]
\item If $p=2$, then $k$ is congruent to one of $1,2,3,7\pmod{8}$. In this case,
\begin{enumerate}
\item If $k = 1$, then $\eta \in M_{C_{2^n}}(X)$.
\item If $k = 8l+1$ with $l\geq 1$, then
\[
P^l \eta + \bbZ/(2)\{P^{l-1}\eta\epsilon\} \subset M_{C_{2^n}}(X),
\]
where $P = \langle \bs,2,8\sigma\rangle$ is the usual Adams periodicity operator and $\epsilon\in\pi_8 S$ is the class detected by $c_0$ in the Adams spectral sequence. In other words, $M_{C_{2^n}}(X)$ contains both height $1$ classes with nonzero Hurewicz image in $KO$.
\item If $k = 8l+2$, then $P^l\eta^2 \in M_{C_{2^n}}(X)$. Moreover, this can only happen for $n\leq 2$.
\item If $k = 8l+3$, then $j_{8l+3}^{(2)}$ has order $8$ and
\begin{enumerate}
\item If $n=1$, then $4j_{8l+3}^{(2)} \in M_{C_2}(X)$;
\item If $n=2$, then $-2(c_1-c_2)j_{8l+3}^{(2)} \in M_{C_4}(X)$;
\item If $n\geq 3$, then $-(2c_1-c_2)j_{8l+3}^{(2)} \in M_{C_{2^n}}(X)$.
\end{enumerate}
\item If $k=2^{l}c-1$ with $2\nmid c$ and $l\geq 3$, then $j_{2^{l}c-1}^{(2)}$ has order $2^{l+1}$, and if $t = \min(n,l)$ then
\[
(2^{t-1}c_1+2^{t-2}c_2+\cdots+c_t) 2^{l+1-t} j_{2^{l}c-1}^{(2)} \in M_{C_{2^n}}(X).
\]
Notably, the generator $j_{2^lc-1}^{(2)}$ is not in $M_{C_{2^n}}(X)$ for any $X \in A(C_{2^{n-1}})$.
\tqed
\end{enumerate}
\end{enumerate}
\end{theorem}

\begin{ex}
Let us illustrate these theorems with an explicit example. Consider $G = C_4$ and $2+[C_2] \in A(C_2)$. If we write the elements of $\Hom(\{1,\ldots,n\},\bbZ)$ as vectors, then \cref{thm:mainroot1} tells us that
\begin{align*}
\phi(\Gamma_{2}(C_2))&=\bbZ\{f^{2,2}_{1,2}, f^{2,2}_{2,2}\}=\bbZ\{(8,0), (4,2)\},& \Gamma_2(C_2) &=\bbZ\{4[C_2],2+[C_2]\}, \\
\phi(\Gamma_3(C_2))&=\bbZ\{f^{2,2}_{1,3}, f^{2,2}_{2,3}\}=\bbZ\{(8,0), (8,4)\},&\Gamma_3(C_2) &= \bbZ\{4[C_2], 4+2[C_2]\}.
\end{align*}
It follows that $2+[C_2] \in \Gamma_2(C_2)\setminus\Gamma_3(C_2)$; equivalently, $M_{C_4}(2+[C_2]) \subset \pi_2 S$. \cref{thm:mainroot2} then tells us that, explicitly,
\[
\eta^2 \in M_{C_4}(2+[C_2]).
\]
See \cref{ex:miexamples} and \cref{ex:miexamplesodd} for several more examples.
\tqed
\end{ex}

\begin{ex}\label{ex:lifting}
Suppose $p > 2$, and let $\rat_{p^n} = \tr_{C_p}^{C_{p^n}}(\ol{\rho}_p^\bbC) = \bbC[C_{p^n}]\otimes_{\bbC[C_p]}\ol{\rho}_p^\bbC$, where $\ol{\rho}_p^{\bbC}$ is the reduced complex regular representation of $C_p$. This is a sum of $p^{n-1}(p-1)$ faithful complex characters, and therefore by \cref{ex:cpn} we have
\[
\im(\Phi^{C_{p^n}}\colon \pi_{\rat_{p^n}}S_{C_{p^n}}\rightarrow\bbZ) = p^2 \bbZ.
\]
Thus if $X \in A(C_{p^{n-1}})$ satisfies $|X^{C_{p^{n-1}}}| = p$ then $|M_{C_{p^n}}(X)| < 2p^{n-1}(p-1)$. \cref{thm:mainroot1} tells us this is optimal: the norm $X = N_e\!\!^{C_{p^{n-1}}}(p)$ satisfies $|X^{C_{p^{n-1}}}| = p$, and more generally $\phi(X) = f^{p,n}_{n,2p^{n-1}(p-1)-1}$, implying $|M_{C_{p^n}}(X)| = 2p^{n-1}(p-1)-1$. This means the following.

The representation sphere $S^{\rat_{p^n}}$ admits an equivariant $(2p^{n-1}(p-1)-1)$-skeleton $\tilde{S}^{\rat_{p^n}}$ for which $S^{\rat_{p^n}}/\tilde{S}^{\rat_{p^n}} = \Sigma^{2p^{n-1}(p-1)}C_{p^n+}$. This means that if $\tilde{X} \in A(C_{p^n})$ is the unique element of the augmentation ideal satisfying $\Phi(\tilde{X}) = X$, then there is a lift $Y$ in
\begin{center}\begin{tikzcd}
&S^0\ar[r,"\tilde{X}"]\ar[d]&S_{C_{p^n}}\\
\Sigma^{2p^{n-1}(p-1)-1}C_{p^n+}\ar[r,"r"]&\tilde{S}^{\rat_{p^n}}\ar[d]\ar[ur,dashed,"Y"]\\
&S^{\rat_{p^n}}\ar[uur,dashed,"\nexists"']
\end{tikzcd},\end{center}
and this lifts no further. The obstruction to a further lift of $Y$ through $S^{\rat_{p^n}}$ is the element $\resp(Y) = r^\ast(Y) \in \pi_{2p^{n-1}(p-1)-1}S$. \cref{thm:mainroot2} tells us that
\[
M_{C_{p^n}}(X) \ni j_{2p^{n-1}(p-1)-1}^{(p)},
\]
and this means that there exists a lift $Y$ for which $\resp(Y)$ is a generator of the $p$-torsion of the image of $J$ in this degree. In this way the $C_{p^n}$-Mahowald invariant gives precise information about what exactly prevents $p$ from being in the image of $\Phi^{C_{p^n}}\colon \pi_{\rat_{p^n}}S_{C_{p^n}}\rightarrow\bbZ$.
\tqed
\end{ex}

\begin{rmk}
As \cref{ex:lifting} shows, \cref{thm:mainroot2} may also be interpreted as the construction of new equivariant lifts of classical elements in the image of $J$. However, this comes with the caveat that these equivariant lifts are generally not stable maps of the form $S^V \to S_{C_{p^n}}$, but instead are stable maps out of a possibly odd-dimensional skeleton of $S^V$. In other words, these lifts are generally not to the $RO(C_{p^n})$-graded equivariant stable stems, but to certain ``spoke-graded'' $C_{p^n}$-equivariant homotopy groups; see \cref{ssec:spoke}.
\tqed
\end{rmk}

\begin{rmk}
In general, these theorems determine the $C_{p^n}$-Mahowald invariant of $X\in A(C_{p^{n-1}})$ in terms of its \textit{fixed point data}. Specifically, define $n\times n$ matrices $A_k^{p,n}$ by
\[
(A_k^{p,n})_{i,j} = f^{p,n}_{j,k}(i).
\]
Then $A_k^{p,n}$ is an upper triangular matrix invertible away from $p$, and the integers $c_1,\ldots,c_n$ appearing in \cref{thm:mainroot2} can be computed as
\[
\begin{pmatrix}
c_1 \\
\vdots \\
c_n\end{pmatrix} = (A_k^{p,n})^{-1} \cdot \begin{pmatrix} |X| \\ \vdots \\ |X^{C_{p^{n-1}}}|\end{pmatrix},
\]
where in general $X \in \Gamma_k(C_{p^{n-1}})$ if and only if these are in fact integers.
\tqed
\end{rmk}

\begin{rmk}
When $G = C_p$ is of prime order, \cref{thm:mainroot2} is a theorem of Mahowald and Ravenel \cite{mahowaldravenel1993root}, relying on a careful analysis of skeleta of stunted projective spaces. In \cite{guillouisaksen2020bredon}, Guillou and Isaksen give a different proof for $p = 2$ using the $C_2$-equivariant Adams spectral sequence. When $p=2$, these treatments prove $P^l\eta \in M_{C_2}(2^{4l+1})$, and Guillou and Isaksen observe that the indeterminacy contains $\bbZ/(2)\{\eta\epsilon\}$ when $l=1$ \cite[Remark 7.3]{guillouisaksen2020bredon}. Our proof shows that the indeterminacy always contains $\bbZ/(2)\{P^{l-1}\eta\epsilon\}$ for $l\geq 1$, which appears to be new.
\tqed
\end{rmk}

We prove these theorems by reducing to a calculation in equivariant $K$-theory, which we then carry out. This reduction uses an equivariant version of the \textit{Adams conjecture} that we piece together in \cref{ssec:adamsconjecture}. In particular, \cref{ex:isomodtorsion} shows that if $V$ is a representation of $C_{p^n}$, then the Hurewicz map
\[
\pi_V S_{C_{p^n}} \rightarrow \pi_V KO_{C_{p^n}}
\]
induces an isomorphism modulo torsion onto the subgroup fixed by the Adams operation $\psi^\ell$, where $\ell$ is a primitive root mod $p^2$ (say $\ell = 3$ if $p=2$). Given this, \cref{thm:fixeddegree} amounts to a computation of these fixed points, and similar if more involved considerations lead to \cref{thm:mainroot1} and \cref{thm:mainroot2}. More generally, we advertise that \cref{thm:equivadams} reduces the following problem to an accessible computation in equivariant $K$-theory.

\begin{question}\label{question:cyclicdegs}
Given a $p$-group $G$ and $G$-representation $V$, what is an explicit description of the image of the cyclic fixed point function
\[
\deg\colon \pi_V S_{G} \rightarrow \Hom(\Sub^\cyc(G),\bbZ),\qquad (\deg y)(C) = \begin{cases}\Phi^Cy & V^C = 0,\\ 0&\text{ otherwise}.
\end{cases}?
\]
This is already interesting for $G = C_{p^n}$, where it would describe the reduced ring $\pi_\star^{\red}S_{C_{p^n}}$ in representation degrees; we have only analyzed the case where $V^{C_p} = 0$.
\tqed
\end{question}

\subsection{Notation}

For convenience of the reader, we collect together here some of the definitions and notation used throughout the paper.

\begin{itemize}
\item In general, we will use $G$ for a finite group, $Z$ for a (typically compact) $G$-space, $X$ for a (possibly virtual) $G$-set, and $V$ for a (non-virtual) $G$-representation.

\item $\pi_\star^G E = \pi_\star E_G$ are the $RO(G)$-graded homotopy groups of a $G$-spectrum $E = E_G$.
\item $\res_H^G$, $\tr_H^G$, and $N_H^G$ are restriction, transfer, and norm maps.
\item $\tr_i^n = \tr_{C_{p^i}}^{C_{p^n}}$ and $\res_i^n = \res_{C_{p^i}}^{C_{p^n}}$ for a given prime $p$.\setlength\itemsep{1em}

\item \setlength\itemsep{0em} $C_m \subset U(1) = T$ is the subgroup of order $m$ with tautological complex character $L$.
\item $d_1,d_2,\ldots$ is the sequence specified in \cref{rmk:sequence} and $V_k = L^{d_1}\oplus\cdots\oplus L^{d_k}$ (\cref{eq:vk}).
\item $\pi_{\star,\ast}^{C_m}E = \pi_{\star,\ast}E_{C_m}$ are the spoke-graded homotopy groups of a $C_m$-spectrum $E = E_{C_m}$ (\cref{def:spokegrading}). In particular, $\pi_{0,2k}^{C_m}E = \pi_{V_k}^{C_m}E$.
\item $\trp$, $\resp$, and $a^{1/2}$ are operations on spoke-graded homotopy defined in \cref{def:tra}.
\item Restrictions and transfers in spoke-graded homotopy are defined in \cref{def:restrspoke}.\setlength\itemsep{1em}

\item \setlength\itemsep{0em} $S$ and $S_G$ are the nonequivariant and $G$-equivariant sphere spectra.
\item $A(G) = \pi_0 S_G$ is the Burnside ring of $G$, and $I(G) \subset A(G)$ is its augmentation ideal.
\item $\Phi^H\colon \Sp^G \to \Sp$ is the functor of $H$-geometric fixed points.
\item $\Phi\colon \Sp^{C_{p^n}} \to \Sp^{C_{p^{n-1}}}$ is genuine $C_p$-geometric fixed points (\cref{def:genuinefp}).
\item $\Phi\colon A(C_{p^n}) \to A(C_{p^n}/C_p)\cong A(C_{p^{n-1}})$ is the $C_p$-fixed point map (\cref{def:burnsideweylphi}).
\item $\Sub(G)$ is the set of subgroups of $G$ and $\Sub^\cyc(G)\subset\Sub(G)$ the cyclic subgroups.
\item $\Cl(\Sub(G),\bbZ)$ is the ring of conjugation-invariant functions $\Sub(G) \to \bbZ$ and $\phi\colon A(G) \to \Cl(\Sub(G),\bbZ)$ is the marks homomorphism (\cref{def:marks2}).
\item We sometimes identify $\Sub(C_{p^{n-1}}) = \{1,\ldots,n\}$, with $i\leftrightsquigarrow C_{p^{i-1}}$, and write $\tilde{\phi} = \phi\circ \Phi\colon A(C_{p^n}) \to \Fun(\{1,\ldots,n\},\bbZ)$ (\cref{rmk:gammak}).\setlength\itemsep{1em}

\item \setlength\itemsep{0em} $S(V)$ is the unit sphere in a given $G$-representation $V$.
\item $SZ$ is the unreduced suspension of a $G$-space $Z$ and $a_Z\colon S^0 \to SZ$ the inclusion of cone points. In particular, $S(S(V))\simeq S^V$ and $a_{S(V)} = a_V$ for a $G$-representation $V$.
\item $D(\bs) = F(\bs,S_G)$ is the functor of (equivariant) Spanier--Whitehead duality.
\item $I_Z = I_Z^G = \ker(A(G) \to \pi_0^G D(\Sigma^\infty_+ Z))\subset A(G)$ for a $G$-space $Z$ (\cref{def:iz}).\setlength\itemsep{1em}

\item \setlength\itemsep{0em} $EG^{\leq\bullet}$ is a cellular filtration on the universal space $EG$ and $\widetilde{E}G^{\leq\bullet}$ an associated filtration on $\widetilde{E}G = S(EG)$ (\cref{eq:egfiltration}). Specifics for $G = C_m$ are in \cref{ssec:spoke}.
\item  $m_k(G) = I_{EG^{\leq k-1}}\subset A(G)$ is the Mahowald filtration (\cref{def:mfilt}).
\item $M_{C_{p^n}}\colon \pi_\star S_{C_{p^{n-1}}} \rightharpoonup \pi_\ast S$ is the $C_{p^n}$-Mahowald invariant (\cref{def:cpnmifull}).
\item $\Gamma_k(C_{p^{n-1}}) = \{X \in A(C_{p^{n-1}}) : |M_{C_{p^n}}(X)|\geq k\} = \Phi(m_k(C_{p^n}))$ (\cref{def:gammafilt}, \cref{rmk:gammak}).
\item $f_{s,k}^{p,n}$ is the $s$th chosen basis element for $\phi(\Gamma_k(C_{p^{n-1}}))$ (\cref{def:thebasis}).
\item $t_{n,i},z_{n,i}\in A(C_{p^n})$ are defined in \cref{def:burnsidebases}.\setlength\itemsep{1em}

\item \setlength\itemsep{0em} $KU_G$ and $KO_G$ are the $G$-spectra of equivariant complex and real $K$-theory, with $\psi^\ell$ the $\ell$th Adams operation acting on $KU_G[1/\ell]$ and $KO_G[1/\ell]$ for $\gcd(\ell,|G|) = 1$.
\item $RU(G) = \pi_0 KU_G$ and $RO(G) = \pi_0 KO_G$ and $R\bbQ(G)$ are the complex, real, and rational representation rings of $G$ (see especially \cref{rmk:rationalreps} for further discussion). 
\item $\beta_V \in \pi_V KU_G$ is the Bott class of a complex $G$-representation $V$ and $e_V = a_V \beta_V = \sum_i (-1)^i \Lambda^i V\in \pi_0 KU_G = RU(G)$ its Euler class (\cref{def:bottandeuler}).
\item $\rat_m$ is the fixed point free irreducible rational representation of $C_m$ (\cref{def:rat}), with dimension equal to the value of the Euler totient function $\varphi(m)$.
\item $\beta_k = \beta_{V_k} \in \pi_{0,2k}KU_{C_{m}}$ and $e_{a,b} = e_{V_b-V_a}$ for $a \leq b$ (\cref{eq:eab}).
\item $M_k = M_k^{\smash{C_{p^n}}} \subset \pi_{0,k}KO_{C_{p^n}}/(\text{torsion})$ and $M_k^\bbC = M_k^{\smash{\bbC,C_{p^n}}}\subset \pi_{0,k}KU_{C_{p^n}}$ are the subgroups fixed by Adams operations, required to further lie in the augmentation ideal when $k=0$ (\cref{def:fixedmodule}, \cref{def:cfixedmodule}).\setlength\itemsep{1em}

\item \setlength\itemsep{0em} $J = J_p$ and $J_G$ are the nonequivariant and $G$-equivariant nonconnective $J$-spectra at a given prime $p$ (\cref{def:j}).
\item $j_{2k(p-1)-1}^{(p)} \in \pi_{2k(p-1)-1} J_p$ is a specified generator (\cref{def:jgens}).
\end{itemize}

\subsection*{Acknowledgements}
The authors would like to thank Mark Behrens and Zhouli Xu for their support and helpful comments on earlier versions of this paper, and the anonymous referee for their careful reading and helpful suggestions.

\section{Mahowald filtrations and Mahowald invariants}\label{sec:lifting}

In this section, we develop the theoretical framework necessary to define the $C_{p^n}$-equivariant Mahowald invariant and reduce the computation to equivariant $K$-theory. We begin by discussing a lifting problem in \cref{ssec:lifting} which generalizes \cref{problem:fpdegree} in the case where $\alpha = V$ is a non-virtual representation (see \cref{ex:degreeaslifting}).

In \cref{ssec:adamsconjecture}, we explain how a version of the Adams conjecture may be used to reduce this problem for cyclic $p$-groups to a computation in equivariant $K$-theory. This reduction is the key input that allows us to access information about the equivariant stable stems, making the work in this paper possible.

The additional flexibility provided by this more general lifting problem allows us to define a \emph{Mahowald filtration} on the Burnside ring $A(G)$ in \cref{ssec:mahowaldfilt}, closely related to the homotopy fixed point and Tate spectral sequences. This filtration is defined in terms of the cellular filtration on the universal space $EG$. In \cref{ssec:spoke}, we give a detailed account of the ``standard'' cell structure on the universal space of a cyclic group, and use this to introduce a formalism of ``spoke-graded'' equivariant homotopy groups for cyclic groups.

Finally, in \cref{ssec:cpnmi} we specialize to cyclic $p$-groups. We formally define the $C_{p^n}$-Mahowald invariant and explain how it may be read off the spoke-graded $C_{p^n}$-equivariant stable stems, generalizing the relation between classical $C_2$-Mahowald invariants and $RO(C_2)$-graded $C_2$-equivariant stable stems.

\subsection{A lifting problem}\label{ssec:lifting}

Let $G$ be a finite group and write $A(G) = \pi_0 S_G$ for the Burnside ring of $G$. Let $Z$ be a compact $G$-space and $SZ$ its unreduced suspension, fitting into a cofiber sequence
\[
Z_+\rightarrow S^0\xrightarrow{a_Z} SZ.
\]

\begin{ex}\label{ex:svsz}
If $V$ is a $G$-representation and $Z = S(V)$ is the unit sphere in $V$, then $SZ = S^V$ and $a_Z = a_V\colon S^0\rightarrow S^V$ is the Euler class of $V$.
\tqed
\end{ex}

We are interested in the following ideal of the Burnside ring.

\begin{defn}\label{def:iz}
Given a compact $G$-space $Z$, we write
\begin{align*}
I_Z = I_Z^G &= \im(\pi_0^G D(\Sigma^\infty SZ)\xrightarrow{a_Z}\pi_0S_G\cong A(G))\\
&= \ker(A(G)\cong \pi_0 S_G\rightarrow \pi_0^G D(\Sigma^\infty_+ Z)).\tag*{$\triangleleft$}
\end{align*}
\end{defn}

\begin{rmk}
In other words, $I_Z\subset A(G)$ consists of those virtual $G$-sets $X$ for which one can solve the lifting problem
\begin{equation}\label{eq:extension}
\begin{tikzcd}
\Sigma^\infty_+Z \ar[dr,"0",dashed]\ar[d]\\
S^0\ar[r,"X"]\ar[d,"a_Z"']&S_G\\
SZ\ar[ur,dashed]
\end{tikzcd}.\end{equation}
\end{rmk}

\begin{ex}
When $Z = G$ is a free $G$-orbit, we may identify
\[
I_G = \ker \left(\pi_0 S_G \to \pi_0^GD(\Sigma^\infty_+ G) \cong \pi_0 S\right) \cong \ker \left(\res^G_e\colon A(G) \to A(e) = \bbZ\right)
\]
as the augmentation ideal $I(G)\subset A(G)$.
\end{ex}

It is worth recording the following basic observation explicitly.

\begin{lemma}\label{prop:idealinclusions}
If there exists a map $f\colon Z\rightarrow Z'$ of compact $G$-spaces, then $I_Z\subset I_{Z'}$.
\qed
\end{lemma}

As a general rule, $G$-equivariant homotopy theory simplifies when $G$ is a $p$-group. In particular, the computation of $I_Z$ can in this case be completely reduced to a $p$-primary problem. Let $I_Z^\wedge\subset A(G)_p^\wedge$ be the ideal controlling the analogue of the lifting problem of \cref{eq:extension} inside the $p$-complete equivariant stable category: 
\[
I_Z^\wedge = \im(\pi_0^G F(\Sigma^\infty SZ,(S_G)_p^\wedge)\rightarrow A(G)_p^\wedge) = \ker(A(G)_p^\wedge\rightarrow \pi_0 F(\Sigma^\infty_+ Z,(S_G)_p^\wedge).
\]
As $Z$ is compact we can identify
\[
\bbZ_p \otimes \pi_0 D(\Sigma^\infty_+ Z) \cong \pi_0 D(\Sigma^\infty_+ Z)_p^\wedge\cong  \pi_0 F(\Sigma^\infty_+ Z,(S_G)_p^\wedge),
\]
and so $I_Z^\wedge$ is determined by $I_Z$ via the formula
\[
I_Z^\wedge = \bbZ_p \otimes I_Z.
\]
This is true for an arbitrary group $G$, but under the assumption that $G$ is a $p$-group one can go the other way as well.

\begin{prop}\label{prop:pgroupdegrees}
If $G$ is a $p$-group, then the square
\begin{center}\begin{tikzcd}
I_Z\ar[r]\ar[d]&I_Z^\wedge\ar[d]\\
A(G)\ar[r]&A(G)_p^\wedge
\end{tikzcd}\end{center}
is Cartesian.
\end{prop}
\begin{proof}
As $\bbZ_p \oplus \bbZ[\tfrac{1}{p}]$ is faithfully flat over $\bbZ$, it suffices to prove that this square is Cartesian after tensoring with $\bbZ_p$ and with $\bbZ[\tfrac{1}{p}]$. The former case is clear, as $I_Z^\wedge\cong\bbZ_p\otimes I_Z$ and $A(G)_p^\wedge = \bbZ_p \otimes A(G)$ as discussed above. It remains to consider the latter case. We have
\[
\bbZ[\tfrac{1}{p}]\otimes I_Z = \ker\left(A(G)[\tfrac{1}{p}]\rightarrow \pi_0 D(\Sigma^\infty_+ Z)[\tfrac{1}{p}]\right).
\]
The splitting of $G$-spectra away from the order of $G$ by geometric fixed points, a convenient reference for which is \cite[Corollary 3.4.28]{schwede2018global}, implies that this splits into a product over conjugacy classes of subgroups $H\subset G$ of the kernels of
\[
\bbZ[\tfrac{1}{p}] \cong \pi_0 S[\tfrac{1}{p}]\rightarrow \pi_0 \Phi^H D(\Sigma^\infty_+ Z)[\tfrac{1}{p}]^{W_GH} \cong \pi_0 D(\Sigma^\infty_+(Z^H))[\tfrac{1}{p}]^{W_GH},
\]
where the last identification uses the assumption that $Z$ is compact. This kernel is either trivial if $Z^H \neq \emptyset$ or everything if $Z^H = \emptyset$. The analogous computation holds for
\[
\bbZ[\tfrac{1}{p}]\otimes I_Z^\wedge = \ker\left(\bbQ_p \otimes A(G)\rightarrow \bbQ_p \otimes \pi_0 D(\Sigma^\infty_+ Z)\right),
\]
so after inverting $p$ the square under consideration is a product of Cartesian squares of the form
\begin{center}\begin{tikzcd}
\bbZ[\tfrac{1}{p}]\ar[d]\ar[r]&\bbQ_p\ar[d]\\
\bbZ[\tfrac{1}{p}]\ar[r]&\bbQ_p\end{tikzcd}\hspace{1cm}and\hspace{1cm}\begin{tikzcd}
0\ar[r]\ar[d]&0\ar[d]\\
\bbZ[\tfrac{1}{p}]\ar[r]&\bbQ_p
\end{tikzcd},\end{center}
and is thus itself Cartesian.
\end{proof}

We make some further remarks about the ideals $I_Z$.

\begin{rmk}\label{rmk:rankiz}
The same proof as in \cref{prop:pgroupdegrees} shows that $I_Z\subset A(G)$ is always contained in the ideal
\[
I = \{X\in A(G) : Z^H \neq \emptyset \Rightarrow |X^H| = 0\text{ for }H\subset G\}\subset A(G),
\]
and moreover that $I_Z[\tfrac{1}{|G|}] = I[\tfrac{1}{|G|}]$, or equivalently, that $I/I_Z$ is $|G|$-power torsion. In particular, this determines the rank of $I_Z$ as a free abelian group in terms of basic information about the $G$-action on $Z$.
\tqed
\end{rmk}

\begin{rmk}
In particular, if $Z$ is nonempty then $I_Z$ is contained in the augmentation ideal $I(G)\subset A(G)$. If we choose a basepoint for the underlying space of $Z$ and write $\Sigma^\infty_- Z = \Cof(G_+\rightarrow \Sigma^\infty_+ Z)$, then it follows that
\[
I_Z = \ker(I(G)\rightarrow \pi_0^G D(\Sigma^\infty_- Z)).\tag*{$\triangleleft$}
\]
\end{rmk}

\begin{rmk}\label{rmk:reducetopgroup}
If $G$ acts freely on a nonempty space $Z$, then we can identify $\pi_0^G D(\Sigma^\infty_- Z)$ with the nonequivariant cohomotopy group $\pi_0 D(\Sigma^\infty Z_{\h G}) \cong \pi_0 D(\Sigma^\infty(Z/G))$. If $p$ is a prime and $G_p\subset G$ a Sylow $p$-subgroup, then the map
\[
\pi_0 D(\Sigma^\infty Z_{\h G})_{(p)} \rightarrowtail \pi_0 D(\Sigma^\infty Z_{\h G_p})_{(p)}
\]
is an injection, and thus $(I(G)/I_Z)_{(p)}\subset I(G_p)/I_Z^{G_p}$. As the quotient $I(G)/I_Z$ is $|G|$-power torsion, it follows that
\[
I(G)/I_Z \cong \bigoplus_{p\text{ dividing } |G|} (I(G)/I_Z)_{(p)}\subset \bigoplus_{p\text{ dividing } |G|} I(G_p)/I_Z^{G_p},
\]
and so
\[
I_Z = \ker\left(I(G)\rightarrow \prod_{p\text{ dividing }|G|}I(G_p)\rightarrow  \prod_{p\text{ dividing }|G|}I(G_p)/I_Z^{G_p}\right).
\]
Thus if $G$ acts freely on $Z$ then $I_Z$ is determined algebraically by the case where $G$ is a $p$-group.
\tqed
\end{rmk}

\begin{rmk}
If we choose for each prime $p$ a Sylow $p$-subgroup $G_p\subset G$, then the kernel of the joint restriction $A(G)\rightarrow \prod_p A(G_p)$ is exactly the kernel of the completion map $A(G)\rightarrow A(G)_{I(G)}^\wedge$ \cite[Proposition 1.10]{laitinen1979burnside}. The Segal conjecture, proved by Carlsson \cite{carlsson1983equivariant}, implies that this can be identified as the intersection $\bigcap_Z I_Z$ as $Z$ ranges through the compact free $G$-spaces, or equivalently just through a skeleton of $EG$. In particular, this intersection is trivial if $G$ is a $p$-group.
\tqed
\end{rmk}

\subsection{Cyclic fixed points and the Adams conjecture}\label{ssec:adamsconjecture}

Let $G$ be a finite group and $Z$ be a compact $G$-space as before. 

\begin{defn}\label{def:marks2}
The group $G$ acts on the set $\Sub(G)$ of its subgroups by conjugation. We write
\[
\Cl(\Sub(G),\bbZ)
\]
for the ring of integer-valued conjugation-invariant functions on $\Sub(G)$, and
\[
\phi\colon A(G) \to \Cl(\Sub(G),\bbZ),\qquad (\phi X)(H) = |X^H|
\]
for the \emph{marks homomorphism}.
\tqed
\end{defn}

The marks homomorphism provides a character theory for the Burnside ring: $\phi$ is an injection, and is an isomorphism after inverting the order of $G$. In particular, the ideal $I_Z\subset A(G)$ of \cref{def:iz} is determined by its image $\phi(I_Z)\subset \Cl(\Sub(G),\bbZ)$.

\begin{ex}\label{ex:burnside}
Under the isomorphism $A(G)\cong \pi_0 S_G$, the marks homomorphism is realized topologically as the degree map
\[
\deg\colon \pi_0 S_G \to \Cl(\Sub(G),\bbZ)
\]
for which, if $f\in \pi_0 S_G$ is represented by a map $f\colon S^U \to S^U$ between representation spheres, then
\[
(\deg f)(H) = \Phi^H(f) = \deg(f^H \colon S^{U^H} \to S^{U^H}) \in \bbZ;
\]
see \cite[Theorem 7.6.7]{dieck1979transformation}.
\tqed
\end{ex}

\begin{ex}\label{ex:degreeaslifting}
If $V$ is a $G$-representation, then following \cref{ex:svsz} we have
\[
I_{S(V)} = \im\left(a_V\colon \pi_V S_G \to \pi_0 S_G\right).
\]
By \cref{ex:burnside}, it follows that $\phi(I_{S(V)})\subset\Cl(\Sub(G),\bbZ)$ is the image of degree map
\[
\deg  = \phi\circ a_V\colon \pi_VS_G \to \pi_0 S_G \to \Cl(\Sub(G),\bbZ)
\]
which sends the class of a map $f\colon S^{V\oplus U} \to S^U$ to the degree function
\[
(\deg f)(H) = \Phi^H(a_V f) =  \begin{cases}\deg(f^H\colon S^{V^H}\to S^{V^H})&V^H = 0,\\ 0&\text{ otherwise}.\end{cases}
\]
Thus $\phi(I_{S(V)})\subset\Cl(\Sub(G),\bbZ)$ is exactly the set of degree functions for stable maps $S^V \to S^0$, i.e.\ those functions of the above form for some $G$-representation $U$ and equivariant map $f\colon S^{V\oplus U} \to S^U$. 
\tqed
\end{ex}

\begin{rmk}\label{rmk:composite}
By \cref{rmk:reducetopgroup}, it follows that if $V$ is a fixed point free $G$-representation, then the set $\phi(I_{S(V)})\subset\Cl(\Sub(G),\bbZ)$ of degree functions for stable maps $S^V\rightarrow S^0$ is determined by the corresponding sets $\phi(I_{S(V)}^{G_p})\subset\Cl(\Sub(G_p),\bbZ)$ for a choice of Sylow $p$-subgroups $G_p\subset G$. In particular, \cref{thm:fixeddegree} algebraically determines the set of degree functions for stable maps $S^V\rightarrow S^0$ whenever $V$ is a fixed point free representation of a composite order cyclic group, although we shall not attempt to give an explicit description of this.
\tqed
\end{rmk}

\cref{ex:degreeaslifting} shows that the ideals $I_Z$ control a strong form of \cref{problem:fpdegree} in the case where $\alpha = V$ is an actual $G$-representation. Computing these ideals seems to be hard in general. The situation becomes significantly more tractable if $G$ is a $p$-group and we restrict $\phi(I_Z)$ to the \textit{cyclic} subgroups of $G$.

Restricting to the cyclic subgroups of $G$ amounts to restricting to the part of $A(G)$ seen by classical representation theory. We summarize the relevant bit of representation theory in the following two remarks.

\begin{rmk}\label{rmk:rationalreps}
Let $R\bbQ(G)\subset RO(G)\subset RU(G)$ denote the rational, real and complex representation rings of $G$. Then there is a commutative diagram
\begin{center}\begin{tikzcd}[column sep=small]
A(G)\ar[rrr,"{\bbQ[\bs]}"]\ar[d,"\phi",tail]&&&R\bbQ(G)\ar[d,"\chi",tail]\ar[r,tail]\ar[dll,"\phi"',tail,dashed]&RO(G)\ar[d,"\chi",tail]\ar[r,tail]&RU(G)\ar[d,"\chi",tail]\\
\Cl(\Sub(G),\bbZ)\ar[r,"\res", two heads]&\Cl(\Sub^\cyc(G),\bbZ)\ar[r,tail]&\ar[r]\Cl(G,\bbZ)\ar[r,tail]&\Cl(G,\bbQ)\ar[r,tail]&\Cl(G,\bbR)\ar[r,tail]&\Cl(G,\bbC)
\end{tikzcd},\end{center}
where $\bbQ[\bs]$ sends a $G$-set to its associated permutation representation, $\Sub^\cyc(G)\subset \Sub(G)$ is the subset of cyclic subgroups, $\Cl(\Sub^\cyc(G),\bbZ)\rightarrow\Cl(G,\bbZ)$ restricts along the map $G\rightarrow \Sub^\cyc(G)$ sending $g\in G$ to the cyclic subgroup it generates, and $\chi$ is the usual character map.

The Adams operation $\psi^\ell$ acts on characters by the $\ell$th power map on $G$, and if $\gcd(\ell,|G|) = 1$ then this may be interpreted as an action on $RO(G)$ and $RU(G)$ by Galois conjugation; see \cite[Part II, \S3]{atiyahtall1969group} for a discussion. In particular, $R\bbQ(G)\subset RU(G)$ lands in the subring fixed by the action of $\psi^\ell$ for all $(\ell,|G|) = 1$. As $\psi^\ell$ acts on characters by the $\ell$th power map on $G$ and as rational representations have integer-valued characters \cite[Chapter 6, Proposition 15]{serre1977linear}, it follows that the rational character map $\chi\colon R\bbQ(G) \to \Cl(G,\bbQ)$ lifts to $\Cl(\Sub^\cyc(G))$ as pictured. This lift is an isomorphism after inverting the order of $G$.
\tqed
\end{rmk}

\begin{rmk}\label{rmk:segalritter}
If for example $G$ is cyclic or an odd order $p$-group, then $R\bbQ(G)\subset RU(G)$ is the subring fixed by the action of $\psi^\ell$ for all $\gcd(\ell,|G|)$. If $G$ is a $2$-group, then it is the fixed subring of $RO(G)$. See \cite[Corollary 9.3.2]{dieck1979transformation}. Moreover, in these cases the map
\[
\bbQ[\bs]\colon A(G) \to R\bbQ(G)
\]
is a surjection, and is an isomorphism when $G$ is cyclic. For $p$-groups, this is the Ritter--Segal theorem \cite{ritter1972induktionssatz,segal1972permutation}, and for cyclic groups it may be seen by direct inspection.
\tqed
\end{rmk}

We now explain how the Adams and Segal conjectures, theorems of Quillen \cite{quillen1971adams} and Carlsson \cite{carlsson1983equivariant} respectively, may be combined to produce a $G$-equivariant version of the Adams conjecture which categorifies the Ritter--Segal theorem.

Suppose for the rest of this subsection that $G$ is a $p$-group, and let $\ell$ be a primitive root mod $p^2$ (say $\ell = 3$ if $p=2$). Write
\[
\psi^\ell\colon (KO_G)_p^\wedge \to (KO_G)_p^\wedge
\]
for the $\ell$th Adams operation acting on $p$-adic equivariant $KO$-theory.

\begin{defn}\label{def:j}
We define the $G$-equivariant $J$-spectrum (at the prime $p$) by
\begin{equation}\label{eq:jg}
J_G = \Fib\left(\begin{tikzcd} (KO_G)_p^\wedge\ar[r,"\psi^\ell-\psi^1"]& (KO_G)_p^\wedge\end{tikzcd}\right).
\end{equation}
When $G = e$ we write $J_p = J_e$ to incorporate the prime into the notation.
\tqed
\end{defn}

\begin{rmk}
The $G$-spectrum $J_G$ does not depend on the choice of $\ell$: it can be identified as the Bousfield localization $L_{KU_G/p}S_G$ of the $G$-equivariant sphere spectrum with respect to equivariant $K$-theory mod $p$, or as the Borel spectrum $F(EG_+,i_\ast J_p)$ associated to the classical $K(1) = KU/p$-local sphere $J_p \simeq S_{K(1)}$. In particular, if $p$ is odd then one is free to replace $KO_G$ with $KU_G$ in the definition. See \cite[Proposition A.4.12]{balderrama2022total} for details and further discussion.
\tqed
\end{rmk}

\begin{rmk}\label{rmk:jses}
An automorphism $\psi^\ell\colon M \to M$ of an abelian group determines an action of $\bbZ$ on $M$. If we abbreviate the group cohomology of this action by
\[
H^\ast(\psi^\ell;M) = H^\ast(\bbZ\{\psi^\ell\},M),
\]
then we may identify 
\[
H^0(\psi^\ell,M) = \ker\left(\psi^\ell-\id\colon M \to M\right),\qquad H^1(\psi^\ell,M) = \coker\left(\psi^\ell-\id\colon M \to M\right).
\]
Thus the defining fiber sequence for $J_G$ yields short exact sequences
\begin{center}\begin{tikzcd}
0\ar[r]&H^1(\psi^\ell;((KO_G)_p^\wedge)^{-1}(Y)) \ar[r]&J_G^0(Y)\ar[r]&H^0(\psi^\ell;((KO_G)_p^\wedge)^0(Y))\ar[r]&0
\end{tikzcd}\end{center}
for any $G$-spectrum $Y$.
\tqed
\end{rmk}

In particular, we have the following.

\begin{lemma}
For any compact $G$-spectrum $Y$, the unit map $J_G \to (KO_G)_p^\wedge$ induces a surjection
\[
J_G^0(Y) \to H^0(\psi^\ell;\bbZ_p \otimes KO_G^0(Y)).
\]
In particular, $\pi_0 J_G/\text{torsion} \cong R\bbQ(G)_p^\wedge$.
\end{lemma}
\begin{proof}
The assumption that $Y$ is compact guarantees that
\[
((KO_G)_p^\wedge)^\ast(Y) \cong \bbZ_p \otimes KO_G^\ast(Y),
\]
so the given surjection follows from the defining fiber sequence for $J_G$ as in \cref{rmk:jses}. Taking $Y = \ast$, this extends to a short exact sequence
\[
0 \to H^1(\psi^\ell;\bbZ_p \otimes \pi_{1}KO_G) \to \pi_0 J_G \to H^0(\psi^\ell;\bbZ_p \otimes \pi_0 KO_G) \to 0.
\]
As discussed in \cref{rmk:rationalreps}, we may identify
\[
H^0(\psi^\ell;\bbZ_p \otimes \pi_0 KO_G) \cong H^0(\psi^\ell;\bbZ_p \otimes RO(G)) \cong R\bbQ(G)_p^\wedge
\]
for any $p$-group $G$, so the final isomorphism follows as $\pi_1 KO_G$ is torsion; see for example \cite[Section 9]{mathewnaumannnoel2017nilpotence}.
\end{proof}

In this way we find that the unit $S_G\rightarrow J_G$ categorifies the map $A(G)\rightarrow R\bbQ(G)$, at least up to $p$-completion; see \cite{bonventreguilloustapleton2022kug} for an integral statement.

\begin{prop}\label{prop:adamsconjecture}
Let $T$ be a compact pointed $G$-space. Then
\begin{enumerate}
\item The map $\bbZ_p \otimes \pi_0^G D(\Sigma^\infty T) \rightarrow \widetilde{J}_G^0(T)$ is a surjection mod torsion;
\item If $p>2$ or the Borel construction $T_{\h G} = EG_+\wedge_G T$ is simply connected, then
\[
\bbZ_p \otimes \pi_0^G D(\Sigma^\infty T) \rightarrow \widetilde{J}_G^0(T)
\]
is a naturally split surjection, this naturality being in maps of pointed $G$-spaces.
\end{enumerate}
\end{prop}
\begin{proof}
For a nonequivariant spectrum $Y$, define $\widehat{D}(Y) = F(Y,S_p^\wedge)$. As $T$ is compact, the Segal conjecture implies
\[
\bbZ_p \otimes \pi_0^G D(\Sigma^\infty T)\cong \pi_0 \widehat{D}(\Sigma^\infty T_{\h G}).
\]
Likewise the identification $J_G \simeq F(EG_+,i_\ast J_p)$ implies
\[
\widetilde{J}^0_G(T) \cong \widetilde{J}_p^0(T_{\h G}).
\]

Thus the proposition reduces to the nonequivariant claim that if $W$ is a pointed space, then
$
\pi_0 \widehat{D}(\Sigma^\infty W) \to \widetilde{J}_p^0(W)
$
is a surjection mod torsion, and is a split surjection if $p>2$ or $W$ is simply connected, this splitting being natural in $W$ as a pointed space. As
$
\widetilde{J}_p^0( W) \cong \pi_0 \Map_\ast(W,\Omega^\infty J_p),
$
this is a consequence of the Adams conjecture in its strong form as a space-level splitting of the image of $J$ off $QS^0$ at a given prime, as we now explain.

Let $j_p$ be the $p$-complete connective $J$-spectrum: this is the connective cover of $J_p$ if $p>2$ and the fiber $\Fib(\psi^3-1\colon ko_2^\wedge\rightarrow ksp_2^\wedge)$ when $p=2$. By the positive resolution of the Adams conjecture by Quillen \cite{quillen1971adams}, but in its space-level form as in e.g.\ \cite[Corollary 4.6]{may1977einfty}, the unit map $S_p^\wedge\rightarrow j_p$ induces the projection in a space-level splitting
\[
\Omega^\infty S_p^\wedge\simeq F\times \Omega^\infty j_p.
\]
It follows that $\pi_0 \widehat{D}(\Sigma^\infty W)\rightarrow \widetilde{j}_p^0(W)$ is a split surjection, natural in the pointed space $W$. So it suffices to prove that $\widetilde{J}_p^0(W)\rightarrow \widetilde{j}_p^0(W)$ is a surjection mod torsion, and an isomorphism if $p>2$ or $W$ is simply connected.

If $p>2$, then $j_p$ is the connective cover of $J_p$, and so this is clear as 
\[
\widetilde{J}_p^0(W) \cong \pi_0\Map_\ast(W,\Omega^\infty J_p) \cong \pi_0\Map_\ast(W,\Omega^\infty \tau_{\geq 0} J_p) \cong \pi_0\Map_\ast(W,\Omega^\infty j_p).
\]
If $p=2$, then a bit more care is needed as $j_2$ differs slightly from the connective cover of $J_2$. Instead, there is a fiber sequence
\[
\Omega^\infty j_2 \rightarrow \Omega^\infty J_2 \rightarrow K(\bbZ/2,0)\times K(\bbZ/2,1),
\]
and so $\widetilde{J}_2^0(W) \to \widetilde{j}_2^0(W)$ is generally only a surjection either mod torsion or if $W$ is simply connected. In either case the proposition follows.
\end{proof}

\begin{rmk}
In particular, naturality applied to the collapse maps $G/H_+\wedge T \rightarrow T$ ensures that the splitting of $\bbZ_p \otimes \pi_0^G D(\Sigma^\infty T)\rightarrow \widetilde{J}_G^0(T)$ is compatible with restrictions.
\tqed
\end{rmk}

\cref{prop:adamsconjecture} has the following consequence.

\begin{theorem}\label{thm:equivadams}
Let $G$ be a $p$-group and $Z$ be a compact $G$-space. Then the restriction of $\phi(I_Z)\subset\Cl(\Sub(G),\bbZ)$ to the cyclic subgroups of $G$ is equal to
\[
\im\left(\phi\circ a_Z\colon H^0(\psi^\ell;\widetilde{KO}{}_G^0(SZ))\rightarrow R\bbQ(G)\rightarrowtail \Cl(\Sub^\cyc(G),\bbZ)\right).
\]
\end{theorem}
\begin{proof}
Let $I_Z^\bbQ = \im \left(\phi\circ a_Z\colon H^0(\psi^\ell;\widetilde{KO}{}_G^0(SZ))\rightarrow R\bbQ(G)\right)$ and write $I_Z^\cyc$ for the image of $I_Z$ under the map $\bbQ[\bs]\colon A(G)\rightarrow R\bbQ(G)$. Then it is equivalent to show $I_Z^\cyc = I_Z^\bbQ$ inside $R\bbQ(G)$. Consider the diagram
\begin{center}\begin{tikzcd}
\pi_0^G D(\Sigma^\infty SZ)\ar[r,"h"]\ar[d,"a_Z"]&H^0(\psi^\ell;\widetilde{KO}{}_G^0SZ)\ar[d,"a_Z"]\\
A(G)\ar[r,"{\bbQ[\bs]}",two heads]&R\bbQ(G)
\end{tikzcd}.\end{center}
The image of the counterclockwise composite is $I_Z^\cyc$ and the image of the right vertical map is $I_Z^\bbQ$, so this furnishes an inclusion $I_Z^\cyc\subset I_Z^\bbQ$ which we claim is an equality. As with \cref{prop:pgroupdegrees}, it suffices to show that the two agree after tensoring with $\bbZ_p$ and after tensoring with $\bbZ[\tfrac{1}{p}]$.

As all spaces in sight are compact, tensoring with $\bbZ_p$ has the effect of $p$-completion. By \cref{prop:adamsconjecture} the top map $h$ is a surjection after $p$-completion, and this proves $\bbZ_p \otimes I_Z^\cyc = \bbZ_p \otimes I_Z^\bbQ$. On the other hand, the map $\phi\colon R\bbQ(G)\rightarrow \Cl(\Sub^\cyc(G),\bbZ)$ becomes an isomorphism after inverting $p$, and the same argument as in the proof of \cref{prop:pgroupdegrees} shows that we can identify both $\phi(I_Z^\cyc[\tfrac{1}{p}])$ and $\phi(I_Z^\bbQ[\tfrac{1}{p}])$ as the set of functions supported on $\{C\in\Sub^\cyc(G) : Z^C = \emptyset\}$, so the theorem follows.
\end{proof}

\begin{cor}\label{cor:reducetoktheory}
If $Z$ is a compact $C_{p^n}$-space, then the Hurewicz map induces an isomorphism
\[
I_Z \cong \im\left(a_Z\colon H^0(\psi^\ell;\widetilde{KO}{}_{C_{p^n}}^0 SZ) \rightarrow R\bbQ(C_{p^n})\right).
\]
\end{cor}
\begin{proof}
The map $A(C)\rightarrow R\bbQ(C)$ is an isomorphism when $C$ is a finite cyclic group, so the corollary follows immediately from \cref{thm:equivadams}.
\end{proof}

\begin{ex}\label{ex:isomodtorsion}
Taking $Z = S(V)$ to be the unit sphere in a $C_{p^n}$-representation $V$, this shows that
\[
\pi_V S_{C_{p^n}}\rightarrow \pi_V KO_{C_{p^n}}
\]
induces an isomorphism mod torsion onto the subgroup fixed by Adams operations.
\tqed
\end{ex}

\begin{question}
Is there a useful analogue of \cref{thm:equivadams} for composite order groups?
\tqed
\end{question}

\subsection{The Mahowald filtration}\label{ssec:mahowaldfilt}

Let $G$ be a finite group and $EG$ be a contractible space with free $G$-action. Fix a choice of equivariant cell structure $EG = \colim_{n\rightarrow\infty} EG^{\leq n}$.

\begin{defn}\label{def:mfilt}
The \textit{Mahowald filtration} on $A(G)$ is defined by
\[
m_k(G) = I_{EG^{\leq k-1}}\subset A(G)
\]
for $k \geq 0$.
\tqed
\end{defn}

\cref{prop:idealinclusions}, together with an application of cellular approximation to the identity map of $EG$, ensures that this is independent of the choice of cell structure. Moreover, \cref{rmk:reducetopgroup} reduces the computation of $m_k(G)$ in general to the case where $G$ is a $p$-group; the Segal conjecture then implies that $\bigcap_k m_k(G) = 0$ and $\lim_k A(G)/m_k(G) = \bbZ \oplus I(G)_p^\wedge$.

\begin{ex}
For any group $G$, we may identify
\[
m_1(G) = I(G)\subset A(G)
\]
as the augmentation ideal. In general, \cref{rmk:rankiz} shows that if $k\geq 1$ then $m_k(G)$ is a free abelian group of rank $|\Sub(G)/\text{conjugacy}| - 1$, and that $m_k(G)\subset I(G)$ is an isomorphism after inverting $|G|$.
\tqed
\end{ex}

The Mahowald filtration is closely tied to the $G$-homotopy fixed point and Tate spectral sequences, as we now explain. 

Write $\widetilde{E}G$ for the unreduced suspension of $EG$, fitting into a cofiber sequence
\[
EG_+\rightarrow S^0 \rightarrow \widetilde{E}G.
\]
We shall omit the $\Sigma^\infty$ when speaking of the suspension spectra of these spaces. Following \cite[Section 9]{greenleesmay1995generalized}, define a $\bbZ$-graded filtration on the spectrum $\widetilde{E}G$ by
\begin{equation}\label{eq:egfiltration}
\widetilde{E}G^{\leq n} = \begin{cases}
S(EG^{\leq n-1})&n\geq 1,\\
S_G&n=0,\\
D(S(EG^{\leq -n-1}))&n\leq -1.
\end{cases}
\end{equation}

Given a $G$-spectrum $E$, write
\[
E^\h = F(EG_+,E),\qquad E^\tate = \widetilde{E}G \wedge E^\h
\]
for the associated Borel and Tate $G$-spectra.
The filtrations of \cref{eq:egfiltration} induce filtrations
\begin{align*}
E^\h &= \lim_{n\rightarrow\infty} F(EG^{\leq n}_+,E),\\
E^\tate &= \colim_{n\rightarrow\infty} E^\h \wedge \widetilde{E}G^{\leq n},
\end{align*}
and Greenlees and May show that the associated spectral sequences can be identified from the $E_2$-page on as homotopy fixed point and Tate spectral sequences
\begin{align*}
E_2^{\alpha,f} &= H^f(G;\pi_{\alpha+f}^e E)\Rightarrow \pi_\alpha^G E^\h \cong \pi_0 F(S^\alpha,E)^{\h G},\\
E_2^{\alpha,f} &= H^f_{\text{tate}}(G;\pi_{\alpha+f}^e E)\Rightarrow \pi_\alpha^G E^\tate \cong \pi_0 F(S^\alpha,E)^{\tate G}.
\end{align*}
They converge if, for example, $E$ is compact. We now have the following.

\begin{prop}\label{prop:mahowaldfiltss}
Fix $X\in A(G)$.
\begin{enumerate}
\item $X\in m_k(G)$ if and only if the image of $X$ in $\pi_0 S_G^\h$ is detected in filtration $\geq k$ in the homotopy fixed point spectral sequence;
\item $X\in m_k(G)$ for $k > 0$ if and only if $X$ is in the augmentation ideal and the image of $X$ in $\pi_0 S_G^\tate$ is detected in filtration $\geq k$ in the Tate spectral sequence.
\end{enumerate}
\end{prop}
\begin{proof}
(1): To say that the image of $X$ in $\pi_0 S^\h$ has filtration $\geq k$ in the homotopy fixed point spectral sequence is exactly to say that the composite 
\[
EG_+^{\leq k-1}\rightarrow S^0\xrightarrow{X}S_G
\]
vanishes, which in turn is equivalent to asking that $X \in I_{EG^{\leq k-1}} = m_k(G)$.

(2): Write $\tilde{X}$ for the image of $X$ in $S_G^\tate$ and just $X$ for its image in $S_G^\h$. To say that $\tilde{X}$ is detected in filtration $\geq k$ in the Tate spectral sequence is to say that there is a lift in
\begin{center}\begin{tikzcd}
S^0\ar[r,"\tilde{X}"]\ar[dr,dashed]&S_G^\tate\\
&\widetilde{E}G^{\leq -k}\wedge S_G^\h\ar[u]
\end{tikzcd}.\end{center}
In particular, $X\colon S^0\rightarrow S_G^\h$ witnesses that $X$ is always detected in filtration $\geq 0$. Consider the diagram
\begin{center}\begin{tikzcd}
S^0\ar[r,"X"]\ar[dr,dashed]&S_G^\h\\
&S_G^\h\wedge \widetilde{E}G^{\leq -k}\ar[u]
\end{tikzcd}.\end{center}
Dualizing, this is equivalent to
\begin{center}\begin{tikzcd}
EG_+^{\leq k-1}\ar[dr,dashed,"0"]\ar[d]\\
S^0\ar[r,"X"]\ar[d]&S^\h_G\\
S(EG^{\leq k-1})\ar[ur,dashed]
\end{tikzcd}.\end{center}
Thus such a lift exists if and only if $X \in m_k(G)$, proving that if $X\in m_k(G)$ then $\tilde{X}$ is detected in filtration $\geq k$ in the Tate spectral sequence. As $m_1(G) = I(G)$ is the augmentation ideal, this proves the forward direction.

Conversely, suppose that $X$ is in the augmentation ideal and $\tilde{X}$ is detected in filtration $\geq k$ in the Tate spectral sequence. The above argument then shows that $X$ is detected in filtration $\geq k$ in the homotopy fixed point spectral sequence, and thus lies in $m_k(G)$ by (1), provided we rule out the possibility that $\tilde{X}$ could lift to some distinct $X'\in \pi_0 S_G^\h$ detected in higher filtration than $X$. To that end, it suffices to prove that the map $\pi_0 S_G^\h \to \pi_0 S_G^\tate$ is an injection when restricted to the augmentation ideal of $\pi_0 S_G^\h$.

To see this, consider the cofiber sequence
\[
EG_+\rightarrow S^\h_G \rightarrow S^\tate_G.
\]
On $\pi_0$, this induces a short exact sequence
\[
0\rightarrow\bbZ\xrightarrow{\tr}\pi_0 S^\h_G\rightarrow \pi_0 S^\tate_G\rightarrow 0.
\]
As $|\tr(1)| = |G|$, no nonzero element in the augmentation ideal of $\pi_0 S_G^\h$ can be in the image of the map $\tr\colon \bbZ \to \pi_0 S_G^\h$, implying that the map $\pi_0 S_G^\h \to \pi_0 S_G^\tate$ is an injection when restricted to the augmentation ideal as claimed.
\end{proof}

\begin{cor}
The associated graded piece $m_k(G)/m_{k+1}(G)$ is isomorphic to the $E_\infty$ page of the homotopy fixed point spectral sequence for $\pi_\star S_G^\h$ in stem $0$ and filtration $k$. If $k\geq 1$, then it is also isomorphic to the $E_\infty$ page of the Tate spectral sequence for $\pi_\star S_G^\tate$ in stem $0$ and filtration $k$.
\qed
\end{cor}

\begin{cor}
If $k\geq 1$, then $m_k(G)/m_{k+1}(G)$ is annihilated by $|G|$.
\qed
\end{cor}

\begin{rmk}
The $G$-homotopy fixed point spectral sequence for $S_G^\h$ may also be understood in more classical terms as the \textit{Atiyah--Hirzebruch spectral sequence}
\[
E_2^{s,\ast} = H^\ast(BG;\pi_{\ast+s} S)\Rightarrow \pi_s D(\Sigma^\infty_+ BG)
\]
for the stable cohomotopy of the classifying space $BG$. An $E_1$-page is determined by a choice of cell structure for $BG$, roughly corresponding to a free $\bbZ[G]$-module resolution of $\bbZ$.
\tqed
\end{rmk}

\subsection{Spoke grading}\label{ssec:spoke}

We wish to specialize the above discussion to the case where $G = C_{p^n}$ is a cyclic $p$-group, and to introduce the $C_{p^n}$-Mahowald invariant. This relies on the existence of a particularly nice model for the classifying space $EC_{p^n}$. However, a bit of care is needed to give a precise development. This subsection fixes conventions for the classifying space of a cyclic group, and in particular develops a theory of ``spoke-graded homotopy groups'' in cyclic-equivariant homotopy theory which is convenient for the study of $C_{p^n}$-Mahowald invariants.

We start with a basic observation. If $C$ is a finite cyclic group, then we can always identify
\[
EC\simeq S(V_\infty),\qquad \widetilde{E}C \simeq S^{V_{\infty}}
\]
for any infinite dimensional fixed point free $C$-representation $V_\infty$. If $V_k\subset V_\infty$ is a $k$-dimensional subrepresentation, then $S(V_k)\subset S(V_\infty)$ is an equivariant $(2k-1)$-skeleton. As a consequence we have the following.

\begin{prop}\label{prop:fpd}
Let $C$ be a finite cyclic group and $V$ be a sum of $k$ faithful complex characters of $C$. Then
\[
I_{S(V)} = m_{2k}(C) \subset A(C).
\]
\end{prop}
\begin{proof}
This follows as $I_{S(V)} = I_{EC^{\leq 2k-1}} = m_{2k}(C)$.
\end{proof}

We now elaborate on the rest of the cell structure. Let $C_m = \langle e^{2\pi i/m}\rangle \subset U(1)$ be the cyclic subgroup of order $m$, with chosen generator $g = e^{2\pi i/m}$ and tautological faithful complex character $L$. All faithful complex characters of $C_m$ are of the form $L^d$ for some integer $d$ coprime to $m$, and are thus obtained by pulling $L$ back along the $d$th power automorphism $C_m \rightarrow C_m$.

The unit sphere $S(L)$ admits a cell structure with one free $0$-cell and one free $1$-cell. This means that there is a cofiber sequence
\[
C_{m+} \rightarrow S(L)_+\rightarrow \Sigma C_{m+}.
\]
To be explicit, identify $S^1 = S(\bbC)$ with its canonical counterclockwise orientation and write $\Sigma(\bs) = S^1\wedge \bs$. Then the equivariant pointed equivalence $\Sigma C_{m+}\simeq S(L)/C_m$ is adjoint to the pointed map $S^1\rightarrow S(L)/C_m$ given by $e^{2\pi i t} \mapsto e^{2\pi it/m}$ for $0\leq t \leq 1$.

Write $SC_m$ for the unreduced suspension of $C_m$, also known as the \textit{spoke sphere} $S^{\Yright}$ \cite[Notation 1.3]{hahnsengerwilson2023odd}. There is always a canonical map $SC_m \rightarrow\Sigma C_{m+}$, and the homotopy coCartesian squares
\begin{center}\begin{tikzcd}
C_m\ar[r]\ar[d]&S(L)\ar[r]\ar[d]&\{1\}\ar[d]\\
\{0\}\ar[r]&S(L)/C_m\ar[r]&SC_m
\end{tikzcd}\end{center}
furnish a map $\Sigma C_{m+}\simeq S(L)/C_m \rightarrow SC_m$. The composite $\Sigma C_{m+}\rightarrow SC_m \rightarrow \Sigma C_{m+}$ is adjoint to the element of $\pi_1 \Sigma C_{m+}$, isomorphic to the free group on $C_{m}$, given by $[g]^{-1} \ast [1]$. Stably, this is $1-g \in \pi_0 \Sigma^\infty_+C_{m}\cong \bbZ[C_m]$.

We now fix, for the rest of this section, a sequence
\[
\ul{d} = (d_1,d_2,\ldots)
\]
of integers coprime to $m$, and set
\begin{equation}\label{eq:vk}
V_k = L^{d_1}\oplus\cdots\oplus L^{d_k}.
\end{equation}
The above cell structure on $S(L)$ twists by an automorphism of $C_m$ to a cell structure on each $S(L^{d_i})$. As
\[
S(V_k)\simeq S(L^{d_1})\ast\cdots\ast S(L^{d_k}),
\]
where $\ast$ is the topological join, these cell structures join together to give a cell structure on $S(V_k)$ and, in the limiting case, on $S(V_\infty)\simeq EC_{m}$. The associated filtration on the spectrum $\tilde{E}C_m$ by \cref{eq:egfiltration} is of the form
\begin{equation}\label{eq:cellstr}\begin{tikzcd}[column sep=1mm]
S^{-\Yright-L^{d_1}}\ar[r]\ar[d]&  S^{-L^{d_1}} \ar[d]\ar[r]& S^{-\Yright}\ar[r]\ar[d]& S^0\ar[d]\ar[r]& S^{\Yright}\ar[d]\ar[r]& S^{L^{d_1}}\ar[d]\ar[r]& S^{L^{d_1}+\Yright}\ar[d]\ar[r]& S^{L^{d_1}+L^{d_2}}\ar[d]\\
\Sigma^{-3}C_{m+}&\Sigma^{-2}C_{m+}&\Sigma^{-1}C_{m+}&C_{m+}&\Sigma C_{m+}&\Sigma^2 C_{m+}&\Sigma^3 C_{m+}&\Sigma^4 C_{m+},
\end{tikzcd}\end{equation}
where for example $S^{-\Yright-L^{d_1}} = D(S^{L^{d_1}}\otimes SC_m)$. The associated Tate spectral sequence thus has an $E_1$-page of the form
\begin{equation}\label{eq:cmtss}
E_1^{\alpha,f} = \pi_{|\alpha|} S^{-f} \Rightarrow \pi_\alpha S_{C_m}^\tate.
\end{equation}
The above discussion shows that the differential on the $E_1$-page is determined by the Tate complex splicing together the free $\bbZ[C_m]$-module resolution
\begin{center}\begin{tikzcd}[column sep=0.75cm]
\bbZ&\ar[l,"\epsilon"']\bbZ{[C_{m}]}&\bbZ{[C_{m}]}\ar[l,"1-g^{d_1}"']&&\bbZ{[C_{m}]}\ar[ll,"g^0+\cdots+g^{m-1}"']&\bbZ{[C_{m}]}\ar[l,"1-g^{d_2}"']&&\bbZ{[C_{m}]}\ar[ll,"g^0+\cdots+g^{m-1}"']&\cdots\ar[l]
\end{tikzcd}\end{center}
and its dual. It is, in particular, sensitive to the choice of $\ul{d}$, with different choices modifying the generator of $E_1^{-f,f} = \pi_f S^{f}\cong \bbZ$ by an integer coprime to $m$. 

\begin{rmk}\label{rmk:sequence}
As we have defined $C_m$ to be equipped with a preferred complex character $L$, it is natural to take $d_i = 1$ for all $i$. However this choice turns out to be completely unusable for our purposes. Instead, we shall take $\ul{d} = (a_1,-a_1,a_2,-a_2,\ldots)$ where $(a_1,a_2,\ldots)$ is the list of positive integers coprime to $m$ in increasing order; the essential properties of this sequence are given in \cref{lem:dseq}. However, such choices do not matter in this section, so for the time being the reader may take any choice they prefer.
\tqed
\end{rmk}

We now introduce spoke-graded homotopy groups.

\begin{defn}\label{def:spokegrading}
Given $\gamma \in RO(C_m)$ and $w\in\bbZ$, define
\[
\pi_{\gamma,w}^{C_{m}}E = E^{-\gamma}_{C_m}(\widetilde{E}C_{m}^{\leq w-1}) = \begin{cases}
\pi_{\gamma+L^{d_1}+\cdots+L^{d_k}}^{C_m}E&w = 2k \geq 0,\\[5pt]
\pi_{\gamma+L^{d_1}+\cdots+L^{d_k}+\Yright}^{C_m}E&w=2k+1\geq 0,\\[5pt]
\pi_{\gamma-(L^{d_1}+\cdots+L^{d_k})}^{C_m}E&w=2k\leq 0,\\[5pt]
\pi_{\gamma-(L^{d_1}+\cdots+L^{d_k}-\Yright)}^{C_m}E&w=2k-1\leq 0
\end{cases}
\]
for a $C_m$-spectrum $E$.
\tqed
\end{defn}

These are the ``bigraded'' or ``synthetic'' homotopy groups associated to the filtered spectrum $\widetilde{E}C_m \otimes E$. In particular, they control and are controlled by information about the Tate spectral sequence for $E$.

\begin{ex}
When $m = 2$, we have $S^{\Yright} = SC_2\simeq S^\sigma$ where $\sigma$ is the real sign representation of $C_2$, and accordingly
\[
\pi_{s,w}^{C_2} E = \pi_{s+w\sigma}^{C_2}E
\]
encode the bigraded homotopy groups of a $C_2$-spectrum $E$.
\tqed
\end{ex}

In general, groups $\pi_{\star,\ast}^{C_m}E$ are constructed to behave similarly to the bigraded homotopy groups of a $C_2$-spectrum. Our goal for the rest of this section is to construct $C_m$-equivariant analogues of restriction, transfer, and Euler class operations which are fundamental to computations in $C_2$-equivariant stable homotopy theory. We start with the following fundamental long exact sequence.

\begin{defn}\label{def:tra}
Given a $C_m$-spectrum $E$, we write
\begin{equation}\label{eq:tra}\begin{tikzcd}
\cdots\ar[r]&\pi_{|\gamma|+w+1}^e E \ar[r,"\trp"]&\pi_{\gamma,w+1}^{C_m}E\ar[r,"a^{1/2}"]&\pi_{\gamma,w}^{C_m}E\ar[r,"\resp"]&\pi_{|\gamma|+w}^eE\ar[r]&\cdots
\end{tikzcd}\end{equation}
for the long exact sequence obtained by applying $E_{C_m}^{-\gamma}(\bs)$ to the cofiber sequence
\begin{center}\begin{tikzcd}
\cdots&\ar[l]\Sigma^w C_{m+}&\widetilde{E}C_m^{\leq w}\ar[l]&\widetilde{E}C_m^{\leq w-1}\ar[l]&\Sigma^{w-1}C_{m+}\ar[l]&\ar[l]\cdots
\end{tikzcd}\end{center}
associated to our chosen filtration of $\widetilde{E}C_m$, as in \cref{eq:cellstr}.
\tqed
\end{defn}

\begin{ex}\label{ex:c2tra}
Taking $m=2$ and $\gamma = s$, \cref{eq:tra} is equivalent to the standard long exact sequence
\begin{center}\begin{tikzcd}
\cdots\ar[r]&\pi_{s+w+1}^eE\ar[r,"\tr_e^{C_2}"]&\pi_{s,(w+1)\sigma}^{C_2}E\ar[r,"a_\sigma"]&\pi_{s,w\sigma}^{C_2}E\ar[r,"\res^{C_2}_e"]&\pi_{s+w}^eE\ar[r]&\cdots
\end{tikzcd}\end{center}
relating nonequivariant and $C_2$-equivariant homotopy groups.
\tqed
\end{ex}

The notation $a^{1/2}\colon \pi_{\star,\ast+1}^{C_m}E \to \pi_{\star,\ast}^{C_m}E$ for the transition maps in $\pi_{\star,\ast}^{C_m}E$ is chosen to be suggestive of the following.

\begin{prop}\label{prop:halfeuler}
We may identify the composite
\[
a^{1/2}\circ a^{1/2} = a_{L^{d_k}} \colon \pi_{\gamma,2k}^{C_m}E \to \pi_{\gamma,2(k-1)}^{C_m}E
\]
as multiplication by the Euler class $a_{L^{d_k}}$ of $L^{d_k}$.
\end{prop}
\begin{proof}
By definition $a^{1/2}\circ a^{1/2}\colon \pi_{\gamma,2k}^{C_m}E \to \pi_{\gamma,2(k-1)}^{C_m}E$ is given by restriction along the inclusion $\widetilde{E}C_m^{\leq 2k-2} \subset \widetilde{E}C_m^{\leq 2k}$, and our choice of filtration on $\widetilde{E}C_m$ is such that this inclusion is exactly the inclusion
\[
a_{L^{d_k}}\colon S^{L^{d_1}\oplus\cdots\oplus L^{d_{k-1}}} \to S^{L^{d_1}\oplus\cdots\oplus L^{d_k}}
\]
if $k \geq 0$ and its dual if $k \leq 0$.
\end{proof}

Moreover, we have the following.

\begin{prop}\label{prop:alocal}
Define the localization
\[
a^{-1/2}\pi_{\gamma,\ast}^{C_{m}} E = \colim(\cdots \rightarrow  \pi_{\gamma,k+1}^{C_{m}}E\xrightarrow{a^{1/2}} \pi_{\gamma,k}^{C_{m}} E \xrightarrow{a^{1/2}}\pi_{\gamma,k-1}^{C_{m}}E\rightarrow\cdots).
\]
Then
\[
a^{-1/2}\pi_{\gamma,\ast}^{C_{m}}E = \pi_{\gamma}^{C_m}\left(E\otimes \widetilde{E}C_m\right).
\]
\end{prop}
\begin{proof}
This is immediate from the definitions as
\[
E_{C_m}^{-\gamma}(\widetilde{E}C_m^{\leq w-1}) \cong \pi_\gamma^{C_m}(E \otimes \widetilde{E}C_m^{\leq 1-w})
\]
by construction.
\end{proof}

\begin{warning}
Let $E$ be a $C_m$-ring spectrum. We have not analyzed the extent to which $\widetilde{E}C_m^{\leq\bullet}$ may be a multiplicative filtration for $m > 2$, and accordingly we do \emph{not} claim that $\pi_{\star,\ast}^{C_m}E$ is a graded ring. In particular, it is \emph{not} the case, with our formalism, that $a^{1/2}$ is multiplication by the class $a^{1/2}\cdot 1 \in \pi_{0,-1}S_{C_m}$, i.e.\ one should think of $a^{1/2}$ not as an element living in a homotopy group but as an operation on homotopy groups (whose definition implicitly depends on the sequence $\ul{d} = (d_1,d_2,\ldots)$ and degree that one is operating on).
\tqed
\end{warning}

Our next goal is to construct and analyze certain restriction and transfer maps on the groups $\pi_{\star,\ast}^{C_m}E$.  If $n\mid m$, then $C_n$ acts freely on the contractible space $EC_m$, and so we can say that $\res^{C_m}_{C_n}EC_m = EC_n$. However, our chosen cell structure on $EC_m$ does \textit{not} restrict to our chosen cell structure on $EC_n$, and accordingly $\res^{C_m}_{C_n}\widetilde{E}C_m^{\leq \bullet}$ is not equivalent, as a filtered object, to $\widetilde{E}C_n^{\leq\bullet}$. Cellular approximation applied to the identity map on $EC_n$ implies that there are comparisons between these cell structures, but to construct well-defined restrictions and transfers and analyze their properties we must be rather more explicit.

We start with the following. 

\begin{construction}
Fix $n\mid m$, and let $q = g^{m/n}$ be the cyclic generator of $C_n \subset C_m$. Then there is a $C_n$-equivariant homotopy retraction
\begin{center}\begin{tikzcd}
C_n\ar[r,"\subset"]\ar[d]&C_m\ar[r,"r"]\ar[d]&C_n\ar[d]\\
S(L)\ar[r,equals]&S(L)\ar[r,equals]&S(L)
\end{tikzcd},\end{center}
where 
\begin{equation}\label{eq:retract}
r(g^k) = q^{\ceil{nk/m}}.
\end{equation}
The left square commutes on the nose and the right square by the $C_n$-equivariant homotopy
\[
I \times C_m\rightarrow S(L),\qquad (t,g^k) \mapsto e^{\frac{2\pi i}{m}\left((1-t)k+t\frac{m}{n}\ceil{\frac{nk}{m}}\right)}.
\]
Passing to unreduced suspensions gives a $C_n$-equivariant retraction of the form 
\begin{center}\begin{tikzcd}
S^0\ar[r,equals]\ar[d]&S^0\ar[r,equals]\ar[d]&S^0\ar[d]\\
SC_{n}\ar[r,"i"]\ar[d]&SC_{m}\ar[r,"r"]\ar[d]&SC_{n}\ar[d]\\
S^{L}\ar[r,equals]&S^{L}\ar[r,equals]&S^{L}
\end{tikzcd}.\end{center}
In other words, this constructs specific $C_n$-equivariant cellular approximations 
\[
i\colon S^L\rightarrow \res^{C_m}_{C_n}S^L,\qquad r\colon \res^{C_m}_{C_n}S^L \to S^L
\]
to the identity on $S^L$ satisfying $r\circ i = 1$ as cellular maps. Here, we write $\res^{C_m}_{C_n}S^L$ for $S^L$ equipped with its $C_n$-equivariant cell structure obtained by restricting its preferred $C_m$-equivariant cell structure discussed above. 
\tqed
\end{construction}

\begin{rmk}
Although we do not strictly need it, we point out that the above maps are in fact dual to each other in the following sense. There is a $C_m$-equivariant equivalence
\[
F(SC_m,S^L)\simeq SC_m
\]
realizing $SC_m$ as self-dual with respect to $S^L$. Indeed, with the embedding $C_m \subset S(L)$ fixed, the Spanier--Whitehead dual $F(SC_m,S^L)$ can be naturally identified as
\[
F(SC_m,S^L)\simeq S(S(L)\setminus C_m),
\]
and this may be identified with $C_m$ via the equivalence $C_m\simeq S(L)\setminus C_m$ sending $e\in C_m$ to a point just clockwise of $1\in S(L)$. With this choice, the inclusion $C_n\subset C_m$ induces a map $S(L)\setminus C_m\rightarrow S(L)\setminus C_n$ which is equivalent to $r\colon C_m\rightarrow C_n$, and as a consequence the diagram
\begin{center}\begin{tikzcd}
\res^{C_{m}}_{C_{m}} F_{C_{m}}(SC_{m},S^{L})\ar[d,"\simeq"]\ar[r,"="]&F_{C_{n}}(SC_{m},S^{L})\ar[r,"i^\ast"]\ar[d,"\simeq"]&F_{C_{n}}(SC_{n},S^{L})\ar[d,"\simeq"]\\
\res^{C_{m}}_{C_{n}}SC_{m}\ar[r,"="]&SC_{m}\ar[r,"r"]&SC_{n}
\end{tikzcd}\end{center}
commutes.
\tqed
\end{rmk}

By pulling the above discussion back along an automorphism of $C_m$, we are able to replace $L$ with $L^d$ for any integer $d$ coprime to $m$. In particular we obtain $C_n$-equivariant cellular retractions 
\[
S^{L^d}\rightarrow \res^{C_m}_{C_n}S^{L^d}\rightarrow S^{L^d}.
\]
This dualizes to a $C_n$-equivariant filtered retraction
\[
S^{-L^d}\rightarrow\res^{C_m}_{C_n}S^{-L^d}\rightarrow S^{-L^d},
\]
where $S^{-L^d}$ is the filtered object $S^{-L^d}\rightarrow D(SC_m)\rightarrow S^0$. The first sequence of filtered objects is the $L^d$-fold suspension of the second, where for example 
\[
\Sigma^{L^d}\left(S^{-L^d}\rightarrow D(SC_{m})\rightarrow S^0\right) \simeq \left(S^0\rightarrow SC_{m}\rightarrow S^{L^d}\right).
\]
Smashing these retractions together for various $d$, we obtain a $C_n$-equivariant filtered retraction
\[
\widetilde{E}C_{n}^{\leq \bullet}\rightarrow \res^{C_m}_{C_n}\widetilde{E}C_{m}^{\leq \bullet}\rightarrow \widetilde{E}C_{n}^{\leq \bullet}.
\]
In particular, by adjunction we now have maps
\begin{equation}\label{eq:filtapprox}
\Ind_{C_{n}}^{C_{m}}\widetilde{E}C_{n}^{\leq \bullet}\rightarrow\widetilde{E}C_{m}^{\leq \bullet},\qquad \widetilde{E}C_{m}^{\leq \bullet}\rightarrow\Ind_{C_{n}}^{C_{m}}\widetilde{E}C_{n}^{\leq \bullet}
\end{equation}
of filtered $C_{m}$-spectra. 

\begin{defn}\label{def:restrspoke}
We write
\begin{equation}\label{eq:restrspoke}
\tr_{C_n}^{C_m} : \pi_{\gamma,w}^{C_{n}} E \rightleftarrows \pi_{\gamma,w}^{C_{m}}E : \res_{C_n}^{C_m}
\end{equation}
for the \emph{restriction} and \emph{transfer} maps in spoke-graded homotopy groups induced by the maps of \cref{eq:filtapprox}.
\tqed
\end{defn}

\begin{ex}\label{ex:restrevenspoke}
If $w = 2k$ is even, then the restriction and transfer maps of \cref{def:restrspoke} are equivalent to the usual restriction and transfer maps
\[
\tr_{C_n}^{C_m} : \pi_{\gamma+V_k}^{C_n}E \to \pi_{\gamma+V_k}^{C_m}E : \res_{C_n}^{C_m}
\]
in the $RO(C_m)$-graded homotopy groups of $E$. Here, $V_k = -V_{-k}$ for $k < 0$.
\tqed
\end{ex}

We record the following for easy reference.

\begin{prop}\label{prop:restreuler}
The restriction and transfer maps in spoke-graded homotopy groups commute with $a^{1/2}$.
\end{prop}
\begin{proof}
This is immediate from their definition in terms of a comparison between the filtered objects $\widetilde{E}C_m^{\leq\bullet}$ and $\widetilde{E}C_n^{\leq\bullet}$ used to define $a^{1/2}$.
\end{proof}

So far we have not used any details of the construction of $\res^{C_m}_{C_n}$ and $\tr^{C_m}_{C_n}$; \cref{prop:restreuler} would hold just as well had we defined these maps via an abstract application of cellular approximation to the identity map on $EC_n \simeq \res^{C_m}_{C_n}EC_m$. These details come into play in the following.

\begin{prop}\label{prop:restr}
Let $w = 2k-1$ be an odd integer and set $d = d_k$. Then
\begin{enumerate}
\item $\resp = \resp\circ \tr_{C_n}^{C_m}\colon \pi_{\gamma,w}^{C_{n}}E\rightarrow \pi_{|\gamma|+w}^e E$.
\item $(1+g^d+\cdots+g^{dm/n-1})\circ \resp = \resp\circ \res_{C_n}^{C_m}\colon \pi_{\gamma,w}^{C_{m}}E\rightarrow \pi_{|\gamma|+w}^e E$.
\end{enumerate}
\end{prop}
\begin{proof}
After replacing $E$ by $\Sigma^{-\gamma}E$, we may as well suppose $\gamma = 0$. We treat the case $w = 1$, as the other cases may be derived from this by restricting along an automorphism of $C_m$.

All of the homotopy groups in question are represented by certain $C_m$-spectra, and the maps between them obtained by restriction along maps between these representing objects. For example,
\[
\resp\colon \pi_{0,1}^{C_{m}} E \rightarrow \pi_1^e E
\]
is obtained by restriction along the first map in the cofiber sequence
\[
\Sigma C_{m+}\rightarrow SC_{m}\rightarrow S^L\rightarrow \Sigma^2 C_{m+}.
\]

Translating to these terms, to prove $\resp = \resp\circ \tr_{C_n}^{C_m}$ we must verify that the diagram
\begin{equation}\label{eq:rtr}\begin{tikzcd}
\Sigma C_{m+}\ar[d,"="]\ar[r]&SC_{m}\ar[d,"r"]\\
\Ind_{C_{n}}^{C_{m}}\Sigma C_{n+}\ar[r]&\Ind_{C_{n}}^{C_{m}}SC_{n}
\end{tikzcd}\end{equation}
commutes, where the right vertical map is adjoint to the unreduced suspension of the function of \cref{eq:retract}. There, we observed that the top square in
\begin{center}\begin{tikzcd}
C_{m}\ar[r,"r"]\ar[d]&C_{n}\ar[d]\\
S(L)\ar[d]\ar[r,"="]&S(L)\ar[d]\\
\Sigma C_{m+}\ar[r,"h",dashed]&\Sigma C_{n+}
\end{tikzcd}\end{center}
commutes up to a specific homotopy, and this induces a map $h$ on cofibers. The map $h$ is determined by its effect on $\pi_1$, which stably is seen to be the map
\[
\bbZ[C_{m}]\rightarrow\bbZ[C_{n}],\qquad g^k \mapsto \begin{cases}q^{kn/m}&\frac{m}{n}\mid k,\\ 0&\text{otherwise}\end{cases}
\]
adjoint to the identity $\bbZ[C_m]\rightarrow\bbZ[C_m]$. Taking unreduced suspensions yields \cref{eq:rtr}.

Next, to prove $(1+g+\cdots+g^{m/n-1})\circ \resp =  \resp \circ \res^{C_m}_{C_n}$ we must verify that the diagram
\begin{center}\begin{tikzcd}
\Ind_{C_n}^{C_m}\Sigma C_{n+}\ar[r]\ar[d,"="']&\Ind_{C_n}^{C_m}SC_n\ar[dd,"i"]\\
\Sigma C_{m+}\ar[d,"1+g+\cdots+g^{m/n-1}"']\\
\Sigma C_{m+}\ar[r]&SC_m
\end{tikzcd}\end{center}
commutes, where the right vertical map is adjoint to the map $SC_n\rightarrow SC_m$ induced by the inclusion $C_n\subset C_m$. Both composites are determined by their corresponding element of $\pi_1^e SC_m$ which, by the cofiber sequence $SC_m\rightarrow \Sigma C_{m+}\rightarrow S^1$, can be stably identified as the augmentation ideal $I(\bbZ[C_m])$. Our discussion of cell structures shows that the bottom horizontal map classifies $1-g$ and the top map is induced from that classifying $1 - q = 1 - g^{m/n}$, so the proposition follows from the identity $(1-g)(1+g+\cdots+g^{m/n-1}) = 1-g^{m/n}$.
\end{proof}

\subsection{The \texorpdfstring{$C_{p^n}$}{C\_pn}-Mahowald invariant}\label{ssec:cpnmi}

We now turn our attention to the $C_{p^n}$-Mahowald invariant. We begin with the following observation. 

\begin{defn}\label{def:genuinefp}
Write
\[
\Phi\colon \Sp^{C_{p^n}}\rightarrow \Sp^{C_{p^{n-1}}},\qquad E \mapsto \Phi^{C_p} E
\]
for the functor sending a $C_{p^n}$-spectrum $E$ to its $C_p$-geometric fixed points equipped with their residual genuine $C_{p^n}/C_p\cong C_{p^{n-1}}$-action.
\tqed
\end{defn}

The following is standard.

\begin{lemma}
If $E$ is a $C_{p^n}$-spectrum, then $(\widetilde{E}C_{p^n}\wedge E)^{C_{p^n}}\simeq (\Phi E)^{C_{p^{n-1}}}$.
\qed
\end{lemma}

In particular, $\pi_\alpha^{C_{p^n}} \widetilde{E}C_{p^n} \cong \pi_{\alpha^{C_p}}S_{C_{p^{n-1}}}$. The filtration on $\widetilde{E}C_{p^n}$ discussed in detail in \cref{ssec:spoke} thus gives a spectral sequence of the form
\begin{equation}\label{eq:tatess}
E_1^{\alpha,f} = \pi_{|\alpha|} S^{-f} \Rightarrow \pi_{\alpha^{C_p}} S_{C_{p^{n-1}}}.
\end{equation}
This spectral sequence is isomorphic to the Tate spectral sequence, and the Segal conjecture implies that it converges to the $p$-completion $\bbZ_p\otimes \pi_\star S_{C_{p^{n-1}}}$, although this convergence is not strictly needed for what follows. We can now give the following.

\begin{defn}\label{def:cpnmifull}
The \textit{$C_{p^n}$-Mahowald invariant} is the relation
\[
M_{C_{p^n}}\colon \pi_\star S_{C_{p^{n-1}}}\rightharpoonup \pi_\ast S
\]
defined by
\[
y \in M_{C_{p^n}}(x)\qquad\iff\qquad y\text{ detects } x \text{ in the Tate spectral sequence.}\tag*{$\triangleleft$}
\]
\end{defn}

\begin{rmk}\label{rmk:gammak}
Recall the filtration on $A(C_{p^{n-1}})$ introduced in \cref{def:gammafilt}:
\[
\Gamma_k(C_{p^{n-1}}) = \{X\in A(C_{p^{n-1}}) : |M_{C_{p^n}}(X)| \geq k \}.
\]
By definition, $\Gamma_k(C_{p^{n-1}})$ consists of all elements in $A(C_{p^{n-1}})$ detected in filtration $\geq k$ in the Tate spectral sequence. By \cref{prop:mahowaldfiltss}, the fixed point homomorphism
\[
\Phi\colon A(C_{p^n})\rightarrow A(C_{p^{n-1}}),\qquad \Phi(X) = X^{C_p}
\]
induces isomorphisms $m_k(C_{p^n}) \cong \Gamma_k(C_{p^{n-1}})$ for $k\geq 1$ for which the diagram 
\begin{center}\begin{tikzcd}
m_k(C_{p^n})\ar[rr,"\Phi","\cong"']\ar[dr,"\tilde{\phi}"',tail]&&\Gamma_k(C_{p^{n-1}})\ar[dl,"\phi",tail]\\
&\Hom(\{1,\ldots,n\},\bbZ)
\end{tikzcd}\end{center}
commutes, where
\[
(\tilde{\phi}X)(i) = |X^{C_{p^{i}}}|,\qquad (\phi X)(i) = |X^{C_{p^{i-1}}}|.
\]
So we view the computation of $m_k(C_{p^n})$ and $\Gamma_k(C_{p^{n-1}})$ as essentially interchangeable.
\tqed
\end{rmk}

Let $E$ be a $C_{p^n}$-spectrum, and consider the homotopy groups $\pi_{\gamma,w}^{C_{p^n}}E$ developed in \cref{ssec:spoke}. In this $C_{p^n}$-equivariant context, it is conceptually convenient to restrict $\gamma$ to $RO(C_{p^{n-1}})$, viewed as a subring of $RO(C_{p^n})$ by restriction along the quotient map $C_{p^n}\rightarrow C_{p^{n-1}}$, so that $\pi_{\star,\ast}^{C_{p^n}}E$ is $RO(C_{p^{n-1}})\times\bbZ$-graded. For example, we can then identify the localization
\[
a^{-1/2}\pi_{\gamma,\ast}^{C_{p^n}} E \cong \pi_\gamma^{C_{p^{n-1}}}\Phi E.
\]
These homotopy groups participate in the following $C_{p^n}$-equivariant analogue of Bruner and Greenlees' $C_2$-equivariant reformulation of the classical Mahowald invariant \cite{brunergreenlees1995bredon}.

\begin{prop}\label{prop:bg}
Fix $x\in \pi_\gamma S_{C_{p^{n-1}}}$, and write $x = \Phi(Y)$ for some $Y \in \pi_{\gamma,k} S_{C_{p^n}}$ with $k$ as large as possible. Then $\resp(Y) \in M_{C_{p^n}}(x)$, and all $C_{p^n}$-Mahowald invariants of $x$ arise this way.
\end{prop}
\begin{proof}
The Tate spectral sequence for $S_G^\tate$ is isomorphic to the spectral sequence associated to the filtration on $\widetilde{E}G$. Hence $x\in \pi_\gamma S_{C_{p^{n-1}}}$ is detected by $y \in \pi_{|\gamma|+k}S$ if and only if $k$ is maximal for which $y$ lifts through $\resp$ to a class $Y \in \pi_{\gamma,k} S_{C_{p^n}}$ satisfying $\Phi(Y) = x$. The proposition is just a restatement of this.
\end{proof}

We separate out the following immediate consequence for convenient reference.

\begin{cor}\label{cor:bg}
$\Gamma_k(C_{p^{n-1}}) = \im\left(\Phi\colon \pi_{0,k}S_{C_{p^n}}\rightarrow \pi_0 S_{C_{p^{n-1}}}\right)\subset A(C_{p^{n-1}})$.
\qed
\end{cor}

We can use this to analyze the $C_{p^n}$-Mahowald invariants of restrictions and transfers.

\begin{prop}\label{prop:roottr}
Fix $1\leq i < n$ and $\gamma \in RO(C_{p^{i-1}})$. Let $x\in \pi_\gamma S_{C_{p^{i-1}}}$. Then
\[
|M_{C_{p^n}}(\tr_{i-1}^{n-1}(x))| \geq |M_{C_{p^i}}(x)|,
\]
and if equality holds then for any $y\in M_{C_{p^i}}(x)$ we have
\begin{enumerate}
\item If $|M_{C_{p^i}}(x)|$ is even, then $\gamma$ is oriented and $p^{n-i}y \in M_{C_{p^n}}(\tr_{i-1}^{n-1}(x))$;
\item If $|M_{C_{p^i}}(x)|$ is odd, then $y\in M_{C_{p^n}}(\tr_{i-1}^{n-1} (x))$.
\end{enumerate}
\end{prop}
\begin{proof}
Set $w = |M_{C_{p^i}}(x)|$. The hypothesis provides $Y\in \pi_{\gamma,w}S_{C_{p^i}}$ satisfying
\[
\Phi(Y) = x,\qquad \resp(Y) = y.
\]
The element $\tr_i^n(Y) \in \pi_{\gamma,w}S_{C_{p^n}}$ then satisfies
\[
\Phi(\tr_i^n(Y)) = \tr_{i-1}^{n-1}(\Phi(Y)) = \tr_{i-1}^{n-1}(x),
\]
proving $|M_{C_{p^n}}(\tr_{i-1}^{n-1}(x))| \geq |M_{C_{p^i}}(x)|$. If equality holds, then $\resp(\tr_i^n(Y)) \in M_{C_{p^n}}(\tr_{i-1}^{n-1}(x))$, so we must identify this element.

If $w = 2k$ is even, then $\resp$ and $\tr$ are the usual restriction and transfer, and the double coset formula gives 
\[
\resp(\tr_i^n(Y)) = \sum_{h\in C_{p^n}/C_{p^i}}(h\cdot y) = (e+g+\cdots+g^{p^{n-i}-1})\cdot y,
\]
where $C_{p^n}$ acts on $\pi_{\gamma+V_k}^e S_{C_{p^n}}\cong \pi_{|\gamma|+2k}S$ trivially if $\gamma$ is oriented and otherwise with a twist by the sign representation. In the former case we obtain $\resp(\tr_i^n(Y)) = p^{n-i}y$ as claimed. In the latter case we obtain $\resp(\tr_i^n(Y)) = 0$, contradicting the assumption that $w=2k$. Finally, if $w$ is odd then
\[
\resp(\tr_i^n(Y)) = \resp(Y) = y
\]
by \cref{prop:restr}.
\end{proof}

Essentially the same proof shows the following.

\begin{prop}\label{prop:rootres}
Fix $1\leq i < n$ and $\gamma \in RO(C_{p^{i-1}})$. Let $x\in \pi_\gamma S_{C_{p^{n-1}}}$. Then
\[
|M_{C_{p^i}}(\res_{i-1}^{n-1}(x))| \geq |M_{C_{p^n}}(x)|,
\]
and if equality holds then for any $y\in M_{C_{p^n}}(x)$ we have
\begin{enumerate}
\item If $|M_{C_{p^n}}(x)|$ is odd, then $\gamma$ is oriented and $p^{n-i}y \in M_{C_{p^i}}(\res_{i-1}^{n-1}(x))$;
\item If $|M_{C_{p^n}}(x)|$ is even, then $y\in M_{C_{p^i}}(\res_{i-1}^{n-1} (x))$.
\qed
\end{enumerate}
\end{prop}

\section{Preliminaries on equivariant \texorpdfstring{$K$}{K}-theory}\label{sec:ktheory}

Following \cref{cor:reducetoktheory}, we are interested in the groups
\[
H^0(\psi^\ell;\pi_{0,\ast}KO_{C_{p^n}}),
\]
where $\ell$ is a primitive root modulo $p^2$ (say $\ell = 3$ if $p=2$). Computing these groups is a bit delicate: the action of $\psi^\ell$ is difficult to control in general, making a naive brute-force approach infeasible. In this short section, we separate out some key observations about the action of $\psi^\ell$ that will allow us to control these fixed points.

\subsection{Equivariant \texorpdfstring{$K$}{K}-theory}\label{ssec:equivktheory}

We assume familiarity with the basic construction and properties of equivariant $K$-theory; an excellent reference for our purposes is \cite[IV.1]{atiyahtall1969group}. We just fix some notation.

\begin{defn}\label{def:bottandeuler}
Given a complex $G$-representation $V$, we write
\[
\beta_V \in \pi_V KU_G
\]
for its Bott class, and
\[
e_V = a_V \beta_V \in \pi_0 KU_G \cong RU(G)
\]
for its $K$-theory Euler class.
\tqed
\end{defn}

An explicit formula for the $K$-theory Euler class is given by
\[
e_V = \Lambda^\ast V = \sum_{i\geq 0}(-1)^i \Lambda^i V,
\]
where $\Lambda^i$ is the $i$th exterior power: for example,
\[
e_L = 1 - L
\]
for any complex character $L$. These classes satisfy
\[
\beta_{U\oplus V} = \beta_U \beta_V,\qquad e_{U\oplus V} = e_U e_V.
\]
One can descend from complex $K$-theory to real $K$-theory via the homotopy fixed point spectral sequence
\[
E_2 = H^\ast(\{\psi^{\pm 1}\},\pi_\star KU_G)\Rightarrow \pi_{\star-\ast}KO_G.
\]
See for example \cite[Section 6]{balderrama2023equivalences} for a general discussion with explicit examples. In particular, if $V$ is a quaternionic representation (such as $\bbH\otimes_\bbC V = V + \ol{V}$ for a complex representation $V$) then the Bott class $\beta_V$ is fixed by $\psi^{-1}$, and descends to $KO_G$ if moreover $V$ has real dimension divisible by $8$.

\subsection{Adams operations and complex characters}\label{ssec:circlegroup}

We recall some of the structure of Adams operations in equivariant $K$-theory. Consider the circle group $T = U(1) \subset \bbC$. Let $L$ denote the tautological complex character of $T$, so that
\[
RU(T) = \bbZ[L^{\pm 1}].
\]
The Adams operation $\psi^\ell$ acts on this by
\[
\psi^\ell(L) = L^\ell
\]
as usual.

\begin{prop}\label{prop:adamsopschar}
The Adams operations act on the Bott class $\beta_L \in \pi_L KU_{T}$ by
\[
\psi^{-1}(\beta_L) = -L^{-1}\beta_L,\qquad \psi^\ell(\beta_L) = (1+L+\cdots+L^{\ell-1})\beta_L
\]
for $\ell \geq 1$.
\end{prop}
\begin{proof}
Restriction along the inclusion of poles $a_L\colon S^0\rightarrow S^L$ gives an injection
\[
a_L\colon \pi_L KU_{T}\cong RU(T)\{\beta_L\}\rightarrow \pi_0 KU_{T}\cong RU(T),\qquad \beta_L\mapsto  e_L =  1-L.
\]
As $\psi^\ell(L) = L^\ell$ for all $\ell\in\bbZ$, it follows that
\[
\psi^\ell(\beta_L) = \frac{1-L^\ell}{1-L}\beta_L,
\]
which gives the stated formula. 
\end{proof}

We make some further observations. In general, the $d$th power map on $\bbC$ restricts to a $T$-equivariant map
\[
S(L)\rightarrow S(L^d)
\]
on unit spheres. Taking unreduced suspensions, this defines a map
\begin{equation}\label{eq:powermap}
\psi_d\colon S^L\rightarrow S^{L^d},
\end{equation}
or equivalently an element 
\[
\psi_d \in \pi_{L-L^d}S_T,
\]
satisfying
\[
\Phi^{C_n}(\psi_d) = \begin{cases}d&n=1,\\ 0&1\neq n \mid d,\\1&n\nmid d.\end{cases}
\]

We make some remarks.

\begin{rmk}\label{rmk:virtual}
If $\gcd(m,d) = 1$, then the restriction of $\psi_d$ to the cyclic subgroup $C_m\subset T$ is the identity on all nonzero fixed points. Using this, one can show that if $\alpha$ is a virtual representation satisfying $|\alpha^{C_n}| = 0$ for all $1\neq n \mid m$ and $|\alpha| = k \neq 0$, then the image of the degree function
\[
\deg\colon \pi_\alpha S_{C_m}\rightarrow \Hom(\Sub(C_m)\setminus\{e\},\bbZ)
\]
depends only on $k$. In particular, this extends \cref{thm:fixeddegree} to the case where $V$ is replaced by a difference of fixed point free representations.
\tqed
\end{rmk}

\begin{rmk}
The restriction of $\psi_{p^t}$ to a cyclic $p$-subgroup $C_{p^n}\subset T$ is an example of what in the context of \cite{bhattacharyaguillouli2022rmotivic, behrenscarlisle2024periodic} one might call a \textit{$v_{0,-1,\ldots,-1,\infty,\ldots,\infty}$-self map}. Here, there are $\min(t,n)$ copies of $-1$ and $n-\min(t,n)$ copies of $\infty$. For example, $\psi_{p^n} = \tr_e^{C_{p^n}}(1)$.
\tqed
\end{rmk}

We do not need the following, but record it for completeness as it can, in principle, be used to extend our description of $H^\ast(\psi^\ell;\pi_{0,\ast}KO_{C_{p^n}})$ to arbitrary fixed point free degrees.

\begin{prop}
The Hurewicz image of $\psi_d$ in $\pi_{L-L^d}KU_T$ is given by
\[
\psi_d = \frac{\psi^d(\beta_L)}{\beta_{L^d}}.
\]
\end{prop}
\begin{proof}
The composite $\psi_d\circ a_L\colon S^0\rightarrow S^L\rightarrow S^{L^d}$ is just the inclusion of poles for $S^{L^d}$, i.e.\ $\psi_d \cdot a_L = a_{L^d}$. It follows that
\[
1-L^d = a_{L^d}\beta_{L^d} = \psi_d a_L \beta_{L^d} = \psi_d a_L \beta_L \beta_L^{-1}\beta_{L^d} = \psi_d (1-L) \beta_L^{-1}\beta_{L^d},
\]
and thus
\[
\psi_d = \frac{1-L^d}{1-L} \cdot \frac{\beta_{L}}{\beta_{L^d}} = \frac{\psi^d(\beta_L)}{\beta_{L^d}}
\]
by \cref{prop:adamsopschar}.
\end{proof}

We end by fixing some more notation regarding the representation theory of cyclic groups. Fix a positive integer $m$ and let
\[
C_m = \langle e^{2\pi i/m}\rangle \subset T
\]
be the cyclic subgroup of order $m$, with generator $g = e^{2\pi i/m}$. As discussed in \cref{rmk:segalritter}, the permutation representation homomorphism
\[
\bbC[\bs]\colon A(C_m)\rightarrow RU(C_m)
\]
identifies the Burnside ring $A(C_m)$ as the rational representation ring $R\bbQ(C_m)$, realized as the subring of $RU(C_m)$ on those representations with rational characters, or equivalently, fixed by the Adams operation $\psi^\ell$ for all $\ell$ coprime to $m$. 

\begin{defn}\label{def:rat}
Write
\begin{equation}
\rat_m = \sum\{L^k : 1\leq k \leq m,\, \gcd(m,k) = 1\}
\end{equation}
for the fixed point free irreducible rational representation of $C_m$. This has dimension equal to the value of the Euler totient function $\varphi(m)$.
\tqed
\end{defn}

The most important case for us is when $m$ is a power of a prime $p$. Here,
\[
\rat_p = \ol{\rho}_p^\bbC = L + \cdots + L^{p-1}
\]
is the reduced complex regular representation of $C_p$, and
\[
\rat_{p^n} = \tr_{C_p}^{C_{p^n}}(\ol{\rho}_p^\bbC) = \bbC[C_{p^n}] - \bbC[C_{p^n}/C_p].
\]
For example,
\[
\rat_{2^n} = L + L^3 + L^5 + \cdots + L^{2^n-1}.
\]

\subsection{Adams operations in the critical degree}

Our computation of $C_{p^n}$-Mahowald invariants hinges on the following observation: if $V$ is a fixed point free representation of a finite group $G$ with rational characters, then $\psi^\ell(\beta_V)\equiv \beta_V\pmod{\tr_e^G(1)}$ for $\gcd(\ell,|G|) = 1$.

Character theory will allow us to reduce to the case of a cyclic group $G = C_m$. We begin with this case. By \cite[Proposition 7.7.7]{dieck1979transformation}, there is an isomorphism
\[
RU(C_m)[e_{L^i}^{-1} : 1\leq i<m]\cong \bbZ[\tfrac{1}{m}](\zeta_m),
\]
where $e_{L^i} = 1-L^i$ and $\zeta_m$ is a primitive $m$th root of unity. Under this map, $L$ is sent to $\zeta_m$. This refines to an identification
\[
\Phi^{C_n}KU_{C_m} \cong KU[\tfrac{1}{n}](\zeta_n)
\]
for any $n\mid m$. Consider the Bott class $\beta_{\rat_m} \in \pi_{\rat_m}KU_{C_m}$ of the representation $\rat_m$ of \cref{def:rat}.

\begin{prop}\label{lem:geofixv1}
For $n\mid m$, we have
\[
\Phi^{C_n}(\beta_{\rat_m}) = \begin{cases}p^{m/n}&n \in p^\bbN\text{ for some prime }p,\\ 1&\text{otherwise}.\end{cases}
\]
\end{prop}
\begin{proof}
As
\[
\res^{C_m}_{C_n}(\beta_{\rat_m}) = \beta_{\rat_n}^{m/n},
\]
it suffices to consider just the case $n = m$. Let $\Phi_m$ denote the $m$th cyclotomic polynomial, and abbreviate $\zeta = \zeta_m$. We then have
\begin{align*}
\Phi^{C_m}(\beta_{\rat_m}) &= \Phi^{C_m}\left(\prod\{\beta_{L^d} :0<d<m,\, \gcd(d,m) = 1\}\right) \\
&= \prod\{(1-\zeta^d) : 0 < d < m,\, \gcd(d,m)=1\}\\
&= \Phi_m(1) = \begin{cases}p&m\in p^\bbN,\\ 1&\text{otherwise},\end{cases}
\end{align*}
as claimed.
\end{proof}

\begin{lemma}\label{lem:adamsopcyclic}
For $\gcd(\ell,m) = 1$ and $k\geq 1$, we have
\[
\psi^\ell(\beta_{\rat_m}^k) = \left(1+\frac{\ell^{k\varphi(m)}-1}{m}\tr_e^{C_m}(1)\right)\cdot \beta_{\rat_m}^k,
\]
where $\varphi(m) = \dim \rat_m$ is the Euler totient function.
\end{lemma}
\begin{proof}
As the geometric fixed point maps
\[
\Phi^C\colon \pi_VKU_{C_m} \to \pi_{V^C}\Phi^CKU
\]
are jointly injective as $C$ ranges over the subgroups of $C_m$, it suffices to verify this identity after applying $\Phi^{C}$ for all subgroups $C\subset C_m$.

When $C=e$, we are claiming that the formula holds after restriction to the trivial group, which is true as $\res^{C_m}_e(\beta_{\rat_m}) = \beta^{\varphi(m)}$ and $\psi^\ell(\beta) = \ell\beta$.

When $C \neq e$, \cref{lem:geofixv1} implies that $\Phi^C(\beta_{\rat_m}^k)$ lands in $\bbZ\subset \pi_0 \Phi^C KU_{C_m}$, and is therefore fixed by $\psi^\ell$. As $\Phi^C(\tr_e^{C_m}(1)) = 0$, we therefore have
\[
\Phi^C\left(1+\frac{\ell^{k\varphi(m)}-1}{m}\tr_e^{C_m}(1)\right)\cdot \beta_{\rat_m}^k = \Phi^C(\beta_{\rat_m}^k) = \psi^\ell\left(\Phi^C(\beta_{\rat_m}^k)\right) = \Phi^C(\psi^\ell(\beta_{\rat_m}^k))
\]
as claimed.
\end{proof}

We can now give the following key proposition.

\begin{prop}
Let $G$ be a finite group and $V$ be a fixed point free $k$-dimensional complex $G$-representation with rational characters. Then the Bott class $\beta_V \in \pi_V KU_G$ satisfies
\[
\psi^\ell(\beta_V) =  \left(1+\frac{\ell^k-1}{|G|}\tr_e^G(1)\right)\beta_V
\]
for $\ell$ coprime to $|G|$.
\end{prop}
\begin{proof}
Character theory implies that the joint restriction
\[
\pi_V KU_G\rightarrow \prod_{C\subset G \text{ cyclic}}\pi_{\Res^G_C V}KU_C
\]
is an injection, and it therefore suffices to prove that this identity holds after restriction to any cyclic subgroup $C_m\subset G$. The assumption that $V$ is a fixed point free representation with rational characters implies that the same is true of $\Res^G_{C_m}V$. As all fixed point free $C_m$-representations with rational characters are multiples of $\rat_m$, it follows that $\res^G_{C_m}\beta_V$ is a power of $\beta_{\rat_m}$, and so the proposition follows from \cref{lem:adamsopcyclic}.
\end{proof}

\begin{cor}\label{lem:adamsop}
If $G$ is a $p$-group of order $p^n$ and $V$ is a fixed point free complex $G$-representation with rational characters of complex dimension $p^kc(p-1)$ (or $2^{k-1}c$ with $k\geq 2$ for $p=2$) with $p\nmid c$, then
\[
(\psi^\ell-\id)(\beta_V) = d\cdot p^{k+1-n}\cdot \tr_e^G(1)\cdot \beta_V
\]
with $p\nmid d$.
\qed
\end{cor}

\begin{rmk}\label{rmk:sphericalforms}
The finite groups admitting a fixed point free complex representation have been completely classified \cite{wolf1967spaces}. For $p$-groups the classification is much simpler. The only abelian groups admitting a fixed point free complex representation are the cyclic groups, so if a group admits a fixed point free complex representation then all its abelian subgroups are cyclic, and the only $p$-groups satisfying this are the cyclic $p$-groups and the generalized quaternion groups \cite[Theorem XII.11.6]{cartaneilenberg1956homological}.
\tqed
\end{rmk}

\section{\texorpdfstring{$C_{p^n}$}{C\_pn}-Mahowald invariants of the Burnside ring}\label{sec:mi}

We now proceed to the proofs of \cref{thm:mainroot1} and \cref{thm:mainroot2}, and as a consequence, of \cref{thm:fixeddegree}. Fix a prime $p$ and positive integer $n$, and let $\ell$ be a primitive root mod $p^2$ (say $\ell = 3$ if $p=2$). 

\begin{defn}\label{def:fixedmodule}
In this section, we write
\[
M_k^{C_{p^n}} = H^0(\psi^\ell;\pi_{0,k} KO_{C_{p^n}})/(\text{torsion})
\]
for $k > 0$ (see \cref{rmk:jses}), taking the convention that $M_0^{C_{p^n}}\subset H^0(\psi^\ell;\pi_{0,0} KO_{C_{p^n}}) \cong R\bbQ(C_{p^n})$ is the augmentation ideal.
\tqed
\end{defn}

In \cref{cor:reducetoktheory}, we saw that $M_k^{C_{p^n}}\cong m_k(C_{p^n})$ for $k \geq 1$, so our main task is to understand these groups and how they fit together as $k$ varies.

We begin in \cref{ssec:burnsidering} by fixing some explicit bases for the Burnside ring $A(C_{p^n})$ and recording various useful identities for later use. In \cref{ssec:eulerclasses}, we then give a convenient description of the spoke-graded homotopy groups
$
\pi_{0,\ast}KU_{C_{p^n}}.
$
Here it is essential that we take spoke-graded homotopy groups to be defined using the sequence described in \cref{rmk:sequence}.

We then compute $M_k^{\bbC, \smash{C_{p^n}}}$, the $KU$-analogue of the groups $M_k^{\smash{C_{p^n}}}$, in \cref{ssec:complexfixedpt}. This is, in short, the hard part of our computation, if not of the paper as a whole. The further descent to $M_k^{\smash{C_{p^n}}}$, which we carry out in \cref{ssec:realfixedpt}, is then mostly a matter of bookkeeping. Finally, we put everything together to prove the main theorems \cref{thm:fixeddegree} and \cref{thm:mainroot1} in \cref{ssec:cpnmfilt} and \cref{thm:mainroot2} in \cref{ssec:minv}.

\subsection{The Burnside ring of a cyclic \texorpdfstring{$p$}{p}-group}\label{ssec:burnsidering}

We begin by fixing some notation and recording some facts about the Burnside ring $A(C_{p^n})\cong R\bbQ(C_{p^n})$, starting with two convenient bases.

\begin{defn}\label{def:burnsidebases}
Define
\begin{align*}
t_{n,i} &= \tr_i^n(1) = [C_{p^n}/C_{p^i}],&&0\leq i \leq n;\\
z_{n,i} &= pt_{n,i} - t_{n,i-1},&&1\leq i \leq n;
\end{align*}
as well as
\begin{align*}
t_{n,-1} &= 0,& t_{n,n+1} = 1;\\
z_{n,0} &= pt_{n,0},&z_{n,n+1} = 0.
\tag*{$\triangleleft$}
\end{align*}
\end{defn}

\begin{lemma}\label{lem:bases}
The following hold.
\begin{enumerate}
\item $t_{n,0},\ldots,t_{n,n}$ is a basis for $A(C_{p^n})$;
\item $t_{n,1},\ldots,t_{n,n}$ is a basis for $A(C_{p^n})/t_{n,0}$;
\item $z_{n,i+1},\ldots,z_{n,n}$ is a basis for $\ker(\res^n_i\colon A(C_{p^n})\rightarrow A(C_{p^i}))$;
\item $1,z_{n,1},\ldots,z_{n,n}$ is a basis for $A(C_{p^n})$.
\end{enumerate}
\end{lemma}
\begin{proof}
Easily verified.
\end{proof}

We now record various useful identities. The proofs are elementary, and we omit them. 

\begin{defn}
Write
\[
\phi_{n,j}\colon A(C_{p^n})\rightarrow\bbZ,\qquad \phi_{n,j}(X) = |X^{C_{p^j}}|,
\]
so that
\[
\phi\colon A(C_{p^n})\rightarrow\Cl(\Sub(C_{p^n}),\bbZ),\qquad (\phi X)(C_{p^j}) = \phi_{n,j}(X)
\]
is the marks homomorphism.
\tqed
\end{defn}

\begin{lemma}\label{lem:marks}
We can identify
\begin{align*}
\phi_{n,j}(t_{n,i}) &= \begin{cases} p^{n-i}&0\leq j \leq i,\\ 0&i<j\leq n;\end{cases}\\
\phi_{n,j}(z_{n,i}) &= \begin{cases}p^{n+1-i}&j = i,\\0&j\neq i.\end{cases}
\tag*{$\qed$}
\end{align*}
\end{lemma}

This gives access to a convenient strategy for verifying identities involving these elements: apply $\phi_{n,j}$ to both sides for $0\leq j \leq n$ and check they agree. We list some.

\begin{prop}\label{prop:multiplybases}
The two bases are related by
\begin{align*}
t_{n,i} &= p^{n-i} - \sum_{i<k\leq n}p^{k-i-1}z_{n,k},\\
z_{n,i} &= pt_{n,i} - t_{n,i-1},
\end{align*}
and satisfy the identities
\begin{align*}
t_{n,i} \cdot t_{n,j} &= p^{n-j} t_{n,i}\qquad \text{for }i\leq j;\\
z_{n,i}\cdot z_{n,j} &= \begin{cases}p^{n+1-i}z_{n,i}&i = j,\\ 0&i\neq j.\end{cases}
\tag*{$\qed$}
\end{align*}
\end{prop}

\begin{rmk}
The basis $1,z_{n,1},\ldots,z_{n,n}$ is in many ways more convenient than the familiar basis $t_{n,0},\ldots,t_{n,n}$ of transitive orbits, particularly when working mainly in the augmentation ideal $I(C_{p^n})$ as we will be doing. \cref{prop:multiplybases} provides an explanation for this: the nonunital ring $I(C_{p^n})$ admits a multiplicative splitting
\[
I(C_{p^n})\cong \prod_{1\leq i \leq n}p^{n+1-i}\bbZ,
\]
with $i$th summand generated by $z_{n,i}$.
\tqed
\end{rmk}

\begin{prop}\label{prop:recoverfrommarks}
An element $X \in A(C_{p^n})$ can be recovered from its marks by
\[
{\everymath={\displaystyle}
\begin{array}{ccclccc}
X &=&{}&{}&\sum_{0\leq i \leq n-1}\frac{\phi_{n,i}(X)-\phi_{n,i+1}(X)}{p^{n-i}}t_{n,i}&+&\phi_{n,n}(X);\\
X &=&\phi_{n,0}(X) &+& \sum_{1\leq i \leq n}\frac{\phi_{n,i}(X) - \phi_{n,0}(X)}{p^{n+1-i}}z_{n,i}.\\
\end{array}
}
\]
\qed
\end{prop}

\begin{ex}\label{ex:eqp}
A basic element in $A(C_{p^n})$ appearing in the computation of the $C_{p^n}$-Mahowald invariant is the element $e_{\rat_{p^n}} = a_{\rat_{p^n}} \beta_{\rat_{p^n}}$, which lives in $A(C_{p^n})\cong R\bbQ(C_{p^n})\subset RU(C_{p^n})$ as a consequence of \cref{lem:adamsopcyclic}. By \cref{lem:geofixv1} we have 
\[
\phi_{n,j}(e_{\rat_{p^n}}) = \begin{cases}0&j=0,\\
p^{p^{n-j}}&1\leq j \leq n,
\end{cases}
\]
and so \cref{prop:recoverfrommarks} tells us that
\begin{align*}
e_{\rat_{p^n}} = -\frac{p^{p^{n-1}}}{p^n}t_{n,0} + \left(\sum_{0<j<n} \frac{p^{p^{n-j}}-p^{p^{n-j-1}}}{p^{n-j}} t_{n,j}\right) + p 
= \sum_{1\leq j \leq n}\frac{p^{p^{n-j}}}{p^{n+1-j}}z_{n,j}.
\end{align*}
\tqed
\end{ex}

\begin{defn}\label{def:burnsideweylphi}
We write
\[
\Phi\colon A(C_{p^n})\rightarrow A(C_{p^n}/C_p)\cong A(C_{p^{n-1}}),\qquad \Phi(X) = X^{C_p}
\]
for the homomorphism sending a $C_{p^n}$-set $X$ to its $C_p$-fixed points with residual $C_{p^n}/C_p\simeq C_{p^{n-1}}$-action.
\tqed
\end{defn}

\begin{lemma}
The homomorphism $\Phi$ induces an isomorphism $A(C_{p^n})/(t_{n,0}) \cong A(C_{p^{n-1}})$, and satisfies
\[
\phi_{n,i}  = \phi_{n-1,i-1}\circ \Phi
\]
for $1\leq i \leq n$.
\qed
\end{lemma}

\begin{ex}
We have $\Phi(e_{\rat_{p^n}}) = N_e^{C_{p^{n-1}}}(p)$.
\tqed
\end{ex}

Our indexing conventions have been chosen to make the following true.

\begin{lemma}\label{prop:mackbasis}
If $y = t$ or $z$, then
\begin{align*}
\tr_n^{n+1}(y_{n,i}) &= y_{n+1,i},\\
\res^n_{n-1}(y_{n,i}) &= py_{n-1,i},\\
\Phi(y_{n,i}) &= y_{n-1,i-1}
\end{align*}
when both sides are defined.
\qed
\end{lemma}

Finally we note the following.

\begin{prop}\label{prop:augmentationfp}
There is a short exact sequence 
\begin{center}\begin{tikzcd}
0\ar[r]&I(C_{p^n})\ar[r,"\Phi"]&A(C_{p^{n-1}})\ar[r,"\epsilon"]&\bbZ/(p^n)\ar[r]&0
\end{tikzcd},\end{center}
where $\epsilon(X)$ is the class of $|X|$.
\end{prop}
\begin{proof}
\cref{prop:mackbasis} implies that $\Phi\colon I(C_{p^n})\rightarrow A(C_{p^{n-1}})$ is an injection with image spanned by $z_{n-1,0},z_{n-1,1},\ldots,z_{n-1,n-1}$. By \cref{lem:bases} $z_{n-1,1},\ldots,z_{n-1,n-1}$ is a basis for $I(C_{p^{n-1}})$, and so $\epsilon$ furnishes an isomorphism $A(C_{p^{n-1}})/(z_{n-1,1},\ldots,z_{n-1,n-1})\cong \bbZ$. The proposition then follows after further modding out by $z_{n-1,0} = p t_{n-1,0}$ as $\epsilon(pt_{n-1,0}) = p^n$.
\end{proof}

\subsection{Complex \texorpdfstring{$K$}{K}-theory Euler classes}\label{ssec:eulerclasses}

We now describe some aspects of $\pi_{0,\ast}KU_{C_{p^n}}$. Recall that the sequence $\ul{d} = (d_1,d_2,\ldots)$ is defined as $\ul{d} = (a_1,-a_1,a_2,-a_2,\ldots)$ where $(a_1,a_2,\ldots)$ is the sequence of positive integers coprime to $p$ in increasing order. The key properties of this sequence are the following.

\begin{lemma}\label{lem:dseq}
The above sequence $\ul{d}$ has the following properties:
\begin{enumerate}
\item If $b-a = p^{i-1}(p-1)$, then $\res^{C_{p^n}}_{C_{p^i}}(L^{d_{a+1}}+\cdots+L^{d_b}) =\rat_{p^i}$.
\item If $a<b$, then $L^{d_{2a+1}}+\cdots+L^{d_{2b}} \cong \bbH\otimes_\bbC(L^{d_{2a+1}}+L^{d_{2a+3}}+\cdots+L^{d_{2b-1}})$ is quaternionic.
\qed
\end{enumerate}
\end{lemma}

Abbreviate
\begin{equation}\label{eq:eab}
V_k = L^{d_1}+\cdots+L^{d_k},\qquad \beta_k = \beta_{V_k},\qquad e_{a,b} = e_{V_b-V_a}.
\end{equation}
We then have
\[
\pi_{0,2k} KU_{C_{p^n}} = RU(C_{p^n})\{\beta_k\},
\]
and the map
\[
a^{i-j}\colon \pi_{0,2i} KU_{C_{p^n}} \rightarrow \pi_{0,2j}KU_{C_{p^n}}
\]
is determined by
\[
a^{i-j}\beta_i = e_{i,j}\beta_j.
\]
Thus understanding $\pi_{0,2\ast}KU_{C_{p^n}}$ amounts to understanding the elements $e_{a,b}$. 

\begin{defn}
Write
\[
\phi_{n,j}\colon RU(C_{p^n})\rightarrow \bbZ[\tfrac{1}{p^j}](\zeta_{p^j})
\]
for the geometric fixed point map, sending $L$ to the primitive root of unity $\zeta_{p^j}$.
\tqed
\end{defn}

These extend the marks homomorphism, and $A(C_{p^n})\cong R\bbQ(C_{p^n})\subset RU(C_{p^n})$ may be identified as the set of virtual representation $\alpha$ for which $\phi_{n,j}(\alpha)\in\bbZ$ for all $0\leq j \leq n$.

\begin{lemma}\label{lem:phiba}
If $b-a = p^{i-1}(p-1)$ then $\phi_{n,i}(e_{a,b}) = p$.
\end{lemma}
\begin{proof}
We compute
\begin{align*}
\phi_{n,i}(e_{a,b}) &= \phi_{i,i}(\res^n_i(e_{a,b}))\\
&= \phi_{i,i}(e_{\rat_{p^i}}) && \text{\cref{lem:dseq}} \\
&= \Phi^{C_{p^i}}(a_{\rat_{p^i}}\beta_{\rat_{p^i}}) \\
&= \Phi^{C_{p^i}}(\beta_{\rat_{p^i}})\\
&= p&&\text{\cref{lem:geofixv1}}
\end{align*}
as claimed.
\end{proof}

\begin{lemma}
If $b-a = p^{i-1}(p-1)c$ then $e_{a,b} z_{n,i} = p^c z_{n,i}$.
\end{lemma}
\begin{proof}
As both sides are in the augmentation ideal, it suffices to verify that $\phi_{n,j}(e_{a,b}z_{n,i}) = p^c\phi_{n,j}(z_{n,i})$ for $1\leq j \leq n$. If $j\neq i$, then both sides vanish by \cref{lem:marks}. If $j=i$, then \cref{lem:phiba} implies $\phi_{n,j}(e_{a,b}) = p^c$. So both sides agree as claimed.
\end{proof}

\begin{lemma}\label{prop:erz}
If $b-a=p^{i-1}(p-1)$ then
\begin{align*}
e_{a,b}t_{n,i} &= \sum_{1\leq j \leq i}p^{p^{i-j}+j-i-1}z_{n,j} \\
&\equiv z_{n,i}\pmod{pz_{n,1},\ldots,pz_{n,i-1}}.
\end{align*}
\end{lemma}
\begin{proof}
By \cref{lem:phiba} we have
\[
\phi_{n,j}(e_{a,b}t_{n,i}) = \begin{cases}
p^{p^{i-j}+n-i}&1\leq j \leq i,\\
0&i<j\leq n.
\end{cases}
\]
The lemma then follows from \cref{prop:recoverfrommarks}. Here, we may ignore $\phi_{n,0}$ as everything in sight is in the augmentation ideal.
\end{proof}

So far we have focused on the even degrees $\pi_{0,2\ast}KU_{C_{p^n}}$. The odd degrees can be understood via the following.

\begin{lemma}\label{lem:kuodd}
The sequence
\begin{center}\begin{tikzcd}
\pi_{0,2k-2}KU_{C_{p^n}}&\ar[l,"a^{1/2}"']\pi_{0,2k-1}KU_{C_{p^n}}&\ar[l,"="']\pi_{0,2k-1}KU_{C_{p^n}}&\ar[l,"a^{1/2}"']\pi_{0,2k}KU_{C_{p^n}}
\end{tikzcd}\end{center}
is isomorphic to
\begin{center}\begin{tikzcd}[column sep=large]
RU(C_{p^n})\beta_{k-1}&\ar[l,tail]\ker(\res^n_0)\beta_{k-1}&\ar[l,"(1-L^{d_k})\beta_{L^{d_k}}^{-1}"',"\cong"]\coker(\tr_0^n)&RU(C_{p^n})\beta_k\ar[l,two heads]
\end{tikzcd}.\end{center}
\end{lemma}
\begin{proof}
As $\pi_{1,2\ast}KU_{C_{p^n}} = 0$, there are short exact sequences
\begin{center}\begin{tikzcd}
0\ar[r]&\pi_{2k}KU\ar[r,"\tr_0^n"]&\pi_{0,2k}KU_{C_{p^n}}\ar[r,"a^{1/2}"]&\pi_{0,2k-1}KU_{C_{p^n}}\ar[r]&0\\
0\ar[r]&\pi_{0,2k-1} KU_{C_{p^n}}\ar[r,"a^{1/2}"]&\pi_{0,2k-2}KU_{C_{p^n}}\ar[r,"\res^n_0"]&\pi_{2k-2}KU\ar[r]&0
\end{tikzcd},\end{center}
and the proposition follows as $a^{1/2}\circ a^{1/2} = a_{L^{d_k}}$ by \cref{prop:halfeuler}.
\end{proof}

\subsection{Complex fixed points}\label{ssec:complexfixedpt}

Consider the group $C_{p^n}$ as fixed for this subsection.

\begin{defn}\label{def:cfixedmodule}
For $k > 0$ let
\[
M_k^\bbC = M_k^{\bbC,C_{p^n}} \subset \pi_{0,k}KU_{C_{p^n}}
\]
be the subgroup fixed by $\psi^\ell$ for all $p\nmid \ell$. When $k = 0$ we modify this by taking $M_0^\bbC$ to be the augmentation ideal of $R\bbQ(C_{p^n})$.
\tqed
\end{defn}

Our goal in this subsection is to prove the following.

\begin{theorem}\label{thm:kufp}
Fix $k>0$, and let $b_s = p^{s-1}(p-1) \ceil{\frac{k}{p^{s-1}}}$ denote the smallest multiple of $p^{s-1}(p-1)$ greater than or equal to $k(p-1)$. Then
\[
M_{2k(p-1)-1}^\bbC = \bbZ\{a^{b_s-k(p-1)}\cdot a^{1/2}t_{n,s}\beta_{b_s} : 1 \leq s \leq n\},
\]
and every map in the composite
\[
a^{(p-1)-1/2}\colon M_{2k(p-1)-1}^\bbC\rightarrow M_{2(k-1)(p-1)}^\bbC
\]
is an isomorphism.
\tqed
\end{theorem}

Write $\widetilde{M}_k^\bbC\subset \pi_{0,k}KU_{C_{p^n}}$ for the subgroup described in \cref{thm:kufp}, so that we wish to prove $\widetilde{M}_k^\bbC = M_k^\bbC$. We proceed in several steps.

\begin{lemma}\label{lem:explicitgens}
Fix $k > 0$ and let
\[
q_s = 2p^{s-1}(p-1)\ceil{\frac{k+1}{2p^{s-1}(p-1)}}
\]
be the smallest multiple of $2p^{s-1}(p-1)$ strictly larger than $k$.
Then $\widetilde{M}_k^\bbC = \bbZ\{y_1,\ldots,y_n\}$ where
\[
y_s = a^{q_s-k}t_{n,s} \beta_{q_s/2}.
\]
\end{lemma}
\begin{proof}
This is just a reindexing of the description in \cref{thm:kufp}.
\end{proof}

\begin{lemma}
If $k$ is a multiple of $p^{i-1}(p-1)$, then $a^{1/2}t_{n,i}\beta_k$ is fixed by $\psi^\ell$. In particular, $\widetilde{M}_k^\bbC\subset M_k^\bbC$.
\end{lemma}
\begin{proof}
If $k = p^{i-1}c(p-1)$, then $\res^n_i(\beta_k) = \beta_{\rat_{p^i}}^c$ by \cref{lem:dseq}. It follows from \cref{lem:adamsopcyclic} that $\psi^\ell(\beta_{\rat_{p^i}}^c) = \beta_{\rat_{p^i}}^c + d \cdot t_{i,0}$ for an integer $d$, and thus
\begin{align*}
\psi^\ell(a^{1/2}t_{n,i}\beta_k) &= a^{1/2}\tr_i^n(\psi^\ell(\beta_{\rat_{p^i}}^c))\\
& = a^{1/2}\tr_i^n(\beta_{\rat_{p^i}}^c + d\cdot t_{i,0})\\
&= a^{1/2}\left(t_{n,i}\beta_k + d\cdot t_{n,0}\right) = a^{1/2}t_{n,i}\beta_k
\end{align*}
as claimed, where the last equality follows since $a^{1/2}t_{n,0} = 0$.
\end{proof}

Consider the injection
\[
\tilde{\phi} = \tilde{\phi}\circ a^{k/2}\colon M_k^\bbC\rightarrow \Hom(\{1,\ldots,n\},\bbZ),\qquad (\tilde{\phi}y)(i) = \phi_{n,i}(a^{k/2}y),
\]
where we identify $a^{k/2}y \in A(C_{p^n})\cong R\bbQ(C_{p^n})$.

\begin{lemma}\label{lem:yintegral}
Fix $k\geq 1$ and let $y_1,\ldots,y_n$ be the generators defined for $\widetilde{M}_k^\bbC$. Then
\[
(\tilde{\phi}y_s)(i) = \begin{cases}
p^{n-s+p^{s-i}\ceil{\frac{k+1}{2p^{s-1}(p-1)}}} & 1\leq i \leq s,\\
0&s<i\leq n.
\end{cases}
\]
In particular, $\widetilde{M}_k^\bbC$ is a free abelian group of rank $n$, and $a^{1/2}\colon \widetilde{M}_k^\bbC\rightarrow \widetilde{M}_{k-1}^\bbC$ is an isomorphism after inverting $p$.
\end{lemma}
\begin{proof}
The generators $y_s$ were identified explicitly in \cref{lem:explicitgens}. Using the notation from there, as $\tilde{\phi}(a^{1/2}) = 1$, we have
\begin{align*}
\tilde{\phi}(y_s) &= \tilde{\phi}(t_{n,s}\beta_{q_s/2}) \\
&= \begin{cases}
p^{n-s+p^{s-i}\ceil{\frac{k+1}{2p^{s-1}(p-1)}}} & 1\leq i \leq s,\\
0&s<i\leq n.
\end{cases}
\end{align*}
by \cref{lem:marks} and \cref{lem:geofixv1}. It follows that $\tilde{\phi}(y_1),\ldots,\tilde{\phi}(y_n)$ are upper triangular in the standard basis for $\Hom(\{1,\ldots,n\},\bbZ)$ with powers of $p$ along the diagonal. Hence they form a basis for $\Hom(\{1,\ldots,n\},\bbZ[\tfrac{1}{p}])$, and the final claims follow from this.
\end{proof}

\begin{lemma}\label{lem:rankm}
$M_k^\bbC$ is a free abelian group of rank $n$, and $a^{1/2}\colon M_{k+1}^\bbC\rightarrow M_k^\bbC$ is an isomorphism after inverting $p$.
\end{lemma}
\begin{proof}
As $\pi_{0,k}KU_{C_{p^n}}$ is a free abelian group of finite rank, so is $M_k^\bbC$. The splitting of $C_{p^n}$-spectra away from $p$ by geometric fixed points implies
\[
\pi_{0,k}KU_{C_{p^n}}[\tfrac{1}{p}]\cong \pi_k KU[\tfrac{1}{p}] \oplus \bigoplus_{1\leq i \leq n}\pi_0 \Phi^{C_{p^i}} KU_{C_{p^n}}[\tfrac{1}{p}] \cong \pi_k KU[\tfrac{1}{p}]\oplus\bigoplus_{1\leq i \leq n}\bbZ[\tfrac{1}{p}](\zeta_{p^i}).
\]
As $\bbZ[\tfrac{1}{p}](\zeta_{p^i})$ is an $\Aut(C_{p^i})$-Galois extension of $\bbZ[\tfrac{1}{p}]$, and the Adams operations act on this through the Galois action, the subgroup of fixed points for the Adams operations on this summand is just $\bbZ[\tfrac{1}{p}]$. The lemma follows, after observing that if $k > 0$ then no nonzero element of $\pi_k KU[\tfrac{1}{p}]$ is fixed by the Adams operations, and if $k=0$ then we have thrown out this summand by restricting to the augmentation ideal.
\end{proof}

Combining the above three lemmas shows that $\widetilde{M}_k^\bbC\subset M_k^\bbC$ is an isomorphism after inverting $p$, and hence to prove that it is an isomorphism it suffices to verify that the basis of $\widetilde{M}_k^\bbC$ consists of elements which are linearly independent mod $p$ in $M_k^\bbC$. The following will allow us to restrict our attention to just certain degrees.

\begin{lemma}\label{lem:isop}
Each map in the composite
\[
a^{(p-1)-1/2}\colon M_{2(k+1)(p-1)-1}^\bbC\rightarrow M_{2k(p-1)}^\bbC
\]
is an isomorphism for $k\geq 0$.
\end{lemma}
\begin{proof}
We treat the cases $p=2$ and $p>2$ separately. First suppose $p=2$. Here, we are claiming that
\[
M_{2k+1}^\bbC\rightarrow M_{2k}^\bbC
\]
is an isomorphism. Consider the short exact sequence
\[
0\rightarrow\pi_{0,2k+1}KU_{C_{2^n}}\rightarrow \pi_{0,2k}KU_{C_{2^n}}\rightarrow \pi_{2k}KU\rightarrow 0.
\]
If $k>0$ then $H^0(\psi^\ell;\pi_{2k}KU) = 0$ and it follows that $M_{2k+1}^\bbC \cong M_{2k}^\bbC$. If $k=0$ then this induces an isomorphism from $M_1^\bbC$ onto the augmentation ideal of $R\bbQ(C_{2^n})$, which is how we defined $M_0^\bbC$.

Next suppose $p>2$. As $a^{1/2}$ is an isomorphism after inverting $p$ by \cref{lem:rankm}, it is an injection, so we must only show that the claimed map is a surjection. It is an isomorphism after inverting $p$, so it suffices to prove that it is a surjection after completing at $p$. As $p > 2$, we can identify $J_{C_{p^n}}$ as the fiber of $\psi^\ell-\psi^1$ on $(KU_{C_{p^n}})_p^\wedge$. As $\pi_{1,k}KU_{C_{p^n}} = 0$, it follows that $\bbZ_p\otimes M_k^\bbC \cong \pi_{0,k}J_{C_{p^n}}$. As $\pi_k J_p = 0$ for $(p-1)\nmid (k+1)$, the exact sequences
\[
\pi_{0,2(k+1)-i}J_{C_{p^n}}\rightarrow \pi_{0,2(k+1)-i-1}J_{C_{p^n}}\rightarrow\pi_{2(k+1)-i-1} J_p
\]
prove that each map in the composite
\[
a^{(p-1)-1/2}\colon \pi_{0,2(k+1)(p-1)-1}J_{C_{p^n}}\rightarrow\pi_{0,2k(p-1)}J_{C_{p^n}}
\]
is a surjection as claimed.
\end{proof}

Thus to compute $M_k^\bbC$ for all $k\geq 1$ it suffices to consider just the cases where $2(p-1)\mid k$, or alternately just the cases where $2(p-1)\mid (k+1)$. Our approach will be to play these cases against each other by an inductive argument. We begin by giving a more convenient description of the groups $\widetilde{M}_{2k(p-1)}^\bbC$.

\begin{lemma}\label{lem:specialcase}
Write $k = p^{l-1}c$ with $p\nmid c$. Then, with notation from \cref{thm:kufp}, we have
\[
\widetilde{M}_{2k(p-1)}^\bbC = \bbZ\{z_{n,i}\beta_{k(p-1)} : 1 \leq i \leq l\} \oplus \bbZ\{a^{b_i-k(p-1)}t_{n,i}\beta_{b_i} : l < i \leq n\}.
\]
\end{lemma}
\begin{proof}
Let $b_s^+$ denote the smallest multiple of $p^{s-1}(p-1)$ strictly greater than $k(p-1)$. Then $b_s^+$ is the smallest multiple of $p^{s-1}(p-1)$ greater than or equal to $(k+1)(p-1)$, so by definition
\[
\widetilde{M}_{2k(p-1)}^\bbC = \bbZ\{a^{b_i^+-k(p-1)}t_{n,i}\beta_{b_i^+}: 1 \leq i \leq n\}.
\]
As $k = p^{l-1}c$ with $p\nmid c$, we can identify
\[
b_i^+ = \begin{cases}
k(p-1)+p^{i-1}(p-1)&i\leq l,\\
b_i & i > l.
\end{cases}
\]
So this further reduces to the assertion that
\begin{align*}
\widetilde{M}_{2k(p-1)}^\bbC &= \bbZ\{e_{k(p-1),k(p-1)+p^{i-1}(p-1)}t_{n,i}\beta_{k(p-1)} : 1 \leq i \leq l\} \\
&\oplus \bbZ\{a^{b_i-k(p-1)}t_{n,i}\beta_{b_i} : l < i \leq n \}.
\end{align*}
Now the lemma follows from \cref{prop:erz}, which implies that
\[
e_{k(p-1),k(p-1)+p^{i-1}(p-1)} t_{n,i} \equiv z_{n,i}\pmod{z_{n,1},\ldots,z_{n,i-1}}
\]
and thus 
\[
\bbZ\{e_{k(p-1),k(p-1)+p^{i-1}(p-1)}t_{n,i}\beta_{k(p-1)} : 1 \leq i \leq l\} = \bbZ\{z_{n,i}\beta_{k(p-1)} : 1 \leq i \leq l\}
\]
as claimed.
\end{proof}

\begin{lemma}\label{lem:evens}
We have
\[
\widetilde{M}_{2k(p-1)}^\bbC = M_{2k(p-1)}^\bbC = I(C_{p^n})\beta_{k(p-1)}
\]
provided $p^{n-1}\mid k$.
\end{lemma}
\begin{proof}
If $p^{n-1}\mid k$, then
\[
\psi^\ell(\beta_{k(p-1)})\equiv \beta_{k(p-1)} \pmod{t_{n,0}}
\]
by \cref{lem:adamsopcyclic}, and it follows that
\[
M_{2k(p-1)}^\bbC = I(C_{p^n})\beta_{k(p-1)} = \bbZ\{z_{n,1},\ldots,z_{n,n}\}\beta_{k(p-1)} = \widetilde{M}_{2k(p-1)}^\bbC.
\]
by \cref{lem:specialcase}.
\end{proof}

\begin{lemma}\label{lem:induction}
Suppose that $\widetilde{M}_{2t(p-1)}^\bbC = M_{2t(p-1)}^\bbC$ for $p^l\mid t$. Then $\widetilde{M}_{2k(p-1)}^\bbC = M_{2k(p-1)}^\bbC$ and $\widetilde{M}_{2k(p-1)-1}^\bbC = M_{2k(p-1)-1}^\bbC$ for $p^{l-1} \mid k$.
\end{lemma}
\begin{proof}
By \cref{lem:specialcase}, we wish to prove
\begin{align*}
M_{2k(p-1)}^\bbC &= \bbZ\{z_{n,i}\beta_{k(p-1)}:1\leq i \leq l\}\oplus\bbZ\{a^{b_i-k(p-1)}t_{n,i}\beta_{b_i}:l<i\leq n\},\\
M_{2k(p-1)-1}^\bbC &= \bbZ\{a^{1/2}t_{n,i}\beta_{k(p-1)}:1\leq i \leq l\}\oplus\bbZ\{a^{b_i-k(p-1)+1/2}t_{n,i}\beta_{b_i}:l<i\leq n\}.
\end{align*}
By \cref{lem:rankm}, it suffices to prove that these elements are linearly independent mod $p$. In other words, if $y_i$ is the $i$th listed element, then we need to prove that
\[
y_i\not\equiv 0 \pmod{p,y_1,\ldots,y_{i-1}}.
\]
If $i\leq l$ then this is clear. By definition $z_{n,i} = p t_{n,i} - t_{n,i-1}$. Thus if $i > l$ then the assertion that $y_i\not\equiv 0 \pmod{p,y_1,\ldots,y_{i-1}}$ for $M_{2k(p-1)-1}^\bbC$ implies the same for $M_{2k(p-1)}^\bbC$, so we restrict our attention to the former case.

Let $b = b_{l+1} - p^l(p-1)$ be the largest multiple of $p^l(p-1)$ smaller than $k(p-1)$. The inductive hypothesis then ensures
\[
M_{2b}^\bbC = \bbZ\{z_{n,i}\beta_{b} : 1 \leq i \leq l+1\} \oplus \bbZ\{a^{b_i-b}t_{n,i}\beta_{b_i} : l+1 < i \leq n\}.
\]
There is a map
\[
\alpha = a^{k(p-1)-b-1/2}\colon M_{2k(p-1)-1}^\bbC\rightarrow M_{2b}^\bbC,
\]
and it suffices now to prove that $\alpha(y_i)$ is nonzero mod $p,\alpha(y_1),\ldots,\alpha(y_{i-1})$ for $i>l$. This map is given by
\[
\alpha(y_i) = \begin{cases}
e_{b,k(p-1)}t_{n,i}\beta_b & 1\leq i \leq l,\\
a^{b_i-b}t_{n,i}\beta_{b_i} = e_{b,b_i}t_{n,i}\beta_b & l<i \leq n.
\end{cases}
\]
\cref{prop:erz} implies $\alpha(y_i)\in\bbZ\{z_{n,i}\beta_b : 1 \leq i \leq l\}$ for $1\leq i \leq l$, so it suffices to prove that $\alpha(y_i)$ is nonzero mod $p,z_{n,1},\ldots,z_{n,l},\alpha(y_{l+1}),\ldots,\alpha(y_{i-1})$ for $l<i$. If $i=l+1$ then \cref{prop:erz} further implies
\[
\alpha(y_{l+1}) \equiv z_{n,l+1}\pmod{z_{n,1},\ldots,z_{n,l}}.
\]
In particular
\[
\bbZ\{z_{n,1}\beta_b,\ldots,z_{n,l+1}\beta_b\} = \bbZ\{z_{n,1}\beta_b,\ldots,z_{n,l}\beta_b\}\oplus\bbZ\{\alpha(y_{l+1})\},
\]
so the inductive hypothesis can be written as
\[
M_{2b}^\bbC = \bbZ\{z_{n,i}\beta_b:1\leq i \leq l \}\oplus\bbZ\{\alpha(y_{l+1})\}\oplus\bbZ\{\alpha(y_i):l+1<i\leq n\},
\]
and the lemma follows.
\end{proof}

We now give the following.

\begin{proof}[Proof of \cref{thm:kufp}]
We must show that $\widetilde{M}_k^\bbC = M_k^\bbC$ for all $k$. \cref{lem:evens} implies that $\widetilde{M}_{2k(p-1)}^\bbC = M_{2k(p-1)}^\bbC$ for $p^{n-1}\mid k$. Then \cref{lem:induction} implies $\widetilde{M}_{2k(p-1)}^\bbC = M_{2k(p-1)}^\bbC$ and $\widetilde{M}_{2k(p-1)-1}^\bbC = M_{2k(p-1)-1}^\bbC$ for $p^{n-2} \mid k$. Iterating this we find that $\widetilde{M}_{2k(p-1)-1}^\bbC = M_{2k(p-1)-1}^\bbC$ for all $k$. Finally, $\widetilde{M}_{k}^\bbC = M_k^\bbC$ for all $k$ by \cref{lem:isop}. 
\end{proof}

\subsection{Real fixed points}\label{ssec:realfixedpt}

The map $KO_{C_{p^n}}\rightarrow KU_{C_{p^n}}$ induces an injection $M_k^{C_{p^n}}\rightarrow M_k^{\bbC,C_{p^n}}$.

\begin{lemma}\label{lem:kufopodd}
If $p > 2$ then $M_k^{C_{p^n}} \cong M_k^{\bbC,C_{p^n}}$.
\end{lemma}
\begin{proof}
The only way that this could fail is if some class in $M_k^{\bbC,C_{p^n}} \subset \pi_{0,k}KU_{C_{p^n}}$ supports a differential in the homotopy fixed point spectral sequence
\[
E_2 = H^\ast(\{\psi^{\pm 1}\};\pi_{\star,k}KU_{C_{p^n}})\Rightarrow\pi_{\star,\ast}KO_{C_{p^n}}.
\]
Hence it suffices to prove that if $k = p^{s-1}c(p-1)$ then $a^{1/2}t_{n,s}\beta_k$ does not support such a differential. But this follows as $a^{1/2}\cdot \eta = 0$ on the $E_2$-page, as can be seen using the identity $a^{1/2}\cdot t_{n,0} = 0$ and, as $p$ is odd, $t_{n,0}\cdot \eta = \eta$.
\end{proof}

A bit more work is needed when $p=2$.

\begin{theorem}\label{thm:kofp}
Fix $k>0$. Then $M_k^{C_{2^n}} \subset M_k^{\bbC,C_{2^n}}$ can be described as follows. 
\begin{enumerate}
\item When $n=1$, we have $a^{1/2} = a_\sigma$ and
\begin{align*}
M_{8k-j}^{C_{2}} &= \bbZ\{a_\sigma^{j-1}\cdot a_\sigma\beta_{4k}\},&&1\leq j \leq 4,\\
M_{8k-5}^{C_{2}} &= \bbZ\{a_\sigma^3\beta_{4k-1}\},\\
M_{8k-6}^{C_{2}} &= \bbZ\{a_\sigma^2\beta_{4k-2}\},\\
M_{8k-7-\epsilon}^{C_{2}} &= \bbZ\{a_\sigma^\epsilon\cdot a_\sigma\beta_{4k-3}\}&&0\leq \epsilon \leq 1.
\end{align*}
In particular, $M_k^{C_{2}} = M_k^{\bbC,C_2}$ except when $k\equiv 3,4,5\pmod{8}$.

\item When $n=2$, we have
\begin{align*}
M_{8k-j}^{C_4} &= \bbZ\{a^{(j-1)/2}\cdot a^{1/2}t_{2,1}\beta_{4k},a^{(j-1)/2}\cdot a^{1/2}\beta_{4k}\}&&1\leq j \leq 4,\\
M_{8k-5}^{C_4} &= \bbZ\{a^{1/2}t_{2,1}\beta_{4k-2},2a^{1/2}\beta_{4k-2}\},\\
M_{8k-6}^{C_4} &= \bbZ\{a^{1/2}\cdot a^{1/2}t_{2,1}\beta_{4k-2},a\beta_{4k-2}\},\\
M_{8k-7-\epsilon}^{C_4} &= \bbZ\{a^{\epsilon/2}\cdot a^{1/2}t_{2,1}\beta_{4k-3},a^{(\epsilon+1)/2}\cdot a\beta_{4k-2}\}&&0\leq \epsilon \leq 1.
\end{align*}
In particular, $M_k^{C_4} = M_k^{\bbC,C_4}$ except when $k\equiv 3,4,5\pmod{8}$.

\item Let $b_s$ be the smallest multiple of $2^{s-1}$ greater than or equal to $4k$. When $n \geq 3$, we have
\begin{align*}
M_{8k-j}^{C_{2^n}} &= \bbZ\{a^{b_i-4k+(j-1)/2}\cdot a^{1/2}t_{n,i}\beta_{b_i} : 1 \leq i \leq n\},&& 1\leq j \leq 4,\\
M_{8k-5-\epsilon}^{C_{2^n}} &= \bbZ\{a^{\epsilon/2}\cdot a^{1/2}t_{n,1}\beta_{4k-2},a^{\epsilon/2}\cdot a^{1/2}t_{n,2}\beta_{4k-2}\}&&0\leq \epsilon \leq 1\\
&\oplus \bbZ\{a^{b_i-4k+2+\epsilon/2}\cdot a^{1/2}t_{n,i}\beta_{b_i} : 3 \leq i \leq n\}\\
M_{8k-7-\epsilon}^{C_{2^n}} &= \bbZ\{a^{\epsilon/2}\cdot a^{1/2}t_{n,1}\beta_{4k-3},a^{1+\epsilon/2}\cdot a^{1/2}t_{n,2}\beta_{4k-2}\}&&0\leq\epsilon\leq 1\\
&\oplus \bbZ\{a^{b_i-4k+3+\epsilon/2}\cdot a^{1/2}t_{n,i}\beta_{b_i} : 3 \leq i \leq n\}.\end{align*}
In particular, $M_k^{C_{2^n}} = M_k^{\bbC,C_{2^n}}$ except when $k\equiv 4,5\pmod{8}$. 
\end{enumerate}
\end{theorem}
\begin{proof}
The difference between $M_k^{C_{2^n}}$ and $M_k^{\bbC,C_{2^n}}$ is controlled by what classes support differentials in the homotopy fixed point spectral sequence for $\pi_{\ast,k}KO_{C_{2^n}}$. We refer the reader to \cite[Section 6]{balderrama2023equivalences} for a detailed discussion of the homotopy fixed point spectral sequence for cyclic-equivariant $K$-theory, and just describe the key points.

If $k$ is odd, then there is nothing in positive filtration in the homotopy fixed point spectral sequence for $\pi_{\ast,k}KO_{C_{2^n}}$. By \cref{lem:dseq} the representations $V_{2k}$ are quaternionic, and thus $\beta_{2k} = \beta_{V_{2k}}$ is fixed by $\psi^{-1}$. There are differentials
\[
d_3(\beta_{2k}) = \begin{cases}
\eta^3 \beta^{-2}\beta_{2k}& k \text{ odd},\\
0&k\text{ even}.
\end{cases}
\]

We now proceed to the computation. When $n = 1$ the theorem is classical, essentially going back to Landweber \cite{landweber1969equivariant}, except that he uses $\psi^2$ instead of $\psi^\ell$; for us a convenient reference is \cite{balderrama2021borel}. So we omit the proof of this case.

Consider $n = 2$ and write $RO(C_4) = \bbZ\{1,\sigma,\lambda\}$ with $\sigma$ the sign representation and $\lambda$ the faithful irreducible representation, with images $L^2$ and  $L+L^3$ in $RU(C_4)$ respectively. Here we have
\begin{align*}
M_{8k-1}^{\bbC,C_4} &= \bbZ\{a^{1/2}t_{2,1}\beta_{4k},a^{1/2}\beta_{4k}\},\\
M_{8k-2}^{\bbC,C_4} &= \bbZ\{at_{2,1}\beta_{4k},a\beta_{4k}\},\\
M_{8k-3}^{\bbC,C_4} &= \bbZ\{a^{1/2}t_{2,1}\beta_{4k-1},a^{3/2}\beta_{4k}\},\\
M_{8k-4}^{\bbC,C_4} &= \bbZ\{at_{2,1}\beta_{4k-1},a^2\beta_{4k}\},\\
M_{8k-5}^{\bbC,C_4} &= \bbZ\{a^{1/2}t_{2,1}\beta_{4k-2},a^{1/2}\beta_{4k-2}\},\\
M_{8k-6}^{\bbC,C_4} &= \bbZ\{at_{2,1}\beta_{4k-2},a\beta_{4k-2}\},\\
M_{8k-7}^{\bbC,C_4} &= \bbZ\{a^{1/2}t_{2,1}\beta_{4k-3},a^{3/2}\beta_{4k-2}\},\\
M_{8k-8}^{\bbC,C_4} &= \bbZ\{at_{2,1}\beta_{4k-3},a^2\beta_{4k-2}\},
\end{align*}
where $t_{2,1} = 1 + \sigma$. Because $\beta_{4k}$ is a permanent cycle, so are all the classes in $M_{8k-1}^{\bbC,C_4}$ and $M_{8k-2}^{\bbC,C_4}$. In degree $8k-4$ the class $at_{2,1}\beta_{4k-1} = (1+\sigma-\lambda)\beta_{4k-2}$ supports a nonzero differential, leaving behind $2at_{2,1}\beta_{4k-1} = a^2 t_{2,1}\beta_{4k}$. It then follows that in degree $8k-3$ the class $a^{1/2}t_{2,1}\beta_{4k-1}$ must support a differential, leaving behind $2a^{1/2}t_{2,1}\beta_{4k-1} = a^{3/2}t_{2,1}\beta_{4k}$. Thus we have shown
\begin{align*}
M_{8k-3}^{C_4} &= \bbZ\{a^{3/2} t_{2,1}\beta_{4k-1},a^{3/2}\beta_{4k}\},\\
M_{8k-4}^{C_4} &= \bbZ\{a^2 t_{2,1}\beta_{4k-1},a^2\beta_{4k}\}.
\end{align*}
Next consider degree $8k-5$. By \cref{lem:kuodd}, the map $\pi_{0,8k-4}KU_{C_4}\rightarrow\pi_{0,8k-5}KU_{C_4}$ realizes the target as isomorphic to $RU(C_4)/(1+L+L^2+L^3)$, say with basis
\[
\pi_{0,8k-5}KU_{C_4} = \bbZ\{a^{1/2}\beta_{4k-2},a^{1/2}L\beta_{4k-2},a^{1/2}L^2\beta_{4k-2}\}.
\]
There are differentials
\[
d_3(a^{1/2}\beta_{4k-2}) = a^{1/2}\eta^3\beta^{-2}\beta_{4k-2},\qquad d_3(a^{1/2}L^2\beta_{4k-2}) = a^{1/2}\eta^3 L^2 \beta^{-2}\beta_{4k-2}.
\]
As $a^{1/2}(1+L+L^2+L^3) = 0$ and $(L+L^3)\cdot \eta = 0$ on the $E_2$-page, it follows that these targets are equal. As $a^{1/2}L\beta^{4k-2}$ is not fixed by $\psi^{-1}$, we find
\[
\pi_{0,8k-5}KO_{C_4} = \bbZ\{a^{1/2}\beta_{4k-2}+a^{1/2}L^2\beta_{4k-2},2a^{1/2}\beta_{4k-2}\} = \bbZ\{a^{1/2}t_{2,1}\beta_{4k-2},2a^{1/2}\beta_{4k-2}\}.
\]
Everything here is fixed by $\psi^3$ so this is $M_{8k-5}^{C_4}$ as claimed. As $\psi^{-1}$ acts freely in degrees $8k-6$ there are no differentials there, and as $\beta_{4k-4}$ is a permanent cycle there are no differentials in degrees $8k-8$. The map $a^{1/2}\colon \pi_{0,8k-7}KU_{C_4}\rightarrow \pi_{0,8k-8}KU_{C_4}$ induces an injective map on the homotopy fixed point spectral sequence, and so there are no differentials in degrees $8k-7$ either. Thus the groups $M_\ast^{C_4}$ are as described.

Now consider $n\geq 3$. As $\beta_{4k}$ is a permanent cycle, all generators of the form $a^j t_{n,i}\beta_{4k}$ are permanent cycles. After throwing out these generators, we are left with
\begin{align*}
M_{8k-1}^{\bbC,C_{2^n}} &\equiv 0,\\
M_{8k-2}^{\bbC,C_{2^n}} &\equiv 0,\\
M_{8k-3}^{\bbC,C_{2^n}} &\equiv \bbZ\{a^{1/2}t_{n,1}\beta_{4k-1}\},\\
M_{8k-4}^{\bbC,C_{2^n}} &\equiv \bbZ\{at_{n,1}\beta_{4k-1}\},\\
M_{8k-5}^{\bbC,C_{2^n}} &\equiv \bbZ\{a^{1/2}t_{n,1}\beta_{4k-2},a^{1/2}t_{n,2}\beta_{4k-2}\},\\
M_{8k-6}^{\bbC,C_{2^n}} &\equiv \bbZ\{at_{n,1}\beta_{4k-2},at_{n,2}\beta_{4k-2}\},\\
M_{8k-7}^{\bbC,C_{2^n}} &\equiv \bbZ\{a^{1/2}t_{n,1}\beta_{4k-3},a^{3/2}t_{n,2}\beta_{4k-2}\},\\
M_{8k-8}^{\bbC,C_{2^n}} &\equiv \bbZ\{at_{n,1}\beta_{4k-3},a^2t_{n,2}\beta_{4k-2}\}.
\end{align*}
Both $\beta_{4k-2}$ and $\sigma\beta_{4k-2}$ support differentials. As
\[
at_{n,1}\beta_{4k-1} = (1+\sigma+V)\beta_{4k-2}
\]
with $V\cdot \eta = 0$ on the $E_2$-page, it follows that $at_{n,1}\beta_{4k-1}$ supports a differential and only $2at_{n,1}\beta_{4k-1} = a^2t_{n,1}\beta_{4k}$ survives. This implies by comparison with $M_{8k-3}\rightarrow M_{8k-4}$ that $a^{1/2}t_{n,1}\beta_{4k-1}$ must also support a differential, leaving behind $2a^{1/2}t_{n,2}\beta_{4k-1} = a^{3/2}t_{n,2}\beta_{4k}$. This describes $M_{8k-j}$ for $1\leq j \leq 4$.

Consider degree $8k-5$. The class $a^{1/2}\beta_{4k-2} \in \pi_{0,8k-5}KU_{C_{2^n}}$ does still support a differential as in the $C_4$-equivariant case, but it is not fixed by $\psi^\ell$, so something different happens now. Because $n\geq 3$, we have
\[
t_{n,0}\cdot \eta = t_{n,1}\cdot \eta = t_{n,2}\cdot \eta = (1+\sigma)\cdot \eta,
\]
and thus as $a^{1/2}t_{n,0} = 0$ it follows that both $a^{1/2}t_{n,1}\beta_{4k-2}$ and $a^{1/2}t_{n,2}\beta_{4k-2}$ are permanent cycles. Hence so are all of
\[
at_{n,1}\beta_{4k-2},\quad at_{n,2}\beta_{4k-2},\quad a^{3/2}t_{n,2}\beta_{4k-2},\quad a^2t_{n,2}\beta_{4k-2}.
\]

We are left considering only $a^{1/2}t_{n,1}\beta_{4k-3}$ and $at_{n,1}\beta_{4k-3}$. Here $at_{n,1}\beta_{4k-3} = z_{n,1}\beta_{4k-4}$ survives because $\beta_{4k-4}$ does, and thus $a^{1/2}t_{n,1}\beta_{4k-3}$ must also survive. This yields the claimed description of $M_\ast$.
\end{proof}

This concludes the computation of the groups $M_\ast^{C_{p^n}}$. 

\subsection{The \texorpdfstring{$C_{p^n}$}{C\_pn}-Mahowald filtration}\label{ssec:cpnmfilt}

We have by now essentially proved \cref{thm:mainroot1} and \cref{thm:fixeddegree}. Consider the function
\[
\phi = \phi\circ a^{k/2}\colon M_k^{C_{p^n}} \rightarrow R\bbQ(C_{p^n}) \cong A(C_{p^n}) \rightarrow\Cl(\Sub(C_{p^n}),\bbZ).
\]
This satisfies $(\phi y)(e) = 0$, so it is equivalent to instead consider
\[
\tilde{\phi}\colon M_k^{C_{p^n}}\rightarrow \Fun(\{1,\ldots,n\},\bbZ),\qquad (\tilde{\phi} y)(i) = \phi_{n,i}(a^{k/2}y).
\]

\begin{lemma}\label{lem:phifp2}
Fix $k > 0$ and $M_k^{C_{p^n}} = \bbZ\{y_1,\ldots,y_n\}$ be the ordered basis described by \cref{thm:kufp} (given \cref{lem:kufopodd}) for $p>2$ and by \cref{thm:kofp} for $p=2$. Then $\tilde{\phi}(y_s) = f_{s,k}^{p,n}$ is as defined in \cref{def:thebasis}.
\end{lemma}
\begin{proof}
If $p>2$ then this is shown in \cref{lem:yintegral}. If $p=2$ then the proof is the same only with several more cases to consider, and we omit the details.
\end{proof}

Recall from \cref{cor:bg} that
\[
\Gamma_k(C_{p^{n-1}}) = \im\left(\Phi\colon \pi_{0,k}S_{C_{p^n}}\rightarrow \pi_0 S_{C_{p^{n-1}}}\right)\subset A(C_{p^{n-1}}).
\]

\begin{theorem}\label{thm:cpnmahowaldfilt}
For $k\geq 1$, the image of $\Gamma_k(C_{p^{n-1}})\subset A(C_{p^{n-1}})$ under the marks homomorphism
\[
\phi\colon A(C_{p^{n-1}})\rightarrow \Hom(\{1,\ldots,n\},\bbZ),\qquad (\phi X)(i) = |X^{C_{p^{i-1}}}|
\]
has as basis the functions $f_{1,k}^{p,n},f_{2,k}^{p,n},\ldots,f_{n,k}^{p,n}$ of \cref{def:thebasis}.
\end{theorem}
\begin{proof}
Combining \cref{cor:reducetoktheory} and the isomorphism $\Gamma_k(C_{p^{n-1}}) \cong m_k(C_{p^n})$ for $k\geq 1$ (see \cref{rmk:gammak}), we find that $\phi(\Gamma_k(C_{p^{n-1}})) = \tilde{\phi}(M_k^{C_{p^n}})$. So the theorem follows from \cref{lem:phifp2}.
\end{proof}

\begin{cor}\label{cor:fpdegree}
Let $V$ be a sum of $k>0$ faithful complex characters of $C_{p^n}$. Then the image of the degree function
\[
\deg\colon \pi_{V} S_{C_{p^n}}\rightarrow \Hom(\{1,\ldots,n\},\bbZ),\qquad (\deg y)(i) = \Phi^{C_{p^{i}}}(y)
\]
has as basis the functions $f_{1,2k}^{p,n},f_{2,2k}^{p,n},\ldots,f_{n,2k}^{p,n}$ of \cref{def:thebasis}.
\end{cor}
\begin{proof}
Combine \cref{thm:cpnmahowaldfilt} and \cref{prop:fpd}.
\end{proof}

\begin{cor}\label{cor:fpimage}
Let $V$ be a sum of $k>0$ faithful complex characters of $C_{p^n}$. Then the image of the $C_p$-geometric fixed point map
\[
\Phi^{C_p}\colon \pi_V S_{C_{p^n}}\rightarrow\pi_{V^{C_p}}S_{C_{p^n}/C_p} \cong \pi_0 S_{C_{p^{n-1}}} \cong A(C_{p^{n-1}})
\]
has as basis the $n$ elements $X_{1,k}^{p,n},\ldots,X_{n,k}^{p,n}$ described in \cref{thm:fpimage}.
\end{cor}
\begin{proof}
This follows from \cref{cor:fpdegree} as $\phi(X_{s,k}^{p,n}) = f_{s,2k}^{p,n}$.
\end{proof}

We now explain another perspective on this computation. In general, if $G$ is a finite group and $\alpha$ is a virtual complex $G$-representation, then the degree function
\[
\deg\colon \pi_\alpha S_G \rightarrow \Cl(\{H \subset G : |\alpha^H| = 0\},\bbZ),\qquad (\deg y)(H) = \Phi^Hy
\]
is a rational isomorphism, and thus an injection mod torsion; see, for example, \cite{greenleesquigley2023ranks}. Here, the assumption that $\alpha$ is complex ensures if $\alpha^H = 0$ then there is a canonical equivalence $\Phi^H S^\alpha\simeq S^0$ ensuring that the degree is well defined. Iriye proved that the torsion subgroup and nilradical of $\pi_\star S_G$ agree \cite{iriye1982nilpotency}, so one can view the computation of these degrees as a computation of the reduced ring
\[
\pi_\star^{\red}S_G \cong \pi_\star S_G / \sqrt{0} = \pi_\star S_G / \text{torsion},
\]
such as described explicitly for the group $G = C_2$ in \cite{belmontxuzhang2024reduced}. In particular, \cref{cor:reducetoktheory} shows
\[
\pi_V^{\red}S_{C_{p^n}} \cong H^0(\psi^\ell;\pi_V KO_{C_{p^n}})/\text{torsion}
\]
for any $C_{p^n}$-representation $V$. In the $C_2$-equivariant case, this interpretation of $\pi_{\star}^\red S_{C_2}$ was observed by Crabb \cite{crabb1989periodicity}, essentially as an interpretation of Landweber's computation of both sides implicit in \cite{landweber1969equivariant}. This gives an effective method for understanding $\pi_{\star}^{\red}S_{C_{p^n}}$ in representation degrees. In particular, our computation shows the following.

\begin{theorem}\label{thm:presentation}
Let $L$ be a faithful complex character of $C_{p^n}$. Then the reduced ring $\pi_{\ast L}^\red S_{C_{p^n}}$ is generated by classes
\[
y_{n,i,c} \in \pi^\red_{(p^{i-1}c(p-1)-1)L}S_{C_{p^n}}
\]
for $c\geq 0$ and $1\leq i \leq n$, except when $p=2$ and $i=1$ where we require $c\not\equiv 2\pmod{4}$. These classes satisfy $y_{n,n,0} = a_L$ and generally $y_{n,i,c} = \tr_i^n(y_{i,i,c})$, and are subject to the relations
\begin{align*}
y_{n,i,c} \cdot y_{n,j,d} &= p^{n-j} a_L y_{n,i,(c+dp^{j-i})}\qquad\qquad\qquad(i\leq j),\\
a_L^{p^{i-1}(p-1)}\cdot y_{n,i,c+1}&=py_{n,i,c}+\sum_{0< j< i}\frac{p^{p^{i-j}}-p^{p^{i-j-1}}}{p^{i-j}}y_{n,j, p^{i-j}c},
\end{align*}
except when $p=2$ and $i=1$ and $c \equiv 2 \pmod{4}$ in which case the second relation must be replaced by $a_L^2 \cdot y_{n,1,4k+3} = 4y_{n,1,4k+1}$.
\end{theorem}
\begin{proof}
In general, as the elements $\psi_{d} \in \pi_{L-L^{d}}S_{C_{p^n}}$ discussed in \cref{ssec:circlegroup} are the identity on all nontrivial geometric fixed points, they produce an isomorphism
\[
\pi_{\ast L}^{\red}S_{C_{p^n}} \cong \pi_{0,2\ast}^{\red}S_{C_{p^n}}.
\]
As discussed above, if $k > 0$ then
\[
\pi_{0,2k}^{\red}S_{C_{p^n}} \cong M_{2k}^{C_{p^n}},
\]
and in general $\pi_{0,-2k}S_{C_{p^n}}\cong A(C_{p^n})/(t_{n,0})$ generated by the Euler class $a_L^k$. So the theorem at hand is a reformulation of our computation of $M_{2\ast}^{C_{p^n}}$, as we now elaborate on. In what follows, we shall abuse notation by silently identifying the elements of $\pi_{\ast L}^\red S_{C_{p^n}}\cong \pi_{0,2\ast}^\red S_{C_{p^n}}\cong M_{2\ast}^{C_{p^n}}$.

Suppose first $p>2$. Our computation of $M_{2\ast}^{C_{p^n}}$ then shows that $\pi_{0,2\ast}^{\red}S_{C_{p^n}}$ is generated by the classes
\[
y_{n,i,c} = a t_{n,i}\beta_{p^{i-1}(p-1)c} = \tr_i^n(a\beta_{p^{i-1}(p-1)c})
\]
for $c\geq 0$ and $1\leq i \leq n$, and that it has a basis given by the unit $1$ together with the elements $a_L^m y_{n,i,c}$ for $0\leq m < p^{i-1}(p-1)$. It therefore suffices to compute enough relations between these generators to reduce all products to a linear combination of these basis elements. The first relation is provided by
\begin{align*}
y_{n,i,c}\cdot y_{n,j,d} &= a t_{n,i}\beta_{p^{i-1}(p-1)c} \cdot a t_{n,j}\beta_{p^{j-1}(p-1)d} \\
&= a^2 \cdot t_{n,i}t_{n,j} \cdot \beta_{p^{i-1}(p-1)c}\beta_{p^{j-1}(p-1)d} \\
&= a^2\cdot p^{n-j}t_{n,i} \cdot \beta_{p^{i-1}(p-1)(c+p^{j-i}d)} &&(i\leq j) \\
&= p^{n-j} a_L y_{n,i,(c+dp^{j-i})}.
\end{align*}
It follows that a product of any number of generators can be written as a multiple of a power of $a_L$ multiplied with a generator, so we reduce to describing how powers of $a_L$ act. Here we have
\begin{align*}
a_L^{p^{i-1}(p-1)}&\cdot y_{n,i,c+1} =  a\cdot \tr_i^n\left(a^{p^{i-1}(p-1)}\beta_{p^{i-1}(p-1)(c+1)}\right) = a \cdot \tr_i^n(e_{\rat_{p^i}})\cdot \beta_{p^{i-1}(p-1)c} \\
&= a \cdot \tr_i^n\left(p  + \left(\sum_{0<j<i} \frac{p^{p^{i-j}}-p^{p^{i-j-1}}}{p^{i-j}} t_{i,j}\right) -\frac{p^{p^{i-1}}}{p^i}t_{i,0}\right)\beta_{p^{i-1}(p-1)c}\quad  (\text{see \ref{ex:eqp}})\\
&= \left(pat_{n,i}+\sum_{0<j<i}\frac{p^{p^{i-j}}-p^{p^{i-j-1}}}{p^{i-j}}at_{n,j}\right)\beta_{p^{i-1}(p-1)c}  \\
&= py_{n,i,c}+\sum_{0< j< i}\frac{p^{p^{i-j}}-p^{p^{i-j-1}}}{p^{i-j}}y_{n,j, p^{i-j}c}
\end{align*}
as claimed.

If $p=2$, then the same computation describes $M_{2\ast}^{\bbC,C_{2^n}}$. By \cref{thm:kofp}, the inclusion $M_{2c}^{C_{2^n}} \subset M_{2c}^{\bbC,C_{2^n}}$ is an isomorphism except when $c\equiv 2 \pmod{4}$, in which case $2at_{n,1}\beta_{4k+2} = a^2t_{n,1}\beta_{4k+3}$ is in the image but $at_{n,1}\beta_{4k+2}$ is not; so we must replace $y_{n,1,4k+2}$ by $a\cdot y_{n,1,4k+3}$, and this leads to the stated relations.
\end{proof}

\subsection{The \texorpdfstring{$C_{p^n}$}{C\_pn}-Mahowald invariant}\label{ssec:minv}

It remains to prove \cref{thm:mainroot2}, concerning values of
\[
M_{C_{p^n}}\colon A(C_{p^{n-1}})\rightharpoonup \pi_\ast S.
\]
If $X\in A(C_{p^{n-1}})$ and $|X|\not\equiv 0 \pmod{p^n}$, then $M_{C_{p^n}}(X) = |X| + p^n \bbZ \subset \pi_0 S$, so we may restrict our attention to degrees $\geq 1$. Moreover, the indeterminacy of $M_{C_{p^n}}$ always contains all $q$-torsion for $p\nmid q$, and so if $|M_{C_{p^n}}(X)| = k \geq 1$ then it is determined by its intersection with the $p$-power torsion in $\pi_k S$.

It is not possible to completely control the indeterminacy of $M_{C_{p^n}}$: there is no reason to expect that it cannot contain elements of essentially arbitrary complexity coming from the torsion in $\pi_{0,\ast}S_{C_{p^n}}$. Our computation will thus take the following form. Let $\pi_\ast j_{(p)} \subset \pi_\ast S$ be the split summand of $p$-power torsion height $1$ classes in positive stems corresponding to a solution of the Adams conjecture at the prime $p$. Thus $\pi_1 j_{(2)} = \bbZ/(2)\{\eta\}$, and if $k\geq 2$ then $\pi_k j_{(p)} \subset \pi_k S \rightarrow \pi_k J_{p}$ is an isomorphism. Then our goal is to completely determine $M_{C_{p^n}}(X)\cap \pi_\ast j_{(p)}$ for all $X\in A(C_{p^{n-1}})$.

\begin{lemma}\label{lem:jcomputeroot}
Let $X\in \Gamma_1(C_{p^{n-1}})\subset A(C_{p^{n-1}})$. Then $M_{C_{p^n}}(X)\cap \pi_\ast j_{(p)}$ can be computed as follows. If $X \notin \Gamma_2(C_{p^{n-1}})$ then $M_{C_{p^n}}(X) = \{\eta\}\subset \pi_1 S$. Otherwise, perform the following procedure:
\begin{enumerate}
\item Write $X = \Phi(\tilde{X})$ with $\tilde{X} \in I(C_{p^n})\subset A(C_{p^n})$.
\item Write $\tilde{X} = a^{k/2} Y$ with $Y \in M_k^{C_{p^n}}$ and $k\geq 2$ as large as possible.
\item Lift $Y$ to an element $\tilde{Y} \in \pi_{0,k}J_{C_{p^n}}$ and set $y = \resp(\tilde{Y}) \in \pi_k J_p$.
\end{enumerate}
Then under the isomorphism $\pi_k j_{(p)} \cong \pi_k J_p$ we have $y \in M_{C_{p^n}}(X)\cap \pi_k j_{(p)}$, and all elements of $M_{C_{p^n}}(X)\cap \pi_k j_{(p)}$ arise this way.
\end{lemma}
\begin{proof}
If $X \in \Gamma_1(C_{p^{n-1}})\setminus\Gamma_2(C_{p^{n-1}})$ then $p=2$ and $M_{C_{2^n}}(X) = \{\eta\}$, simply as this is the unique nonzero element of $\pi_1 S$.

So suppose $k\geq 2$ and $X \in \Gamma_k(C_{p^{n-1}})\setminus\Gamma_{k+1}(C_{p^{n-1}})$. By \cref{prop:bg}, if we write $\tilde{X} = a^{k/2} Z$ with $Z \in \pi_{0,k}S_{C_{p^n}}$ then $\resp(Z) \in M_{C_{p^n}}(X)$, and all of $M_{C_{p^n}}(X)$ arises this way. Consider the diagram
\begin{center}\begin{tikzcd}
A(C_{p^n})_p^\wedge\ar[r,"\cong"]&R\bbQ(C_{p^n})_p^\wedge\\
\bbZ_p \otimes \pi_{0,k}S_{C_{p^n}}\ar[r,two heads]\ar[u,"a^{k/2}"]\ar[d,"\resp"']&\pi_{0,k}J_{C_{p^n}}\ar[r,two heads]\ar[l,bend right,"s"']\ar[u,"a^{k/2}"']\ar[d,"\resp"]&\bbZ_p \otimes M_k^{C_{p^n}}\\
\bbZ_p\otimes \pi_k S \ar[r, two heads]&\pi_k J\ar[l,bend right,"s"']
\end{tikzcd}.\end{center}
Here, splittings $s$ exist as indicated making the bottom square commute by \cref{prop:adamsconjecture}. Moreover the image of $s\colon \pi_k J_p \rightarrow\bbZ_p \otimes \pi_k S$ is exactly $\pi_k j_{(p)}\subset \pi_k S$ by definition. Hence if we follow the listed procedure to produce $\tilde{Y} \in \pi_{0,k}J_{C_{p^n}}$, then $Z = s(\tilde{Y})$ satisfies $\resp(Z) = s(y) \in M_{C_{p^n}}(X) \cap \pi_k j_{(p)}$, and all such elements arise this way.
\end{proof}

At this point we have amassed enough information that it is essentially a matter of careful bookkeeping to carry out this process. We begin by describing the indeterminacy introduced in step (3).

\begin{lemma}\label{lem:indeterminacy}
Fix $k\geq 1$. Then the canonical surjection
\[
\pi_{0,k}J_{C_{p^n}} \rightarrow \bbZ_p \otimes M_k^{C_{p^n}}
\]
is an isomorphism except when $p=2$ and $k\equiv 0,1\pmod{8}$, in which case
\begin{enumerate}
\item The kernel $K$ is isomorphic to $\bbZ/(2)$, and
\item If $k=8l+\delta$ with $\delta\in\{0,1\}$ and $l\geq 1$, then under the isomorphism $\pi_k J_2 \cong \pi_k j_{(2)} \subset \pi_k S$ we have $\resp(K) = \bbZ/(2)\{P^{l-1}\epsilon\eta^\delta\}$.
\end{enumerate}
Here, $\epsilon \in \pi_8 S$ is the class detected by $c_0$ in the Adams spectral sequence, and $P = \langle \bs,2,8\sigma\rangle$ is the Adams periodicity operator.
\end{lemma}
\begin{proof}
In general, there is a short exact sequence
\[
0\rightarrow H^1(\psi^\ell; \bbZ_p\otimes \pi_{1,k}KO_{C_{p^n}})\rightarrow \pi_{0,k}J_{C_{p^n}}\rightarrow \bbZ_p\otimes M_k^{C_{p^n}}\rightarrow 0.
\]
The group $\bbZ_p\otimes \pi_{1,k}KO_{C_{p^n}}$ vanishes except when $p = 2$ and $k = 8l+\delta$ for $\delta \in\{0,1\}$, where it is isomorphic to $\bbZ/(2)\{\eta^{1+\delta}\beta_{4l}\}$. As $\resp(\beta_{4l}) = \res^{C_{2^n}}_e(\beta_{4l}) = \beta^{4l}$ is the Bott class in $\pi_{8l}KO$, these lift the classes in $\pi_{8l+\delta}J$ in the image of the boundary map $\pi_{8l+1+\delta}KO = \bbZ/(2)\rightarrow\pi_{8l+\delta} J_2$. If $l \geq 1$, then the image of this subgroup under the splitting of $\pi_{8l+\delta} S\rightarrow \pi_{8l+\delta}J$ is the kernel of the composite $\pi_{8l+\delta}j_{(2)}\rightarrow\pi_{8l+\delta}S\rightarrow\pi_{8l+\delta}KO$, which is exactly  $\bbZ/(2)\{P^{l-1}\epsilon\eta^\delta\}$ as claimed.
\end{proof}

Let $\tilde{\pi}_k J_p$ be the quotient of $\pi_k J_p$ by the image of elements in the kernel of $\pi_{0,k}J_{C_{p^n}}\rightarrow \bbZ_p\otimes M_k^{C_{p^n}}$ described in \cref{lem:indeterminacy}. It follows that there is a well-defined injection
\[
\resp\colon M_k^{C_{p^n}}/M_{k+1}^{C_{p^n}}\rightarrowtail \tilde{\pi}_k J_p.
\]
The target is a finite cyclic group of known order, so our task is to name generators for $\tilde{\pi}_\ast J_p$ for which we can describe the map $\resp\colon M_k^{C_{p^n}}\rightarrow \tilde{\pi}_k J_p$. We begin by identifying the quotients $M_k^{C_{p^n}}/M_{k+1}^{C_{p^n}}$ in certain critical degrees.

\begin{lemma}\label{lem:rescritical}
Suppose $k = 2p^{n-1}(p-1)c-1$, or $2^{n}c-1$ with $n\geq 3$, so that
\[
M_k^{C_{p^n}} = \bbZ\{a^{1/2}t_{n,1}\beta_{(k+1)/2},\ldots,a^{1/2}t_{n,n}\beta_{(k+1)/2}\}.
\]
Then there is an isomorphism
\[
M_k^{C_{p^n}}/M_{k+1}^{C_{p^n}}\cong \bbZ/(p^n),\qquad a^{1/2}t_{n,s}\beta_{(k+1)/2} \mapsto p^{n-s}.
\]
\end{lemma}
\begin{proof}
As in the proof of \cref{lem:evens}, we have
\[
M_{k+1}^{C_{p^n}} = I(C_{p^n})\{\beta_{(k+1)/2}\},\qquad M_k^{C_{p^n}} = A(C_{p^n})/(t_{n,0})\{a^{1/2}\beta_{(k+1)/2}\}.
\]
Now the lemma follows as in the proof of \cref{prop:augmentationfp}:
\[
\coker(I(C_{p^n})\rightarrow A(C_{p^n})/(t_{n,0})) \cong \bbZ/(p^n),\qquad X \mapsto |X|,
\]
and $|t_{n,s}| = p^{n-s}$.
\end{proof}

\begin{lemma}\label{lem:resodd}
Suppose $k = 8l-5$, so that
\[
M_{k}^{C_8} = \bbZ\{a^{1/2}t_{3,1}\beta_{4l-2},a^{1/2}t_{3,2}\beta_{4l-2},a^{5/2}\beta_{4l}\}.
\]
Then there is an isomorphism
\[
M_{k}^{C_8}/M_{k+1}^{C_8} \cong \bbZ/(8),\qquad \begin{cases}
a^{1/2}t_{3,1}\beta_{4l-2} &\mapsto -2,\\
a^{1/2}t_{3,2}\beta_{4l-2} &\mapsto 1,\\
a^{5/2}\beta_{4l} &\mapsto 0.\end{cases}
\]
\end{lemma}
\begin{proof}
We have
\[
M_{k+1}^{C_8} = \bbZ \{a^2 t_{3,1}\beta_{4l}, a^2 t_{3,2}\beta_{4l},a^2\beta_{4l}\},
\]
and
\begin{align*}
a^{1/2}\cdot a^2 t_{3,1}\beta_{4l} &= 4a^{1/2}t_{3,1}\beta_{4l-2},\\
 a^{1/2}\cdot a^2 t_{3,2}\beta_{4l} &= a^{1/2}t_{3,1}\beta_{4l-2} + 2 a^{1/2}t_{3,2}\beta_{4l-2}.
\end{align*}
The lemma follows.
\end{proof}

We can use these to introduce generators for the groups $\tilde{\pi}_k J_p$. A choice is only necessary if $k\equiv -1\pmod{2(p-1)}$, or $k\equiv -1\pmod{4}$ if $p= 2$, as $\tilde{\pi}_k J_{C_{p^n}}$ has at most one nonzero element in all other cases. Recall that if $k = 2p^{l-1}c(p-1)-1$ with $p\nmid c$, then $\pi_k J_p \cong \bbZ/(p^l)$ if $p > 2$, and $\pi_k J_2 \cong \bbZ/(2^{l+1})$ if $l\geq 2$.

\begin{defn}\label{def:jgens}
We define generators of $\pi_k J_p \cong \pi_k j_{(p)}$ for $k\equiv -1\pmod{2(p-1)}$, or $k\equiv -1\pmod{4}$ if $p=2$, where $k\geq 0$, as follows.
\begin{enumerate}
\item Suppose $p > 2$. Write $k = 2p^{l-1}c(p-1)-1$ with $p\nmid c$, and set
\[
j_k^{(p)} = \resp(a^{1/2}\beta_{\rat_{p^l}}^c),\qquad \text{where }a^{1/2}\beta_{\rat_{p^l}}^c \in M_k^{C_{p^l}}.
\]
\item Suppose $p = 2$ and $k \equiv -1\pmod{8}$. Write $k = 2^lc-1$ with $l\geq 3$ and $2\nmid c$. Now let $j_k^{(2)}$ be any element satisfying
\[
2j_k^{(2)} = \resp(a^{1/2}\beta_{\rat_{2^l}}^c),\qquad \text{where }a^{1/2}\beta_{\rat_{2^l}}^c \in M_k^{C_{2^l}}.
\]
In other words, we only really specify a generator for $2\pi_k j_{(2)}$ when $k\equiv -1\pmod{8}$. 
\item Suppose $p = 2$ and $k\equiv -1 \pmod{4}$. Write $k = 8l-5$, and define
\[
j_k^{(2)} = \resp(a^{1/2}t_{3,2}\beta_{4l-2}),\qquad \text{where } a^{1/2}t_{3,2}\beta_{4l-2} \in M_k^{C_8}.
\]
\end{enumerate}
That these are in fact generators follows from \cref{lem:rescritical} and \cref{lem:resodd}.
\tqed
\end{defn}

\begin{question}
Is $j_k^{(p)} \in \pi_k j_{(p)}$ the ($p$-torsion part of) the $J$-image of a generator of $\pi_k O$?
\tqed
\end{question}

\begin{prop}\label{prop:imr}
Fix $k\geq 1$. Then the nonzero values of $\resp \colon M_k^{C_{p^n}}\rightarrow \tilde{\pi}_k J$ on our chosen basis are given as follows.
\begin{enumerate}
\item Suppose $p > 2$. Then $\tilde{\pi}_k J = 0$ unless $k = 2p^{l-1}c(p-1)-1$ for some $l\geq 1$ and $p\nmid c$. Now
\[
\resp(a^{1/2}t_{n,i}\beta_{p^{l-1}c(p-1)}) = p^{l-i}j_{2p^{l-1}c(p-1)-1}^{(p)}
\]
for $1\leq i \leq \min(n,l)$.
\item Suppose $p = 2$ and $8k = 2^lc$ with $2\nmid c$.
\begin{enumerate}
\item If $n = 1$, then
\begin{align*}
\resp(a_\sigma \beta_{4k}) &= 2^l j_{8k-1}^{(2)},\\
\resp(a_\sigma^3\beta_{4k-1}) &= 4j_{8k-5}^{(2)},\\
\resp(a_\sigma^2 \beta_{4k-2}) &= P^{k-1}\eta^2,\\
\resp(a_\sigma \beta_{4k-3}) &= P^{k-1}\eta.
\end{align*}
\item If $n = 2$, then
\begin{align*}
\resp(a^{1/2}t_{2,1}\beta_{4k}) &= 2^lj_{8k-1}^{(2)},\\
 \resp(a^{1/2}\beta_{4k}) &= 2^{l-1}j_{8k-1}^{(2)},\\
\resp(a^{1/2}t_{2,1}\beta_{4k-2}) &= -2j_{8k-5}^{(2)},\\
 \resp(2a^{1/2}\beta_{4k-2}) &= 2j_{8k-5}^{(2)},\\
\resp(a\beta_{4k-2}) &= P^{k-1}\eta^2,\\
\resp(a^{1/2}t_{2,1}\beta_{4k-3}) &= P^{k-1}\eta.
\end{align*}
\item If $n \geq 3$, then
\begin{align*}
\resp(a^{1/2}t_{n,i}\beta_{4k}) &= 2^{1+l-i}j_{8k-1}^{(2)},\\
\resp(a^{1/2}t_{n,1}\beta_{4k-2}) &= -2j_{8k-5}^{(2)},\\
\resp(a^{1/2}t_{n,2}\beta_{4k-2}) &= j_{8k-5}^{(2)},\\
\resp(a^{1/2}t_{n,1}\beta_{4k-3}) &= P^{k-1}\eta.
\end{align*}
the first line holding for $1\leq i \leq \min(n,l)$.
\end{enumerate}
\end{enumerate}
\end{prop}
\begin{proof}
We just treat the case $p > 2$. The case $p = 2$ is handled the same way in degrees congruent to $-1\pmod{4}$, albeit with more cases. Outside these degrees, $\tilde{\pi}_k J_2$ has at most one nonzero element and $M_k^{C_{2^n}}$ has at most one basis element which is not divisible by $a^{1/2}$, and this determines the value of $\resp$.

First suppose $l\leq n$, and write
\[
M_k^{C_{p^n}} = \bbZ\{a^{1/2}t_{n,1}\beta_{(k+1)/2},\ldots,a^{1/2}t_{n,l}\beta_{(k+1)/2}\}\oplus N
\]
where $N$ consists of those basis elements plainly divisible by $a^{1/2}$. By \cref{prop:restr} (compare \cref{prop:roottr}), we then have
\[
\resp(a^{1/2}t_{n,i}\beta_{(k+1)/2}) = \resp(\tr_l^n(a^{1/2}t_{l,i}\beta_{(k+1)/2})) = \resp(a^{1/2}t_{l,i}\beta_{(k+1)/2}),
\]
so may reduce to the case $l = n$. Now the claim follows from the definition $j_k^{(p)} = \resp(a^{1/2}\beta_{(k+1)/2})$ together with \cref{lem:rescritical}.

Next suppose $l > n$, and write
\[
M_k^{C_{p^n}} = \bbZ\{a^{1/2}t_{n,1}\beta_{(k+1)/2},\ldots,a^{1/2}t_{n,n}\beta_{(k+1)/2}\}
\]
By \cref{prop:restr} (compare \cref{prop:rootres}) we have
\[
\resp(a^{1/2}t_{n,n}\beta_{(k+1)/2}) = \resp(\res^l_n(a^{1/2}t_{l,n}\beta_{(k+1)/2})) = p^{l-n}\resp(a^{1/2}t_{l,l}\beta_{(k+1)/2}) = p^{l-n}j_k^{(p)}.
\]
The remaining values then follow from \cref{lem:rescritical}.
\end{proof}

Although not needed, we find it worth pointing out the following interpretation of \cref{prop:imr}.

\begin{cor}\label{cor:todabracket}
Suppose when $p=2$ that $8\mid 2^nt$. Then the image of $\beta_{\rat_{p^n}}^t \in \pi_{t\rat_{p^n}}KU_{C_{p^n}}$ in the cofiber
\[
\pi_{t\rat_{p^n}}(KU_{C_{p^n}} \otimes \Cof\left(\nabla\colon C_{p^n+}\rightarrow S^0)\right) = \pi_{0,2tp^{n-1}(p-1)-1}KU_{C_{p^n}}
\]
lifts to an element of the composition Toda bracket $\langle a^{1/2},\nabla,\hat{\alpha}\rangle \subset \pi_{0,2tp^{n-1}(p-1)-1}S_{C_{p^n}}$, where $\hat{\alpha} \colon S^{t\rat_{p^n}-1}\rightarrow C_{p^n+}$ is adjoint to an element $\alpha \in \pi_{2tp^{n-1}(p-1)-1}S$ of order $p^n$ in the image of $J$.
\end{cor}
\begin{proof}
Consider the diagram
\begin{center}\begin{tikzcd}
&SC_{p^n}\ar[d]&\Sigma C_{p^n+}\ar[d,"\nabla"]\\
S^{t\rat_{p^n}}\ar[r,"\hat{\alpha}"]\ar[ur,dashed,"y"]&\Sigma C_{p^n+}\ar[r,"\nabla"]\ar[ur,dashed]&S^1\ar[r,"a^{1/2}"]&\Sigma SC_{p^n}
\end{tikzcd}.\end{center}
Because $a^{1/2}$ is the cofiber of $\nabla$, there is a preferred nullhomotopy of $a^{1/2}\circ\nabla$, and the subset of $\langle a^{1/2},\nabla,\hat{\alpha}\rangle$ corresponding to this is exactly the set of lifts $y$. In other words, it is exactly the set of $y \in \pi_{0,2tp^{n-1}(p-1)-1}S_{C_{p^n}}$ satisfying $\resp(y) = \alpha$.

The image of $\beta_{\rat_{p^n}}^t$ in $\pi_{0,2tp^{n-1}(p-1)-1}KU_{C_{p^n}}$ is $a^{1/2}\beta_{\rat_{p^n}}^t$. The assumption that $8\mid p^nt$ when $p=2$ ensures that this lives in $M_{2p^{n-1}t(p-1)-1}^{C_{p^n}}$. \cref{prop:imr} then implies that $\resp(a^{1/2}\beta_{\rat_{p^n}}^t) \in \pi_{2p^{n-1}t(p-1)-1}J_p$ is an element of order $p^n$. Now the same argument from \cref{lem:jcomputeroot} ensures that $a^{1/2}\beta_{\rat_{p^n}}^t$ lifts to a class $y \in \pi_{0,2tp^{n-1}(p-1)-1}S_{C_{p^n}}$ for which $\resp(y)$ is an element of order $p^n$ in the image of $J$.
\end{proof}

We can now give the proof of \cref{thm:mainroot2}.

\begin{theorem}\label{thm:micomputation}
Suppose $k\geq 1$ and $X\in \Gamma_k(C_{p^{n-1}})\setminus \Gamma_{k+1}(C_{p^{n-1}})$, and write $\phi(X) = c_1 f_{1,k}^{p,n}+\cdots+c_n f_{n,k}^{p,n}$. Then
\begin{enumerate}
\item If $p>2$, then $k=2p^{l-1}c(p-1)-1$ for some $l\geq 1$ and $p\nmid c$. In this case $j_{2p^{l-1}c(p-1)-1}^{(p)}$ has order $p^l$, and if $t = \min(n,l)$ then
\[
M_{C_{p^n}}(X)\cap \pi_{2p^{l-1}c(p-1)-1} j_{(p)} = \{(p^{t-1}c_1+p^{t-2}c_2+\cdots+c_t) \cdot p^{l-t}j_{2p^{l-1}c(p-1)-1}^{(p)}\}.
\]
\item If $p=2$, then $k$ is congruent to one of $1,2,3,7\pmod{8}$. In this case,
\begin{enumerate}
\item If $k = 1$, then $M_{C_{2^n}}(X) = \{\eta\}$.
\item If $k = 8l+1$ with $l\geq 1$, then
\[
M_{C_{2^n}}(X)\cap \pi_{8l+1} j_{(2)} = P^l \eta + \bbZ/(2)\{P^{l-1}\eta\epsilon\}.
\]
\item If $k = 8l+2$, then 
\[
M_{C_{2^n}}(X)\cap \pi_{8l+2} j_{(2)} = \{P^l \eta^2\}.
\]
Moreover, this can only happen for $n\leq 2$.
\item If $k = 8l+3$, then $j_{8l+3}^{(2)}$ has order $8$ and
\begin{enumerate}
\item If $n=1$, then 
\[
M_{C_2}(X) \cap \pi_{8l+3} j_{(2)} = \{4j_{8l+3}^{(2)}\}.
\]
\item If $n=2$, then 
\[
M_{C_4}(X) \cap \pi_{8l+3} j_{(2)} = \{-2(c_1-c_2)j_{8l+3}^{(2)}\}.
\]
\item If $n\geq 3$, then 
\[
M_{C_{2^n}}(X)\cap \pi_{8l+3} j_{(2)} = \{-(2c_1-c_2)j_{8l+3}^{(2)}\}.
\]
\end{enumerate}
\item If $k=2^{l}c-1$ with $2\nmid c$ and $l\geq 3$, then $j_{2^{l}c-1}^{(2)}$ has order $2^{l+1}$, and if $t = \min(n,l)$ then
\[
M_{C_{2^n}}(X) \cap \pi_{2^{l}c-1} j_{(2)} = \{(2^{t-1}c_1+2^{t-2}c_2+\cdots+c_t) 2^{l+1-t} j_{2^{l}c-1}^{(2)}\}.
\]
\end{enumerate}
\end{enumerate}
\end{theorem}
\begin{proof}
By \cref{lem:phifp2}, to say that $\phi(X) = c_1 f_{1,k}^{p,n}+\cdots+c_n f_{n,k}^{p,n}$ is to say that $X \in A(C_{p^{n-1}})$ lifts to $\widetilde{X} = a^{k/2}(c_1y_1+\cdots+c_ny_n)$ in $A(C_{p^n})$, where $y_i$ is the $i$th basis element of $M_k^{C_{p^n}}$. Now the theorem follows from \cref{lem:jcomputeroot} and \cref{prop:imr}.
\end{proof}

\begin{ex}\label{ex:miexamples}
Up to modifying $\nu$ and $\sigma$ by an odd integer, we have
\begin{gather*}
M_{C_2}(2) \ni \eta,\quad M_{C_4}(2[C_2]) \ni \eta,\quad M_{C_4}(4) \ni \eta,\quad M_{C_8}(2[C_4]) \ni \eta,\quad M_{C_8}(8)\ni \eta, \\
M_{C_2}(4) \ni \eta^2,\qquad M_{C_4}(2+[C_2]) \ni \eta^2,\\
M_{C_4}(4+2[C_2]) \ni 2\nu,\qquad M_{C_4}(4[C_2]) \ni -2\nu,\\
M_{C_8}(4[C_4/C_2]+2[C_4]) \ni 2\nu,\qquad M_{C_8}(4+6[C_4/C_2]) \ni 6\nu,\\
M_{C_8}(4[C_4]) \ni -\nu,\quad M_{C_8}(2[C_4/C_2]+[C_4]) \ni \nu,\quad M_{C_8}(2+[C_4/C_2]+3[C_4]) \ni 2\sigma.\tag*{$\triangleleft$}
\end{gather*}
\end{ex}

\begin{ex}\label{ex:miexamplesodd}
For any prime $p>2$ we have 
\[
j_{2(p-1)-1}^{(p)} \in M_{C_{p^n}}(p[C_{p^{n-1}}]),\qquad j_{2(p-1)-1}^{(p)} \in M_{C_{p^n}}(p^n).
\]
The first is implicit in \cref{prop:imr} and is an instance of \cref{prop:roottr}; it is realized by
\[
\Phi(a^{p-1/2}\cdot a^{1/2}t_{n,1}\beta_{p-1}) = p[C_{p^{n-1}}],\qquad\resp(a^{1/2}t_{n,1}\beta_{p-1}) = j_{2(p-1)-1}^{(p)}.
\]
Hence to verify the second it suffices to show that $p^n - p[C_{p^{n-1}}]$ lives in higher filtration, which follows from the identity $p^n - p[C_{p^{n-1}}] = \Phi(pz_{n,2}+p^2z_{n,3}+\cdots+p^{n-1}z_{n,n})$ as $z_{n,2},\ldots,z_{n,n} \in A(C_{p^n})$ are already in the image of $a^{p-1}\colon M_{2(p-1)}^{C_{p^n}}\rightarrow A(C_{p^n})$.
\tqed
\end{ex}

\cref{thm:micomputation} has the following corollary.

\begin{cor}\label{cor:imagemahowald}
The full intersection $M_{C_{p^n}}(A(C_{p^{n-1}}))\cap \pi_\ast j_{(p)}$ can be described as follows.
\begin{enumerate}
\item If $p > 2$, then it consists of exactly the elements of order at most $p^n$ in $\pi_\ast j_{(p)}$.
\item If $p = 2$, then
\begin{enumerate}
\item If $n \leq 2$, then it consists of exactly the elements of order at most $2^n$ in $\pi_k j_{(2)}$ for $k\equiv 2,3,7\pmod{8}$, as well as both classes in $\pi_k j_{(2)}$ for $k\equiv 1\pmod{8}$ that are not in the kernel of the map $\pi_k j_{(2)}\rightarrow \pi_k KO$.
\item If $n\geq 3$, then it consists of exactly the elements of order at most $2^n$ in $\pi_k j_{(2)}$ for $k\equiv 3\pmod{8}$ and $2\pi_k j_{(2)}$ for $k\equiv 7\pmod{8}$, as well as both classes in $\pi_k j_{(2)}$ for $k\equiv 1\pmod{8}$ that are not in the kernel of the map $\pi_k j_{(2)}\rightarrow \pi_k KO$.
\end{enumerate}
\end{enumerate}
In particular, if $p>2$ or $n\geq 3$ then, modulo indeterminacy, $M_{C_{p^n}}(A(C_{p^{n-1}}))$ consists of exactly those elements of order at most $p^n$ in the image of the \textit{complex} $J$-homomorphism.
\qed
\end{cor}

\begin{ex}
The class $\eta^2$ is in the image of the $C_{2^n}$-Mahowald invariant only for $n\in\{1,2\}$, realized in either case by $\resp(a\beta_2) = \eta^2$, with compatibility being an instance of \cref{prop:restr}. The fact that this does not work for $n\geq 3$ then corresponds to the fact that $a \beta_2 = (1-L^{-1})\beta_L \in \pi_{0,2}KO_{C_{2^n}}\subset \pi_{0,2}KU_{C_{2^n}}$ is only fixed by $\psi^\ell$ for $n \in\{1,2\}$.
\tqed
\end{ex}

\begingroup
\raggedright
\bibliography{refs}

\newcommand{\etalchar}[1]{$^{#1}$}
\begin{thebibliography}{BHN{\etalchar{+}}19}

\bibitem[AMS98]{andomoravasadofsky1998completions}
Matthew Ando, Jack Morava, and Hal Sadofsky.
\newblock Completions of {$\mathbf Z/(p)$}-{T}ate cohomology of periodic
  spectra.
\newblock {\em Geom. Topol.}, 2:145--174, 1998.

\bibitem[Ang21]{angeltveit2021picard}
Vigleik Angeltveit.
\newblock The picard group of the category of {$C_n$}-equivariant stable
  homotopy theory, 2021.

\bibitem[AT69]{atiyahtall1969group}
M.~F. Atiyah and D.~O. Tall.
\newblock Group representations, {$\lambda$}-rings and the {$J$}-homomorphism.
\newblock {\em Topology}, 8:253--297, 1969.

\bibitem[Bal24]{balderrama2022total}
William Balderrama.
\newblock Total power operations in spectral sequences.
\newblock {\em Trans. Am. Math. Soc.}, 377(7):4779--4823, 2024.

\bibitem[Bal26]{balderrama2021borel}
William Balderrama.
\newblock The {$C_2$}-equivariant {$K(1)$}-local sphere.
\newblock {\em Math. Z.}, 312(2):Paper No. 52, 29, 2026.

\bibitem[{Bal}ar]{balderrama2023equivalences}
William {Balderrama}.
\newblock {Equivalences of the form $\Sigma^V X \simeq \Sigma^W X$ in
  equivariant stable homotopy theory}.
\newblock {\em Doc. Math.}, to appear.

\bibitem[BC25]{behrenscarlisle2024periodic}
Mark Behrens and Jack Carlisle.
\newblock Periodic phenomena in equivariant stable homotopy theory, 2025.

\bibitem[Beh07]{behrens2007some}
Mark Behrens.
\newblock Some root invariants at the prime 2.
\newblock In {\em Proceedings of the Nishida Fest conferences on homotopy
  theory, Kinosaki, Japan, July 28--August 1 2003 and August 4--8, 2003.
  Dedicated to G\^or\^o Nishida on the occasion of his 60th birthday}, pages
  1--40. Coventry: Geometry \& Topology Publications, 2007.

\bibitem[BG95]{brunergreenlees1995bredon}
Robert Bruner and John Greenlees.
\newblock The {Bredon}-{L{\"o}ffler} conjecture.
\newblock {\em Exp. Math.}, 4(4):289--297, 1995.

\bibitem[BGH20]{barthelgreenleeshausmann2020balmer}
Tobias Barthel, J.~P.~C. Greenlees, and Markus Hausmann.
\newblock On the {B}almer spectrum for compact {L}ie groups.
\newblock {\em Compos. Math.}, 156(1):39--76, 2020.

\bibitem[BGL22]{bhattacharyaguillouli2022rmotivic}
Prasit Bhattacharya, Bertrand Guillou, and Ang Li.
\newblock An {{\(R\)}}-motivic {{\(v_1\)}}-self-map of periodicity {{\(1\)}}.
\newblock {\em Homology Homotopy Appl.}, 24(1):299--324, 2022.

\bibitem[BGS22]{bonventreguilloustapleton2022kug}
Peter~J. {Bonventre}, Bertrand~J. {Guillou}, and Nathaniel~J. {Stapleton}.
\newblock {On the $KU_G$-local equivariant sphere}.
\newblock {\em arXiv e-prints}, page arXiv:2204.03797, April 2022.

\bibitem[BHN{\etalchar{+}}19]{barthelhausmannnaumannnikolausnoelstapleton2019balmer}
Tobias Barthel, Markus Hausmann, Niko Naumann, Thomas Nikolaus, Justin Noel,
  and Nathaniel Stapleton.
\newblock The {B}almer spectrum of the equivariant homotopy category of a
  finite abelian group.
\newblock {\em Invent. Math.}, 216(1):215--240, 2019.

\bibitem[BQ23]{botvinnikquigley2023symmetries}
Boris {Botvinnik} and J.~D. {Quigley}.
\newblock {Symmetries of exotic spheres via complex and quaternionic Mahowald
  invariants}.
\newblock {\em arXiv e-prints}, page arXiv:2309.04275, September 2023.

\bibitem[Bre67]{bredon1967equivariant}
Glen~E. Bredon.
\newblock Equivariant stable stems.
\newblock {\em Bull. Amer. Math. Soc.}, 73:269--273, 1967.

\bibitem[Bre68]{bredon1968equivariant}
G.~E. Bredon.
\newblock Equivariant homotopy.
\newblock Proc. {Conf}. {Transform}. {Groups}, {New} {Orleans} 1967, 281-292
  (1968)., 1968.

\bibitem[BS16]{balmersanders2016spectrum}
Paul Balmer and Beren Sanders.
\newblock The spectrum of the equivariant stable homotopy category of a finite
  group.
\newblock {\em Inventiones mathematicae}, 208(1):283--326, sep 2016.

\bibitem[BXZ24]{belmontxuzhang2024reduced}
Eva Belmont, Zhouli Xu, and Shangjie Zhang.
\newblock The reduced ring of the {$RO(C_2)$}-graded {$C_2$}-equivariant stable
  stems.
\newblock {\em Proc. Amer. Math. Soc. Ser. B}, 11:1--14, 2024.

\bibitem[Car84]{carlsson1983equivariant}
Gunnar Carlsson.
\newblock Equivariant stable homotopy and {Segal}'s {Burnside} ring conjecture.
\newblock {\em Ann. Math. (2)}, 120:189--224, 1984.

\bibitem[CE56]{cartaneilenberg1956homological}
Henri Cartan and Samuel Eilenberg.
\newblock {\em Homological algebra}.
\newblock Princeton University Press, Princeton, NJ, 1956.

\bibitem[Cra89]{crabb1989periodicity}
M.~C. Crabb.
\newblock Periodicity in {$\bf Z/4$}-equivariant stable homotopy theory.
\newblock In {\em Algebraic topology ({E}vanston, {IL}, 1988)}, volume~96 of
  {\em Contemp. Math.}, pages 109--124. Amer. Math. Soc., Providence, RI, 1989.

\bibitem[DM84]{davismahowald1984spectrum}
Donald~M. Davis and Mark Mahowald.
\newblock The spectrum {$(P\wedge b{\rm o})\sb{-\infty }$}.
\newblock {\em Math. Proc. Cambridge Philos. Soc.}, 96(1):85--93, 1984.

\bibitem[FLM01]{fausklewismay2001picard}
H.~Fausk, L.~G. Lewis, Jr., and J.~P. May.
\newblock The {P}icard group of equivariant stable homotopy theory.
\newblock {\em Adv. Math.}, 163(1):17--33, 2001.

\bibitem[GI20]{guillouisaksen2020bredon}
Bertrand~J. Guillou and Daniel~C. Isaksen.
\newblock The {B}redon-{L}andweber region in {$C_2$}-equivariant stable
  homotopy groups.
\newblock {\em Doc. Math.}, 25:1865--1880, 2020.

\bibitem[GM95]{greenleesmay1995generalized}
J.~P.~C. Greenlees and J.~P. May.
\newblock Generalized {T}ate cohomology.
\newblock {\em Mem. Amer. Math. Soc.}, 113(543):viii+178, 1995.

\bibitem[GQ23]{greenleesquigley2023ranks}
J.~P.~C. Greenlees and J.~D. Quigley.
\newblock Ranks of {$RO(G)$}-graded stable homotopy groups of spheres for
  finite groups {$G$}.
\newblock {\em Proc. Amer. Math. Soc. Ser. B}, 10:101--113, 2023.

\bibitem[GS96]{greenleessadofsky1996tate}
J.~P.~C. Greenlees and Hal Sadofsky.
\newblock The {T}ate spectrum of {$v_n$}-periodic complex oriented theories.
\newblock {\em Math. Z.}, 222(3):391--405, 1996.

\bibitem[Hau77]{hauschild1977zerspaltung}
H.~Hauschild.
\newblock Zerspaltung {\"a}quivarianter {Homotopiemengen}.
\newblock {\em Math. Ann.}, 230:279--292, 1977.

\bibitem[HLSX22]{hopkinslinshixu2022intersection}
Michael~J. Hopkins, Jianfeng Lin, XiaoLin~Danny Shi, and Zhouli Xu.
\newblock Intersection forms of spin 4-manifolds and the {${\rm
  Pin}(2)$}-equivariant {M}ahowald invariant.
\newblock {\em Comm. Amer. Math. Soc.}, 2:22--132, 2022.

\bibitem[HS96]{hoveysadofsky1996tate}
Mark Hovey and Hal Sadofsky.
\newblock Tate cohomology lowers chromatic {B}ousfield classes.
\newblock {\em Proc. Amer. Math. Soc.}, 124(11):3579--3585, 1996.

\bibitem[HSW23]{hahnsengerwilson2023odd}
Jeremy Hahn, Andrew Senger, and Dylan Wilson.
\newblock Odd primary analogs of real orientations.
\newblock {\em Geom. Topol.}, 27(1):87--129, 2023.

\bibitem[Iri82a]{iriye1982equivariant}
Kouyemon Iriye.
\newblock Equivariant stable homotopy groups of spheres with involutions. {II}.
\newblock {\em Osaka J. Math.}, 19(4):733--743, 1982.

\bibitem[Iri82b]{iriye1982nilpotency}
Kouyemon Iriye.
\newblock The nilpotency of elements of the equivariant stable homotopy groups
  of spheres.
\newblock {\em J. Math. Kyoto Univ.}, 22(2):257--259, 1982.

\bibitem[Iri89]{iriye1989images}
Kouyemon Iriye.
\newblock On images of the fixed-point homomorphism in the {${\bf
  Z}/p$}-equivariant stable homotopy groups.
\newblock {\em J. Math. Kyoto Univ.}, 29(1):159--163, 1989.

\bibitem[KL24]{kuhnlloyd2022chromatic}
Nicholas~J. Kuhn and Christopher J.~R. Lloyd.
\newblock Chromatic fixed point theory and the {Balmer} spectrum for
  extraspecial 2-groups.
\newblock {\em Am. J. Math.}, 146(3):769--812, 2024.

\bibitem[Lai79]{laitinen1979burnside}
Erkki Laitinen.
\newblock On the {B}urnside ring and stable cohomotopy of a finite group.
\newblock {\em Math. Scand.}, 44(1):37--72, 1979.

\bibitem[Lan69]{landweber1969equivariant}
P.~S. Landweber.
\newblock On equivariant maps between spheres with involutions.
\newblock {\em Ann. Math. (2)}, 89:125--137, 1969.

\bibitem[LLQ22]{lilormanquigley2022tate}
Guchuan Li, Vitaly Lorman, and J.~D. Quigley.
\newblock Tate blueshift and vanishing for real oriented cohomology theories.
\newblock {\em Adv. Math.}, 411 A:51, 2022.
\newblock Id/No 108780.

\bibitem[Mah67]{mahowald1967metastable}
Mark Mahowald.
\newblock {\em The metastable homotopy of { $S^n$ }}.
\newblock Memoirs of the American Mathematical Society, No. 72. American
  Mathematical Society, Providence, R.I., 1967.

\bibitem[May77]{may1977einfty}
J.~Peter May.
\newblock {\em {{\(E_\infty\)}} ring spaces and {{\(E_\infty\)}} ring spectra.
  {With} contributions by {Frank} {Quinn}, {Nigel} {Ray}, and {Jorgen}
  {Tornehave}}, volume 577 of {\em Lect. Notes Math.}
\newblock Springer, Cham, 1977.

\bibitem[MNN17]{mathewnaumannnoel2017nilpotence}
Akhil Mathew, Niko Naumann, and Justin Noel.
\newblock Nilpotence and descent in equivariant stable homotopy theory.
\newblock {\em Advances in Mathematics}, 305:994--1084, 2017.

\bibitem[MR93]{mahowaldravenel1993root}
Mark~E. Mahowald and Douglas~C. Ravenel.
\newblock The root invariant in homotopy theory.
\newblock {\em Topology}, 32(4):865--898, 1993.

\bibitem[MS88]{mahowaldshick1988root}
Mark Mahowald and Paul Shick.
\newblock Root invariants and periodicity in stable homotopy theory.
\newblock {\em Bull. London Math. Soc.}, 20(3):262--266, 1988.

\bibitem[Qui71]{quillen1971adams}
D.~Quillen.
\newblock The {Adams} conjecture.
\newblock {\em Topology}, 10:67--80, 1971.

\bibitem[Qui22]{quigley2019tmf}
J.~D. Quigley.
\newblock {{\(\mathrm{tmf}\)}}-based {Mahowald} invariants.
\newblock {\em Algebr. Geom. Topol.}, 22(4):1789--1839, 2022.

\bibitem[Rit72]{ritter1972induktionssatz}
J\"urgen Ritter.
\newblock Ein {I}nduktionssatz f\"ur rationale {C}haraktere von nilpotenten
  {G}ruppen.
\newblock {\em J. Reine Angew. Math.}, 254:133--151, 1972.

\bibitem[Sad92]{sadofsky1992root}
Hal Sadofsky.
\newblock The root invariant and {$v_1$}-periodic families.
\newblock {\em Topology}, 31(1):65--111, 1992.

\bibitem[Sch18]{schwede2018global}
Stefan Schwede.
\newblock {\em Global homotopy theory}, volume~34 of {\em New Mathematical
  Monographs}.
\newblock Cambridge University Press, Cambridge, 2018.

\bibitem[Seg72]{segal1972permutation}
Graeme Segal.
\newblock Permutation representations of finite {$p$}-groups.
\newblock {\em Quart. J. Math. Oxford Ser. (2)}, 23:375--381, 1972.

\bibitem[Ser77]{serre1977linear}
Jean-Pierre Serre.
\newblock {\em Linear representations of finite groups}, volume Vol. 42 of {\em
  Graduate Texts in Mathematics}.
\newblock Springer-Verlag, New York-Heidelberg, french edition, 1977.

\bibitem[Shi87]{shick1987root}
Paul Shick.
\newblock On root invariants of periodic classes in {$Ext_A(Z/2,Z/2)$}.
\newblock {\em Trans. Amer. Math. Soc.}, 301(1):227--237, 1987.

\bibitem[tD79]{dieck1979transformation}
Tammo tom Dieck.
\newblock {\em Transformation Groups and Representation Theory}.
\newblock Lecture Notes in Mathematics. Springer Berlin Heidelberg, 1979.

\bibitem[tDP78]{dieckpetrie1978geometric}
Tammo tom Dieck and Ted Petrie.
\newblock Geometric modules over the {Burnside} ring.
\newblock {\em Invent. Math.}, 47:273--287, 1978.

\bibitem[Tor82]{tornehave1982equivariant}
J.~Tornehave.
\newblock Equivariant maps of spheres with conjugate orthogonal actions.
\newblock Current trends in algebraic topology, {Semin}. {London}/{Ont}. 1981,
  {CMS} {Conf}. {Proc}. 2, 2, 275-301 (1982)., 1982.

\bibitem[Wol67]{wolf1967spaces}
J.~A. Wolf.
\newblock Spaces of constant curvature.
\newblock New {York}-{St}. {Louis}-{San}
  {Francisco}-{Toronto}-{London}-{Sydney}: {McGraw}-{Hill} {Book} {Comp}. {XV},
  408 p. (1967)., 1967.

\end{thebibliography}
\bibliographystyle{alpha}
\endgroup

\end{document}